\documentclass[a4paper]{amsart}
\usepackage{amsmath,amssymb}

\DeclareSymbolFont{bbold}{U}{bbold}{m}{n}
\DeclareSymbolFontAlphabet{\mathbbold}{bbold}
\newcommand{\bbone}{\mathbbold{1}}

\DeclareSymbolFontAlphabet{\amsbb}{AMSb}
\makeatother
\usepackage{mbboard}
\renewcommand{\mathbb}[1]{\amsbb{#1}}

\usepackage[shortlabels]{enumitem}
  \setitemize[1]{leftmargin=2em}
  \setenumerate[1]{leftmargin=2em}
\usepackage{tikz-cd}
\usepackage[unicode,colorlinks=true,linktocpage=true,citecolor=blue,hyperindex,breaklinks]{hyperref}
\usepackage{graphicx,color}

\newtheorem{theorem}{Theorem}
\numberwithin{theorem}{subsection}
\newtheorem{thm}[theorem]{Theorem}
\newtheorem{proposition}[theorem]{Proposition}
\newtheorem{propn}[theorem]{Proposition}

\newtheorem{corollary}[theorem]{Corollary}
\newtheorem{cor}[theorem]{Corollary}
\newtheorem{lemma}[theorem]{Lemma}
\theoremstyle{definition}
\newtheorem{definition}[theorem]{Definition}
\newtheorem{defn}[theorem]{Definition}

\newtheorem{examples}[theorem]{Examples}
\newtheorem{notation}[theorem]{Notation}
\newtheorem{remark}[theorem]{Remark}
\newtheorem{warning}[theorem]{Warning}

\providecommand{\op}{\mathrm{op}}
\providecommand{\xel}{\mathrm{el}}

\providecommand{\xev}{\mathrm{ev}}

\providecommand{\xint}{\mathrm{int}}

\newcommand{\xFun}{\operatorname{Fun}}

\newcommand{\xInd}{\operatorname{Ind}}

\DeclareMathOperator{\colimP}{colim}
\newcommand{\colim}{\mathop{\colimP}}

\newcommand{\xCat}{\operatorname{Cat}}

\newcommand{\xCS}{\operatorname{CtsSeg}}

\newcommand{\PSh}{\operatorname{P}}
\newcommand{\xMap}{\operatorname{Map}}
\newcommand{\Map}{\operatorname{Map}}

\newcommand{\xAlg}{\operatorname{Alg}}

\newcommand{\xSet}{\operatorname{Set}}

\DeclareMathOperator{\xSeg}{Seg}

\newcommand{\xB}{\mathcal{B}}
\newcommand{\xE}{\mathcal{E}}

\newcommand{\xI}{\mathcal{I}}

\newcommand{\xxO}{\mathcal{O}}
\newcommand{\xxP}{\mathcal{P}}
\newcommand{\xxQ}{\mathcal{Q}}

\newcommand{\xcc}{\mathcal{C}}
\newcommand{\xdd}{\mathcal{D}}

\newcommand{\xS}{\mathcal{S}}

\newcommand{\xV}{\mathcal{V}}
\newcommand{\xX}{\mathcal{X}}

\newcommand{\xU}{\mathcal{U}}

\newcommand{\xfe}{\mathfrak{e}}
\newcommand{\xfc}{\mathfrak{c}}

\newcommand{\id}{\operatorname{id}}

\newcommand{\xF}{\mathbb{F}}

\newcommand{\icat}{$\infty$-category}
\newcommand{\icats}{$\infty$-categories}
\newcommand{\iopd}{$\infty$-operad}
\newcommand{\iopds}{$\infty$-operads}
\newcommand{\isoto}{\xrightarrow{\sim}}

\newcommand{\xto}[1]{\xrightarrow{#1}}

\newcommand{\from}{\leftarrow}
\newcommand{\xfrom}[1]{\xleftarrow{#1}}
\newcommand{\csquare}[8]{ %
\[ %
\begin{tikzpicture} %
\matrix (m) [matrix of math nodes,row sep=3em,column sep=2.5em,text height=1.5ex,text depth=0.25ex] %
{ #1 \pgfmatrixnextcell #2 \\ %
  #3 \pgfmatrixnextcell #4 \\ }; %
\path[->,font=\footnotesize] %
(m-1-1) edge node[auto] {$#5$} (m-1-2)%
(m-1-1) edge node[left] {$#6$} (m-2-1)%
(m-1-2) edge node[auto] {$#7$} (m-2-2)%
(m-2-1) edge node[below] {$#8$} (m-2-2);%
\end{tikzpicture}%
\]%
}

\newcommand{\nolabelcsquare}[4]{\csquare{#1}{#2}{#3}{#4}{}{}{}{}}

\makeatletter
\def\@tocline#1#2#3#4#5#6#7{\relax
  \ifnum #1>\c@tocdepth 
  \else
    \par \addpenalty\@secpenalty\addvspace{#2}%
    \begingroup \hyphenpenalty\@M
    \@ifempty{#4}{%
      \@tempdima\csname r@tocindent\number#1\endcsname\relax
    }{%
      \@tempdima#4\relax
    }%
    \parindent\z@ \leftskip#3\relax \advance\leftskip\@tempdima\relax
    \rightskip\@pnumwidth plus4em \parfillskip-\@pnumwidth
    #5\leavevmode\hskip-\@tempdima
      \ifcase #1
       \or \hskip -1em \or \hskip 1em \or \hskip 3em \else \hskip 5em \fi%
      #6\nobreak\relax
    \hfill\hbox to\@pnumwidth{\@tocpagenum{#7}}
      \par
    \nobreak
    \endgroup
  \fi}
\makeatother

\newcommand{\xCor}{V}
\newcommand{\name}[1]{\ensuremath{\text{\textup{#1}}}}

\newcommand{\CorDF}{\xCor_{\DF}}
\newcommand{\CorO}{\xCor_{\bbO}}
\newcommand{\Fin}{\name{Fin}}
\newcommand{\simp}{\bbDelta}
\newcommand{\DF}{\simp_{\mathbb{F}}}
\newcommand{\DFX}{\simp_{\mathbb{F},X}}
\newcommand{\DFXop}{\DFX^{\op}}
\newcommand{\DV}{\simp^{\mathcal{V}}}
\newcommand{\DVop}{\simp^{\mathcal{V},\op}}
\newcommand{\DFint}{\simp_{\mathbb{F},\name{int}}}
\newcommand{\DFel}{\simp_{\mathbb{F},\name{el}}}
\newcommand{\DFelI}{\simp_{\mathbb{F},\name{el}/I}}
\newcommand{\DFelIop}{(\simp_{\mathbb{F},\name{el}/I})^{\op}}
\newcommand{\DFi}{\DF^{1}}
\newcommand{\DFiop}{\DF^{1,\op}}
\newcommand{\DFiint}{\bbDelta^{1}_{\mathbb{F},\name{int}}}

\newcommand{\DFiV}{\DF^{1,\mathcal{V}}}
\newcommand{\DFiU}{\DF^{1,\mathcal{U}}}
\newcommand{\DFiVop}{\DF^{1,\mathcal{V},\op}}

\newcommand{\DFiVint}{\bbDelta^{1,\mathcal{V}}_{\mathbb{F},\name{int}}}

\newcommand{\DFV}{\DF^{\mathcal{V}}}
\newcommand{\DFU}{\DF^{\mathcal{U}}}

\newcommand{\DFXV}{\DFX^{\mathcal{V}}}
\newcommand{\DFXVop}{\DFX^{\mathcal{V},\op}}

\newcommand{\bari}{\overline{\imath}}
\newcommand{\barj}{\overline{\jmath}}
\newcommand{\bartau}{\overline{\tau}}
\newcommand{\tauel}{\tau_{\name{el}}}
\newcommand{\tauint}{\tau_{\name{int}}}
\newcommand{\bartauel}{\bartau_{\name{el}}}
\newcommand{\bartauint}{\bartau_{\name{int}}}
\newcommand{\Dext}{\partial_{\name{ext}}}

\newcommand{\PU}{\PSh(\mathcal{U})}

\newcommand{\DFPU}{\DF^{\PU}}

\newcommand{\DFVk}{\DF^{\mathcal{V}^{\kappa}}}
\newcommand{\DFVint}{\DFint^{\mathcal{V}}}

\newcommand{\DFVintop}{\DFint^{\mathcal{V},\op}}

\newcommand{\DFVel}{\DFel^{\mathcal{V}}}
\newcommand{\DFUel}{\DFel^{\mathcal{U}}}
\newcommand{\DFVelop}{\DFel^{\mathcal{V},\op}}

\newcommand{\DFVop}{\DF^{\mathcal{V},\op}}
\newcommand{\DFUop}{\DF^{\mathcal{U},\op}}

\newcommand{\bbO}{\bbOmega}
\newcommand{\bbOop}{\bbO^{\op}}
\newcommand{\bbOX}{\bbO_{X}}

\newcommand{\bbOint}{\bbO_{\name{int}}}
\newcommand{\bbOintT}{\bbO_{\name{int}/T}}
\newcommand{\bbOel}{\bbO_{\name{el}}}
\newcommand{\bbOelT}{\bbO_{\name{el}/T}}
\newcommand{\bbOelTop}{(\bbO_{\name{el}/T})^{\op}}
\newcommand{\bbOV}{\bbO^{\mathcal{V}}}
\newcommand{\bbOU}{\bbO^{\mathcal{U}}}
\newcommand{\bbOVop}{\bbO^{\mathcal{V},\op}}
\newcommand{\bbOVint}{\bbOint^{\mathcal{V}}}
\newcommand{\bbOVintop}{\bbOint^{\mathcal{V},\op}}
\newcommand{\bbOVel}{\bbOel^{\mathcal{V}}}
\newcommand{\bbOUel}{\bbOel^{\mathcal{U}}}
\newcommand{\bbOVelop}{\bbOel^{\mathcal{V},\op}}

\newcommand{\bbF}{\mathbb{F}}

\newcommand{\Set}{\name{Set}}
\newcommand{\Seg}{\name{Seg}}
\newcommand{\Ind}{\name{Ind}}
\newcommand{\Fun}{\name{Fun}}
\newcommand{\PSeg}{\PSh_{\Seg}}
\newcommand{\PkSeg}{\PSh_{\kappa\name{-Seg}}}
\newcommand{\PSSeg}{\PSh_{\mathbb{S}\name{-Seg}}}
\newcommand{\PCts}{\PSh_{\name{Cts}}}
\newcommand{\PCS}{\PSh_{\xCS}}
\newcommand{\PM}{\PSh_{\name{Mon}}}
\newcommand{\PCM}{\PSh_{\name{CtsMon}}}
\newcommand{\PCSS}{\PSh_{\name{Cts}\,\mathbb{S}\name{-Seg}}}
\newcommand{\PCoS}{\PSh_{\name{CS}}}
\newcommand{\PCCS}{\PSh_{\name{CCS}}}
\newcommand{\blank}{\text{\textendash}}
\newcommand{\PrL}{\name{Pr}^{\mathrm{L}}}
\newcommand{\Cat}{\name{Cat}}
\newcommand{\CatI}{\Cat_{\infty}}
\newcommand{\LCatI}{\widehat{\Cat}_{\infty}}
\newcommand{\IFF}{if and only if}

\newcommand{\PSU}{\PSh_{\mathbb{S}}(\mathcal{U})}

\newcommand{\Sub}{\name{Sub}}

\newcommand{\OOp}{\mathcal{O}p}

\newcommand{\sOp}{\mathbf{Op}}

\newcommand{\Opd}{\name{Opd}}
\newcommand{\OpdV}{\Opd^{\mathbf{V}}}
\newcommand{\OpdI}{\Opd_{\infty}}
\newcommand{\OpdIV}{\OpdI^{\mathcal{V}}}
\newcommand{\Alg}{\name{Alg}}
\newcommand{\AlgDFS}{\name{Alg}_{\DF^{\op}/\mathcal{S}}}
\newcommand{\AlgOS}{\name{Alg}_{\bbO^{\op}/\mathcal{S}}}
\newcommand{\AlgOSet}{\name{Alg}_{\bbO^{\op}/\Set}}
\newcommand{\AlgOpSet}{\name{Alg}_{\OOp/\Set}}

\newcommand{\angled}[1]{\langle #1 \rangle}

\newcommand{\ie}{i.e.\@}
\newcommand{\cf}{cf.\@}

\usepackage{eucal}

\usepackage[utf8]{inputenc}
\title{Enriched $\infty$-operads}
\author{Hongyi Chu}
\address{MPIM, Bonn, Germany}

\author{Rune Haugseng}
\address{NTNU, Trondheim, Norway}
\date{\today}
\begin{document}
\begin{abstract}
  In this paper we initiate the study of enriched \iopds{}. We
  introduce several models for these objects, including enriched
  versions of Barwick's Segal operads and the dendroidal Segal spaces
  of Cisinski and Moerdijk, and
  show these are equivalent. Our main results are a version of Rezk's
  completion theorem for enriched \iopds{}: localization at the fully
  faithful and essentially surjective morphisms is given by the full
  subcategory of \emph{complete} objects, and a rectification theorem:
  the homotopy theory of \iopds{} enriched in the \icat{} arising from
  a nice symmetric monoidal model category is equivalent to the
  homotopy theory of strictly enriched operads.
\end{abstract}
\maketitle
\tableofcontents

\section{Introduction}\label{sec introduction}
Operads are a convenient formalism for parametrizing many algebraic
structures of interest in mathematics. Roughly speaking, an
\emph{operad}\footnote{The term operad is often used only for the
  single-object version of this notion, with the many-object case we
  consider known as a \emph{coloured operad} or \emph{multicategory};
  we have chosen to use the shorter term operad for the more general
  notion, as seems to be increasingly common in the current
  literature.} is a structure similar to a category, but instead of
morphisms with a single source and target, an operad has
\emph{multimorphisms} with a \emph{list} of objects as a source;
moreover, there is a $\Sigma_{n}$-action on the multimorphisms with
$n$ inputs that permutes these. If $\mathbf{C}$ is a symmetric
monoidal category then one can view $\mathbf{C}$ as an operad where a
multimorphism $(x_{1},\ldots,x_{n}) \to y$ is given by a morphism
$x_{1} \otimes \cdots \otimes x_{n} \to y$. If $\mathbf{O}$ is an
operad we can then define an \emph{$\mathbf{O}$-algebra} in
$\mathbf{C}$ to be a functor of operads $\mathbf{O} \to
\mathbf{C}$.
Many interesting algebraic structures arise as algebras in this sense,
including associative algebras, commutative algebras, and enriched
categories with a fixed set of objects.

For many purposes however, it is necessary to generalize operads to
\emph{enriched} operads --- here we replace the \emph{set} of
multimorphisms with an object of some symmetric monoidal
category. Then we can, for example, describe Lie algebras or
Poisson 
algebras as algebras for operads enriched in abelian
groups.

In topology, we often encounter operads enriched in topological spaces
or simplicial sets, known as topological and simplicial
operads. Indeed, it was this setting that originally motivated the
introduction of operads back in the 1970s: $n$-fold loop spaces admit
natural multiplications where algebraic identities, such as
associativity, only hold up to coherent homotopy, and this structure
can be codified as the structure of an algebra over a topological
operad $E_{n}$, defined using spaces of ``little discs'' in
$\mathbb{R}^{n}$. These operads were introduced by
Boardman--Vogt~\cite{BoardmanVogt} and May~\cite{May}, who both proved
versions of the \emph{recognition principle} for $n$-fold loop spaces:
 $n$-fold loop spaces are precisely the spaces that admit the
structure of a \emph{grouplike}\footnote{Meaning the induced
associative multiplication on the set of connected components makes
this a group.} $E_{n}$-algebra.

For applications in algebraic topology we typically only care about
the weak homotopy types of the spaces of multimorphisms in a
simplicial or topological operad. We are therefore led to consider the
\emph{homotopy theory} of such operads. This can be done by imposing a
model structure, with the weak equivalences being a suitable notion of
maps that are ``fully faithful and essentially surjective up to
homotopy'' (often called Dwyer--Kan equivalences). Such a model
structure on simplicial operads has been constructed by
Cisinski--Moerdijk~\cite{CisinkiMoerdijk3} and by
Robertson~\cite{Robertson}.

Unfortunately, for many purposes this model structure is not as
well-behaved as one might have hoped. For example, Boardman and Vogt
constructed a tensor product of simplicial operads whose internal Hom
gives simplicial operads of algebras, but this is well-known not to be
homotopically well-behaved\footnote{In particular, it does not make
  simplicial operads a symmetric monoidal model category; see the discussion at
  \url{https://mathoverflow.net/questions/198205/boardman-vogt-tensor-product}.}
so that these simplicial operads of algebras are typically not
homotopically meaningful.

We can improve the situation by replacing simplicial operads by a
weakly equivalent, but more flexible, notion in the form of
\emph{\iopds{}}. Roughly speaking, an \iopd{} is analogous to a
simplicial operad, but composition of multimorphisms is not strictly
associative, but rather associative up to coherent homotopy. The
first, and by far the best developed, model for \iopds{} is that of
Lurie~\cite{ha}; other models include the dendroidal sets of
Moerdijk--Weiss~\cite{MoerdijkWeiss}, the dendroidal Segal spaces of
Cisinski--Moerdijk~\cite{CisinkiMoerdijk2}, and the complete Segal
operads of Barwick~\cite{bar}. (These are all known to be equivalent,
due to the results of
\cite{bar,HeutsHinichMoerdijk,ChuHaugsengHeuts}.) There is an analogue
of the Boardman--Vogt tensor product for \iopds{}, and on the \icat{}
of \iopds{} this is as well-behaved as one might wish.

However, there are other homotopical contexts where we want to
consider enriched operads. For example, in algebraic settings we often
encounter operads enriched in chain complexes (usually called
\emph{dg-operads}), where we only care about the specific chain
complexes up to quasi-isomorphism. Similarly, in stable homotopy
theory we might want to consider operads enriched in spectra.  In both
cases model structures with the DK-equivalences as weak equivalences
were constructed by Caviglia~\cite{Caviglia,Caviglia2},\footnote{See
  also Remark~\ref{rmk:fixedobmodstr} for a discussion of model
  structures with a fixed set of objects.} but in the case of
dg-operads only over a field of characteristic zero --- indeed, even
with a fixed set of objects there do not seem to exist model structures
on dg-operads in positive characteristic.

In this paper we extend the $\infty$-categorical approach to
homotopy-coherent algebraic structures to contexts such as these, by
laying the foundations for a theory of \emph{enriched \iopds{}}. This
allows for enrichment in any symmetric monoidal \icat{}, thereby
giving a well-behaved homotopy theory of \iopds{} enriched in, for
example, spectra, chain complexes (in \emph{any} characteristic), and
modules over a commutative ring spectrum, as well as in more exotic
contexts such as quasicoherent sheaves on a derived stack. Although we
obtain the same homotopy theory as that presented by model categories
of enriched operads (when these exist), our $\infty$-categorical
approach is better-behaved in several respects --- most notably, we
obtain the homotopically correct \icats{} of algebras for enriched
\iopds{} simply as the right adjoint to a natural tensoring with
\icats{}.

Although our concerns in the present paper are foundational, we expect
the theory of enriched \iopds{} to have a number of interesting
applications, particularly in the context of \emph{Koszul duality}.
Koszul duality for dg-operads was first introduced in
\cite{GinzburgKapranov} and was later studied in, for example,
\cite{GetzlerJones,FresseKoszul,FresseEn,LodayVallette,Vallette}.
Though it is currently best understood over a field of characteristic
zero (which is the context for the papers just cited), Koszul duality
also occurs in spectra \cite{ChingBar,ChingHarper} where it is closely
related to Goodwillie calculus \cite{ChingId}. For spectra, however,
it seems likely that a full understanding of Koszul duality requires
$\infty$-categorical methods, as coalgebraic structures in spectra are
difficult to work with using model categories. More generally, Koszul
duality should occur for stable symmetric monoidal \icats{} ---
indeed, in this setting Francis--Gaitsgory~\cite{FrancisGaitsgory}
have used the expected properties of enriched \iopds{} to obtain
Koszul duality equivalences under certain finiteness hypotheses
(including in the case of chiral algebras) and also conjectured how
this should generalize.

\subsection{Main Results}
Our first goal is to define \iopds{} enriched in a symmetric monoidal
\icat{} $\mathcal{V}$ and set up their ``algebraic'' homotopy theory
(i.e.\ without inverting the fully faithful and essentially surjective
maps). In \S\ref{sec segal presheaves} we do this by considering an
enriched analogue of Barwick's approach to \iopds{}: we define, given
a presentably symmetric monoidal \icat{} $\mathcal{V}$, a notion of
\emph{continous Segal presheaves} on an \icat{} $\DFV$. We then show
that the \icat{} $\PCS(\DFV)$ of these objects has several pleasant
properties:
\begin{itemize}
\item it is presentable,
\item it is tensored and cotensored over Segal spaces --- giving Segal
  spaces of algebras by adjunction,
\item if $\mathcal{V}$ is the \icat{} $\mathcal{S}$ of spaces, with
  the Cartesian product as symmetric monoidal structure, then it is
  equivalent to Barwick's \icat{} of Segal operads.
\item it is functorial for lax symmetric monoidal functors.
\end{itemize}
To show the last point, we prove that $\PCS(\DFV)$ is equivalent to an
alternative model, using algebras in $\mathcal{V}$ for \icats{}
$\DFX^{\op}$, for which this functoriality is obvious. (This model
also makes sense without assuming presentability, but it is not our
main focus as our other results are more easily established using
Segal presheaves.)

Just as in the case of Segal spaces, to obtain the ``correct'' \icat{}
of enriched \iopds{} we need to invert the fully faithful and
essentially surjective maps. We introduce these in \S\ref{sec ffes}
and then prove our first main result, an analogue of Rezk's completion
theorem for Segal spaces in this context:
\begin{thm}
  The \icat{} $\OpdIV$ obtained as the localization of $\PCS(\DFV)$ at
  the fully faithful and essentially surjective morphisms is given by
  the full subcategory of complete objects, i.e.\ those whose
  underlying Segal space is complete in the sense of Rezk.
\end{thm}
The proof follows the same outline as the completion theorem for
enriched \icats{} from \cite{enriched}, which in turn is a variant of
Rezk's original proof \cite{Rezk}; the main new ingredient needed is
the tensoring of $\PCS(\DFV)$ with complete Segal spaces, constructed
in \S\ref{sec tensor product}. This moreover gives natural \icats{} of
algebras for enriched \iopds{} by adjunction.

In \S\ref{sec dendr} we turn to an enriched analogue of the dendroidal
Segal spaces of Cisinski and Moerdijk: We again consider a notion of
continuous Segal presheaves, now on an \icat{} $\bbOV$ enhancing the
dendroidal category $\bbO$. Our main result here extends the
comparison result of \cite{ChuHaugsengHeuts} to the enriched setting:
\begin{thm}
  There is an equivalence of \icats{} $\PCS(\DFV) \simeq \PCS(\bbOV)$.
\end{thm}
Once the definitions are set up, the proof proceeds essentially as in
\cite{ChuHaugsengHeuts}.

Finally, in \S\ref{sec alg} we prove our third main result, which
relates our enriched \iopds{} to the existing homotopy theories of
operads strictly enriched in model categories. Specifically, we prove
the following rectification theorem:
\begin{thm}
  Suppose $\mathbf{V}$ is a nice symmetric monoidal model category,
  and let $\mathbf{V}[W^{-1}]$ denote the symmetric monoidal \icat{}
  obtained by inverting the weak equivalences. Then the \icat{}
  $\OpdI^{\mathbf{V}[W^{-1}]}$ of \iopds{} enriched in
  $\mathbf{V}[W^{-1}]$ is equivalent to the \icat{}
  $\Opd^{\mathbf{V}}[\name{DK}^{-1}]$, obtained by inverting the
  Dwyer--Kan equivalences between strictly enriched
  $\mathbf{V}$-operads.
\end{thm}
This result applies, for example, with $\mathbf{V}$ being simplicial
sets (where we recover a result of
Cisinski--Moerdijk~\cite{CisinkiMoerdijk3}), symmetric spectra, or
chain complexes over a field of characteristic $0$. To prove this we
first show that $\mathcal{V}$-enriched \iopds{} with a fixed set of
objects are equivalent to algebras in $\mathcal{V}$ for the \iopd{}
obtained from the classical operad for $S$-coloured operads; the
rectification result then reduces to a rectification result for operad
algebras due to Pavlov and Scholbach~\cite{PavlovScholbach}.

\subsection{Notation and Terminology}
\begin{itemize}
\item We assume the existence of three nested Grothendieck universes;
  the sets contained in them are called \emph{small}, \emph{large} and
  \emph{very large}, respectively.
\item To the greatest extent possible, we work with
  $\infty$-categories without mentioning their specific implementation
  as quasicategories. In particular, we do not distinguish
  notationally between a category and its nerve and all categorical constructions such as taking (co)limits should be understood in the $\infty$-categorical setting. 
\item We write $\xS$ for the $\infty$-category of spaces (or
  $\infty$-groupoids) and, for an $\infty$-category $\xcc$, we write
  $\PSh(\xcc)$ for the $\infty$-category $\xFun(\mathcal{C}^\op,\xS)$
  of presheaves of spaces on $\xcc$.
\item To ease notation we will often leave the Yoneda embedding
  implicit, i.e.\ if $\mathcal{C}$ is a small \icat{} and $c$ is an
  object of $\mathcal{C}$ we will also use $c$ to denote the presheaf
  in $\PSh(\mathcal{C})$ represented by $c$.
\item We denote the usual simplicial indexing category by $\simp$.
\item We denote the unit of a symmetric monoidal \icat{} $\mathcal{V}$
  by $\bbone_{\mathcal{V}}$, or just $\bbone$ if $\mathcal{V}$ is
  clear from the context.
\item We write $\mathbb{F}$ for a skeleton of the category
  $\name{Fin}$ of finite sets, spanned by
  $\mathbf{n} := \{1,\ldots,n\}$. Similarly, we write $\mathbb{F}_{*}$
  for a skeleton of the category $\name{Fin}_{*}$ of finite pointed
  sets, spanned by
  $\langle n \rangle := \mathbf{n}_{+} := (\{*,1,\ldots,n\},
  *)$. 
\item For a finite set $K$, we write $K_{+}$ for the pointed set
  obtained from $K$ by adjoining a disjoint basepoint. If $|K| = n$,
  we will often implicitly identify $K_{+}$ with $\langle n \rangle$
  and thus regard $K_{+}$ as an object of $\mathbb{F}_{*}$.
\end{itemize}

For the reader's convenience we also recall some definitions and
notational conventions related to symmetric monoidal \icats{} and
\iopds{}, as presented in \cite{ha}.

\begin{definition}\label{def F*}
  A morphism $f\colon\langle m\rangle \to \langle n\rangle$ in $\mathbb{F}_{*}$ is called
  \emph{inert} if the preimage $f^{-1}(i)$ of $i$ has exactly one
  element for every $i \neq *$.  For
  $1\leq i,j \leq n$, we write
  $\rho_i\colon \langle n\rangle\to \langle 1\rangle$ for the inert map
  determined by
  \begin{equation*}
    \rho_i(j)= \begin{cases}
      1             & \text{if }i=j\\
      \ast               & \text{otherwise}.
    \end{cases}
  \end{equation*}
  A morphism $f\colon\langle m\rangle \to \langle n\rangle$ is called
  \emph{active} if $f^{-1}(\ast)=\ast$. The inert and active morphisms
  form a factorization system on $\xF_{*}$.
\end{definition}

\begin{definition}
  A \emph{symmetric monoidal $\infty$-category} is a coCartesian
  fibration $\xcc^\otimes\to \mathbb{F}_*$ such that, for
  every $n\geq 0$, the induced
  functors
  $\rho_{i,!}\colon \xcc^\otimes_{\langle
    n\rangle}\to\xcc^\otimes_{\langle 1\rangle}$ for $0 < i \leq n$
  exhibit $\xcc^\otimes_{\langle n\rangle}$ as the product
  $(\xcc^\otimes_{\langle 1\rangle})^{\times n}$. We often denote a
  symmetric monoidal \icat{} just by $\xcc^\otimes$, leaving the
  coCartesian fibration to $\xF_{*}$ implicit. Moreover, if
  $\mathcal{C}^{\otimes}$ denotes a symmetric monoidal \icat{} then we
  write $\xcc$ for $\xcc^\otimes_{\langle 1\rangle}$ and refer to this
  as the \emph{underlying $\infty$-category} of $\xcc^\otimes$. In
  this situation we will also, somewhat informally, refer to
  $\mathcal{C}^{\otimes}$ as a \emph{symmetric monoidal structure} on
  $\mathcal{C}$.
\end{definition}

\begin{definition}
  We say that a symmetric monoidal $\infty$-category $\xV^\otimes$ is
  \emph{presentably symmetric monoidal} if the underlying \icat{}
  $\mathcal{V}$ is presentable and the tensor product preserves
  colimits in each variable.
\end{definition}

\begin{remark}
  A symmetric monoidal $\infty$-category corresponds (via the
  straightening equivalence of \cite[\S 3.2]{ht}) to a functor
  $F\colon \xF_{*}\to \xCat_\infty$ such that the map
  $F(\langle n\rangle)\to F(\langle 1\rangle)^{\times n}$ induced by
  the maps $\rho_{i}$ is an
  equivalence. Thus the definition of symmetric monoidal \icats{} is
  analogous to that of (special) $\Gamma$-spaces, introduced as models for
  $E_{\infty}$-spaces by Segal \cite{SegalCatCohlgy}.
\end{remark}

\begin{definition}
  An \emph{$\infty$-operad} is a functor
  $p\colon \xxO \to \xF_{*}$ satisfying the following
  conditions:
  \begin{enumerate}
  \item For every inert morphism
    $f\colon \langle m\rangle\to \langle n\rangle$ and every object
    $x\in \xxO_{\langle m\rangle}$, there is a $p$-coCartesian
    lift of $f$ at $x$.
    \item Let $x\in \xxO_{\langle m\rangle}$, $y\in \xxO_{\langle n\rangle}$ be objects and let $f\colon \langle m\rangle\to \langle n\rangle$ be a morphism in $\xF_{*}$. Let $\xMap_{\xxO}^f(x,y)$ denote the union of those connected components of $\xMap_{\xxO}(x,y)$ which lie over $f$. The coCartesian lifts $y\to \rho_{i,!}(y)$ of the inert morphisms $\rho_{i}\colon \langle n\rangle\to \langle 1\rangle, 1\leq i\leq n$, induce a map 
    \[\xMap_{\xxO}^f(x,y)\to \prod_{1\leq i\leq n}\xMap_{\xxO}^{\rho_i\circ f}(x,\rho_{i,!}(y))\] which is an equivalence of spaces.
    \item The functors $\rho_{i,!}\colon \xxO_{\langle
                    n\rangle}\to\xxO_{\langle 1\rangle},1\leq i\leq n$,
                  given by coCartesian pushforward along $\rho_{i}$, give an equivalence \[\xxO_{\langle n\rangle} \to(\xxO_{\langle 1\rangle})^{\times n}.\]
  \end{enumerate}
\end{definition}
\begin{warning}
  For $\infty$-operads, our notational convention is slightly
  different from that of \cite{ha}, where \iopds{} are generally
  denoted $\mathcal{O}^{\otimes}$, with $\mathcal{O}$ referring to the
  \icat{} $\mathcal{O}^{\otimes}_{\angled{1}}$. For \iopds{} that are not
  symmetric monoidal \icats{}, however, this \icat{} is typically not
  of particular interest and is only rarely referred to. We therefore
  use the simpler notation $\mathcal{O}$ for an \iopd{}, without
  having a special notation for the \icat{}
  $\mathcal{O}_{\angled{1}}$.
\end{warning}

\subsection{Acknowledgments}
Hongyi: Many results contained in this paper originate from my
doctoral thesis. I thank my advisor Markus Spitzweck for suggesting
this interesting topic in the first place as well as for all the time
spent in numerous discussions during the course of this work. I would
also like to thank Denis-Charles Cisinski, David Gepner and Hadrian
Heine for interesting and helpful discussions. Moreover, I thank the
DFG for supporting me throughout my PhD and the Labex CEMPI
(ANR-11-LABX-0007-01) for enabling me to work on --- and finish ---
the manuscript at hand.

We thank Dmitri Pavlov and David White for helpful comments on model
structures for operad algebras.

\section{Enriched $\infty$-Operads as Segal Presheaves}\label{sec
  segal presheaves}

In this section we introduce and study our first model for
$\mathcal{V}$-enriched \iopds{}, as presheaves on a certain \icat{}
$\DFV$ satisfying Segal and continuity conditions. In
\S\ref{subsec:enrcat} we warm up to this by recalling the analogous
definition of enriched \icats{} as continuous Segal presheaves from
\cite{enriched}. In \S\ref{sec:DF} we then recall the definition of
the category $\DF$ and some of its basic properties from \cite{bar},
before we define $\DFV$ and continuous Segal presheaves on it in
\S\ref{subsec DFV}; these form a full subcategory $\PCS(\DFV)$ of
presheaves on $\DFV$. Next, \S\ref{subsec SegAlg} introduces an
alternative model for enriched \iopds{}, using algebras in
$\mathcal{V}$ for certain \icats{} $\DFX^{\op}$; we prove this is
equivalent to $\PCS(\DFV)$, but this model has the advantage that
certain functoriality properties are obvious.

The goal of \S\ref{subsec psh} is to show that if we enrich in the
\icat{} $\PSh(\mathcal{U})$ of presheaves on a small symmetric
monoidal \icat{} $\mathcal{U}$ (using the Day convolution) then
$\PCS(\DF^{\PSh(\mathcal{U})})$ is equivalent to Segal presheaves on
the much smaller \icat{} $\DF^{\mathcal{U}}$; as a special case, we
see that continous Segal presheaves for $\mathcal{S}$ are equivalent
to Barwick's Segal operads, i.e.\ Segal presheaves on $\DF$.  We
extend this in \S\ref{subsec enr loc} to get a small presentation of
$\PCS(\DFV)$ when $\mathcal{V}$ is a localization of a presheaf
\icat{} $\PSh(\mathcal{U})$; this allows us to show that $\PCS(\DFV)$ is presentable.

In \S\ref{subsec inner anodyne} we then study inner anodyne maps in
$\PSh(\DFV)$, in order to construct the tensoring of $\PCS(\DFV)$ with
Segal spaces in \S\ref{sec tensor product}. Finally, in \S\ref{sec under icat}
we discuss the underlying enriched \icat{} of an enriched \iopd{},
which will be needed in \S\ref{sec ffes}.

\subsection{Enriched $\infty$-Categories as Segal Presheaves}\label{subsec:enrcat}
As motivation for our definitions of enriched \iopds{} as Segal
presheaves, in this subsection we discuss the analogous definition in
the simpler setting of enriched \icats{}. This model of enriched
\icats{} as Segal presheaves was briefly introduced in \cite{enriched}.
We begin by recalling the definition of Rezk's Segal spaces,
introduced in \cite{Rezk} as a model for \icats{}.

\begin{definition}
  A presheaf $F\colon \simp^\op\to \xS$ is called a \emph{Segal space}
  if it satisfies the Segal condition, i.e., for every
  $[n]\in \simp^\op$, the map
  \[F([n])\to F([1])\times_{F([0])}\ldots\times_{F([0])}F([1]),\]
  induced by the maps
  $[1] \simeq \{i-1,i\} \hookrightarrow [n]$ and the
  maps $[0] \to [n]$ in $\simp$, is an equivalence in $\xS$. We write $\PSeg(\simp)$ for the full subcategory of $\PSh(\simp)$ spanned by Segal spaces.
\end{definition}
\begin{remark}\label{rem Segal spaces}
  For every $[n]\in \simp^\op$, the object
  $F([1])\times_{F([0])}\ldots\times_{F([0])}F([1])$ occurring in the
  previous definition should be thought of as the space of sequences of
  $n$ composable morphisms. Similarly, $F([n])$ should be interpreted
  as the space of all these sequences together with sequences of
  composites of adjacent maps. The Segal condition then says that
  there exists a homotopy coherent composition for composable
  morphisms.  By the Yoneda Lemma, a presheaf $F\colon \simp^\op\to \xS$ is a
  Segal space if and only if it is local with respect to all spine
  inclusions
  $\Delta^1\amalg_{\Delta^0}\ldots\amalg_{\Delta^0}\Delta^1\to
  \Delta^n$.
\end{remark}

\begin{defn}
  Given a symmetric monoidal
  \icat{} $\mathcal{V}^{\otimes}$, let $\mathcal{V}_{\otimes} \to
  \xF_{*}^{\op}$ denote the Cartesian fibration corresponding to the
  same functor $\xF_{*} \to \xCat_\infty$ as $\xV^\otimes\to \xF_{*}$. We define the \icat{} $\DV$
  by the pullback square
  \csquare{\DV}{\mathcal{V}_{\otimes}}{\simp}{\xF_{*}^{\op},}{}{}{}{\name{V}^\op}
  where the functor $\name{V} \colon \simp^\op \to \xF_{*}$ takes $[n]$ to
  $\angled{n}$ and a morphism $f \colon [n] \to [m]$ in $\simp$ to the morphism
  $\name{V}(f)\colon \angled{m} \to \angled{n}$ in $\mathbb F_*$ given by
  \[ \name{V}(f)(i) \mapsto
  \begin{cases}
    j, & \text{ if } f(j-1) < i \leq f(j),\\
    *, & \text{otherwise}.
  \end{cases}
  \]
\end{defn}

\begin{remark}
  If we regard an object $[n]\in \simp$ as a linear tree with $n$
  vertices and $n+1$ edges, then the functor $\name{V}^\op$ takes
  $[n]$ to the disjoint union of the set of vertices of $[n]$ with the
  base point. The functor $\name V$ is identical to the
  composite of the functor
  $\name{Cut}\colon \simp\to \name{Assoc}^{\otimes}$ defined in
  \cite[Construction 4.1.2.9]{ha} and the forgetful functor
  $\name{Assoc}^{\otimes}\to \mathbb F_*$.  In particular, it follows
  from the construction that an object of $\DV$ lying over
  $[n] \in \simp$ can be identified with a sequence
  $(v_{1},\ldots, v_{n})$ of objects of $\mathcal{V}$; we write
  $[n](v_{i})_{1\leq i\leq n}$ for this object.
\end{remark}

\begin{remark}
  There is no need to require $\mathcal{V}$ to be symmetric monoidal
  here --- we can equally well work with monoidal \icats{}, which can
  be described as coCartesian fibrations over $\simp^{\op}$ satisfying
  Segal conditions, and this is the definition used in
  \cite{enriched}. We have chosen to restrict to the symmetric
  monoidal case here for consistency with our discussion of \iopds{}
  below.
\end{remark}

\begin{defn}
  Suppose $\mathcal{V}$ is a symmetric monoidal \icat{}. A presheaf $F
  \in \PSh(\DV)$ is a \emph{Segal presheaf} if for every object $[n](v_{i})_{1\leq i\leq n}$ the map
  \[ F([n](v_{i})_{1\leq i\leq n}) \to F([1](v_{1})) \times_{F([0])}
  \cdots \times_{F([0])} F([1](v_{n})),\]
  induced by composition with the Cartesian morphisms over $\rho_{i}
  \colon [1] \to [n]$ and $[0] \to [n]$, is an equivalence.
\end{defn}

\begin{defn}
  Suppose $\mathcal{V}$ is a presentably symmetric monoidal
  \icat{}. Then a presheaf $F \in \PSh(\DV)$ is a \emph{continuous
    Segal presheaf} if it is a Segal presheaf and it is
  \emph{continuous} in the sense that the functor
  \[ \mathcal{V}^{\op} \simeq (\DVop)_{[1]} \to
  \mathcal{S}_{/F([0])^{\times 2}} \]
  preserves limits.
\end{defn}

\begin{remark}
  A continous Segal presheaf $\mathcal{C}$ on $\DV$ encodes a
  $\mathcal{V}$-enriched \icat{} in the following way: $\mathcal{C}([0])$ is the
  space of objects, and for $x,y \in \mathcal{C}([0])$ the functor 
  \[ \mathcal{C}([1](\blank,x,y)) \colon \mathcal{V}^{\op} \to \mathcal{S} \]
  given by taking the fibre of $\mathcal{C}([1](\blank))$ at $x,y$, preserves
  limits. Since $\mathcal{V}$ is presentable, it is therefore
  represented by an object $\mathcal{C}(x,y) \in \mathcal{V}$ --- this
  gives the morphisms from $x$ to $y$. 
\end{remark}

\subsection{The Category $\DF$}\label{sec:DF}
We now wish to introduce a definition of enriched \iopds{} analogous to
that of enriched \icats{} we just discussed. Our starting point will
be Barwick's definition of \iopds{} as Segal presheaves on a category
$\DF$. In this subsection we recall the definition of this category and
its basic structure, before turning to the relevant Segal conditions
in the following subsection.

\begin{definition}\label{def Delta' Phi}
  We let $\simp'_\mathbb{F} \to \simp$ be the Grothendieck fibration
  $\simp'_\mathbb{F}\to \simp$ associated to the functor
  $\Fun(\blank, \mathbb{F})$, where we view the objects of $\simp$ as
  categories in the usual way. 
\end{definition}
\begin{remark}
  An object in $\simp'_{\bbF}$ is of the form $([m], f)$, where
  $[m]\in\simp$ and $f$ is a functor $[m] \to \bbF$. We can think of
  this object as a sequence of morphisms
  $f(0)\to f(1)\to \ldots\to f(m)$ in $\mathbb{F}$. A morphism
  $([n],g)\to ([m], f)$ in $\simp'_\mathbb{F}$ is given by a morphism
  $\alpha\colon [n]\to [m]$ in $\simp$ and a natural transformation
  $\phi\colon g\to f \circ \alpha$. The morphism
  $(\alpha, \phi)\colon ([n],g)\to ([m], f)$ is then a Cartesian lift
  of $\alpha$ at the object $([m], f)$ if and only if the natural
  transformation $\phi\colon g\to f \circ \alpha$ is a natural
  isomorphism.
\end{remark}

\begin{definition}\label{def Delta_F}
  Let $\DF$ be the subcategory of $\simp'_\mathbb{F}$ containing all
  the objects, but only the morphisms
  $(\alpha, \phi)\colon ([n],g)\to ([m], f)$ which satisfy the
  following conditions:
  \begin{enumerate}
  \item For every $k$,  $0\leq k \leq n$, the morphism $\phi_k\colon g(k)\to f(\alpha(k))$ is injective. 
  \item For $k,l$, $0\leq k\leq l\leq n$, the induced square       
    \[
    \begin{tikzcd}
      g(k) \arrow{r}{\phi_{k}}
      \arrow{d} & f(\alpha(k)) \arrow{d} \\
      g(l) \arrow{r}{\phi_{l}} & f(\alpha(l)).
    \end{tikzcd}
    \]
    is a pullback square in $\mathbb{F}$. 
  \end{enumerate}
  We say an object $([m],f)\in \DF$ has \emph{length $m$}.
\end{definition}
\begin{notation}\label{no Iij}
  For an object $([m], f)\in \DF$ and $0\leq i\leq j\leq m$, let
  $f^{i,j}\colon f(i)\to f(j)$ denote the image of the morphism $i\to
  j$ in $[m]$ under the functor $f\colon [m]\to \mathbb{F}$. For a
  subset $S \subseteq f(j)$, let $f(i)_{S}$ denote the induced fibre
  product $f(i)\times_{f(j) }S$. We often write $I$ for an object
  $([m],f)\in \DF$, if there is no need to emphasize the length of the
  object. For $I=([m], f)$ and $k\in f(m)$, we write $I_k$ for the
  object in $\DF$ given by the sequence 
  \[f(0)_{\{k\}}\to \cdots\to f(m-1)_{\{k\}}\to \{k\} .\]
\end{notation}

\begin{remark}\label{rem forest}
As mentioned above, an object $([m], f)$ in $\DF$ is
given by a sequence of morphisms
    $$f(0)\to f(1)\to\ldots\to f(m)$$
    in $\mathbb{F}$. Objects in $\DF$ can be thought of
    as graphs (with levels), with \emph{vertices} and
    \emph{edges} given by the sets
    $\coprod_{i=1}^{m}f(i)$ and
    $\coprod_{i=0}^{m}f(i)$, respectively. Given a
    vertex
    $v\in f(i)\subseteq
    \coprod_{i=1}^{m}f(i)$,
    we say that an edge
    $e\in \coprod_{i=0}^{m}f(i)$ is an
    \emph{incoming} edge of $v$ if and only if
    $e\in f(i-1)_v\subseteq f(i-1)$ and $e$ is
    the unique \emph{outgoing} edge of $v$ if and only if
    $e$ and $v$ correspond to the same element in
    $f(i)$.
    Here is an example of a typical object in $\DF$ 
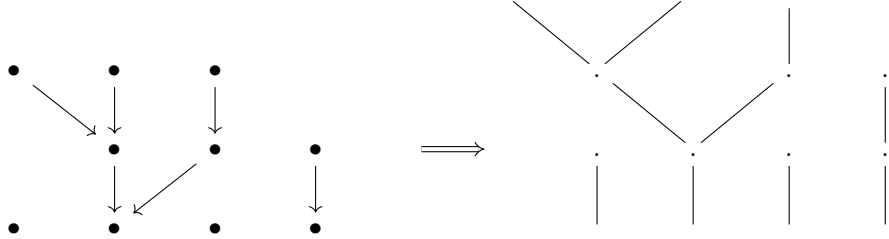
\[\begin{tikzcd}
{} & \ar[dr, dash] & & \ar[dl, dash]   & \ar[d, dash]        &               \\
{} & & . \ar[dr, dash] &                 & . \ar[dl, dash]  &    & .\ar[d, dash] \\
& . \ar[d, dash] & & . \ar[d, dash]  & & .\ar[d, dash]       & .\ar[d, dash] \\
& ~              &  & ~              & & ~                   & ~             
\end{tikzcd}
\]
which corresponds to a sequence $\mathbf{3}\to \mathbf{3}\to \mathbf{4}$ of finite sets.
  We see that an object $([m], f)$ can be regarded as a (non-planar)
  tree if $f(m)=\mathbf{1}$, and more generally as a finite \emph{forest}
  containing $|f(m)|$ trees.


\end{remark}
\begin{definition}\label{def inert}
  We say a morphism $\alpha \colon [n] \to [m]$ in $\simp$ is
  \emph{inert} if it is the inclusion of a subinterval in $[m]$,
  i.e. if $\alpha(i) = \alpha(0)+i$ for all $i$, and \emph{active} if
  it preserves the boundary, i.e. if $\alpha(0) = 0$ and
  $\alpha(n) = m$.
We say a map $(\alpha, \phi) \colon ([n], g) \to ([m], f)$ in $\DF$
  is
  \begin{enumerate}
    \item \emph{inert} if $\alpha$ is inert in $\simp$,
    \item \emph{active} if $\alpha$ is active in $\simp$ and $\phi_l\colon g(l)\to f(\alpha(l))$ is an isomorphism for every $0\leq l\leq n$.
  \end{enumerate}
We write $\DFint$ for the subcategory of $\DF$ containing only the inert morphisms in $\DF$.
\end{definition}
\begin{remark}\label{rem factorization system}
  The active and inert
  maps form a factorization system on $\simp$ which can be lifted
  along the Cartesian fibration $\DF\to \simp$, to give
  an active-inert factorization system on
  $\DF$ as follows: Any morphism $(\alpha,\phi)\colon ([n], g)\to ([m], f)$ in $\DF$ defines an active-inert factorization $[n]\overset{\alpha'}{\to} [k]\overset{\alpha''}{\to}[m]$ in $\simp$ and by Definition~\ref{def Delta_F} the inert map $\alpha''$ can be lifted to a map $(\alpha'',\phi'')\colon ([k], h)\to ([m], g)$ where, for $0\leq l\leq k$, $h(l):= f(\alpha''(l))\times_{f(\alpha(n))}\phi_n(g(n))$ and $\phi_i$ is the canonical inclusion $h(l)\hookrightarrow f(\alpha''(l))$. It is then clear that $(\alpha,\phi)$ factors as the composite of $(\alpha', \phi')$ and $(\alpha'',\phi'')$ where $\phi'_i\colon g(l)\simeq f(\alpha(l))\times_{f(\alpha(n))} \phi_n(g(n))\simeq h(\alpha'(l))$ is an equivalence for every $0\leq l\leq n$.
\end{remark}

In Remark~\ref{rem forest} we saw that objects in
$\DF$ can be interpreted as forests. Now we want to introduce and study the simplest examples of forests, which
are edges and corollas.  
\begin{definition}\label{def cor Phi}
  The \emph{edge} is the object
  $\xfe := ([0],\mathbf{1})\in \DF$ --- this is the
  trivial tree with no vertices and a single edge. A \emph{corolla} is a tree with
  exactly one vertex; more precisely, the \emph{$n$-corolla} $\xfc_n$
  is $([1], \mathbf{n} \to \mathbf{1})$. We write
  $\simp_\mathbb{F,\xel}$ for the full subcategory of
  $\simp_\mathbb{F,\xint}$ spanned by the corollas $\xfc_{n}$ and the edge
  $\xfe$. For an object $I\in \DF$, we let $\DFelI$ denote the category $\DFel\times_{\simp_{\mathbb{F},\xint}} \simp_{\mathbb{F},\xint/I}$.
\end{definition} 
    
\begin{remark}\label{rem corolla are vertices}
  In Remark~\ref{rem forest} we defined the sets of edges and vertices
  of an object $([m], f)$ to be $\coprod_{i=0}^{m}f(i)$ and
  $\coprod_{i=1}^{m}f(i)$. It clearly follows from the previous
  definition that the set of edges coincides with the set of morphisms
  $\xfe\to ([m], f)$ in $\DF$, while the set of vertices is given by
  the set of isomorphism classes of inert morphisms from corollas to
  $([m], f)$. In particular, each vertex $v\in\coprod_{i=1}^mf(i)$
  corresponds to an isomorphism class of inert morphisms of the form
  $\phi\colon \xfc_n\to ([m], f)$ such that $\phi(g(1))=v$.
\end{remark}

\begin{definition}\label{def cor functor}
  We write $\CorDF\colon \DF^{\op} \to \xF_{*}$ for the functor
  taking $([m], f) \in \DF$ to $(\coprod_{i=1}^m f(i))_{+}$ (i.e.\ the
  set of vertices of $([m], f)$ when viewed as a forest) and a
  morphism $(\alpha,\phi)\colon ([n], g)\to ([m], f)$ in $\DF$ to the map
  $(\coprod_{i=1}^m f(i))_{+} \to (\coprod_{j=1}^n g(j))_{+}$ given on
  the component $f(i)$ by the map
  $f(i) \to (\coprod_{j=1}^n g(j))_{+}$ taking $x \in f(i)$ to an
  object $y \in g(j)$ if $\alpha(j-1)< i \leq \alpha(j)$ and
  $f^{i,\alpha(j)}(x)=\phi_j(y)$, and to the base point $*$ otherwise.
\end{definition}
\begin{remark}
  According to Remark~\ref{rem corolla are vertices} an element
  $y\in g(j)$ corresponds to an isomorphism class of inert maps of the
  form $\xfc_p\to ([n],g)$ where $p=| g(j-1)_y| $. By composing one of
  these morphisms with $(\alpha,\phi)$ we obtain a map
  $\xfc_p\to ([m], f)$, well-defined up to isomorphism. The image is a
  subtree of $([m], f)$, given by the sequence
\[\mathbf{p}\simeq f(\alpha(j-1))_{\phi_j(y)}\to \ldots\to f(\alpha(j)-1)_{\phi_j(y)}\to\mathbf{1}.\]
Clearly, a corolla $x\in f(i)$ in $([m], f)$ lies in this subtree if
and only if $\alpha(j-1)< i \leq \alpha(j)$ and
$f^{i,\alpha(j)}(x)=\phi_j(y)$. Unpacking the previous definition we
therefore see that the map $\CorDF(\alpha, \phi)$ carries a corolla
$x\in \coprod_{i=1}^m f(i)$ to $y\in g(j)$ if $x$ is a corolla of the
subtree in $([m], f)$ induced by $y$ and $(\alpha,\phi)$. Since
different corollas in $([n],g)$ induce subtrees in $([m], f)$ which
have disjoint sets of incoming edges, the object $y$ satisfying the necessary
conditions is unique if it exists. Of course, it may happen that there
are corollas in $([m], f)$ that are not in the image of any corollas
in $([n],g)$; in this case, the map $\CorDF(\alpha, \phi)$ takes these corollas to the base point.
\end{remark}

\begin{lemma}
  The functor $\CorDF\colon \DF^{\op} \to \xF_{*}$  preserves inert-active
  factorizations.
\end{lemma}
\begin{proof}
  Let $(\alpha,\phi)\colon ([n],g)\to ([m], f)$ be an inert map in $\DF$ and
  let $y$ be a corolla in $([n],g)$. We want to show that
  $\CorDF(\alpha,\phi)$ is inert, i.e.
  $\CorDF(\alpha,\phi)^{-1}(y)\simeq \mathbf{1}$. The previous remark reveals that
  the set $\CorDF(\alpha,\phi)^{-1}(y)$ consists of corollas in the subtree induced
  by $y$ and $(\alpha,\phi)$. Since $\alpha$ is inert, the definition
  of $\DF$ implies that this subtree is a
  corolla equivalent to $y$.

  Now suppose $(\alpha,\phi)\colon ([n],g)\to ([m], f)$ is an active
  map in $\DF$. Then $\phi_k \colon g(k)\to f(\alpha(k))$ is an isomorphism for
  all $k$, and we have $\alpha(0) = 0$, $\alpha(n) = m$. This implies
  that, for every $k$ and $x\in f(k)$, there exists $l$ and
  $y\in g(l)$ such that $\alpha(l-1)< k \leq \alpha(l)$ and
  $f^{k,\alpha(l)}(x)=\phi_l(y)$.  Therefore, we see that the functor
  $\CorDF$ preserves inert and active morphisms, and in particular
  inert-active factorizations.
\end{proof}

\subsection{Segal Presheaves on $\DFV$}\label{subsec DFV}
We now want to define an enriched version of Barwick's \emph{Segal
  operads}, which we refer to as \emph{Segal presheaves} on
$\DF$. Before we do this, let us first recall Barwick's definition:
\begin{definition}\label{def:SegPsh}
  A presheaf $\xxO\colon \simp^\op_\mathbb{F}\to \xS$ is called a
  \emph{Segal presheaf} if the canonical map
  \[ \xxO(I)\to \lim_{J\in \DFelIop}
  \xxO(J)\]
  is an equivalence for every $I \in \DF$.  In this case we also say
  that $\xxO$ satisfies the \emph{Segal condition}; we write
  $\PSh_{\xSeg}(\DF)$ for the full subcategory of $\PSh(\DF)$ spanned
  by the Segal presheaves.
\end{definition}
Just as Segal spaces model $\infty$-categories, Segal
presheaves on $\DF$ model $\infty$-operads. In the latter case we
identify $\mathcal{O}(\xfc_n)$ with the space of multimorphisms with $n$ source
objects and one target object. The Segal condition identifies
$\mathcal{O}(I)$ for $I \in \DF$ as the space of composable trees of such
multimorphisms of shape $I$, and the functor $\mathcal{O}$ tells us how to compose
this data to a single multimorphism in a homotopy-coherently
associative manner.

\begin{remark}
  In the paper \cite{bar} Barwick introduced the theory of
  \emph{operator categories}. Speaking somewhat informally, we can
  think of an operator category as encoding the operations and
  coherences for a family of algebraic structures; key examples are
  the category $\mathbb{F}$ of finite sets and the category
  $\mathbb{O}$ of finite ordered sets. For every operator category
  $\Phi$, Barwick defines a category $\simp_\Phi$ (see
  \cite[Definition 2.4]{bar}) whose objects encode the tree-like
  structures of composable operations in $\Phi$. The category $\DF$
  defined above is a special case of this, and Definition~\ref{def:SegPsh}
  can be extended to the more general notion of Segal presheaves on
  $\simp_{\Phi}$ as in \cite[Definition~2.6]{bar} (where these are called
  \emph{Segal $\Phi$-operads}). Our work on enriched Segal presheaves
  on $\DF$ in this and the next section has an obvious variant for a
  general operator category $\Phi$, but we have chosen to state our
  results only for the most important case of symmetric operads in
  order to present the underlying idea as transparently as
  possible. Moreover, the lack of a good replacement for the dendroidal category for arbitrary operator categories (except $\mathbb F$ and $\mathbb O$) prevents us from generalizing the results of the last two sections.
\end{remark}

It is convenient to recall some alternative characterizations of Segal
presheaves, for which we need some notation:
\begin{notation}
  For $I = ([n], f)$ in $\DF$, we write
  $I|_{ij} := ([j-i], f|_{\{i,i+1,\ldots,j\}})$ for
  $0 \leq i < j \leq n$ and $I|_{i} := ([0], f(i))$.
\end{notation}
By rewriting the colimits in Definition~\ref{def:SegPsh} in various
ways, we get:
\begin{propn}\label{propn:SegDFcond}
  The following are equivalent for a presheaf $\mathcal{O} \in
  \PSh(\DF)$:
  \begin{enumerate}
  \item $\mathcal{O}$ is a Segal presheaf.
  \item $\mathcal{O}$ is local with respect to the morphisms
    \[ I_{\Seg} := \colim_{J \in \DFelIop} J \to I\]
    for all $I \in \DF$.
  \item $\mathcal{O}$ is local with respect to the morphisms
    \[ I|_{01} \amalg_{I|_{1}} I|_{12} \amalg_{I|_{2}} \cdots
    \amalg_{I|_{n-1}} I|_{(n-1)n} \to I\]
    for all $I \in \DF$,
    \[ \coprod_{i \in \mathbf{m}}([1], \mathbf{n_{i}} \to \mathbf{1})
    \to ([1], \mathbf{n} \to \mathbf{m}),\]
    for all $\mathbf{n} \to \mathbf{m}$ (including $\mathbf{m} =
    \mathbf{0}$), and
    \[ \coprod_{i \in \mathbf{m}}([0], \mathbf{1}) 
    \to ([0], \mathbf{m}),\]
    for all $\mathbf{m}$ (including $\mathbf{m} = \mathbf{0}$).
  \item $\mathcal{O}$ is local with respect to the morphisms
    \[ I_{\Seg} \to I\]
    for all $I = ([n],f)$ such that $f(n) = \mathbf{1}$, and
    \[ \coprod_{i \in f(n)} I_{i} \to I \]
    for all $I = ([n], f)$, where $I_{i} = ([n], f_{i})$ is obtained
    by taking the fibres at $i \in f(n)$.
  \item $\mathcal{O}|_{\DFint}$ is the right Kan extension of
    $\mathcal{O}|_{\DFel}$.\qed
  \end{enumerate}
\end{propn}

We now want to define an enriched version of this model for
\iopds{}. More precisely, for $\mathcal{V}$ a symmetric monoidal
\icat{} we will define an \icat{} $\DFV$ and then introduce a notion of
Segal presheaves on this \icat{}.

\begin{definition}\label{def V vee}
  Given a symmetric monoidal
\icat{} $\mathcal{V}^{\otimes}$, let $\mathcal{V}_{\otimes} \to \xF_{*}^{\op}$ denote the Cartesian fibration corresponding to the same functor $\xF_{*}\to \xCat_\infty$ as $\xV^\otimes\to \xF_{*}$. We define the \icat{} $\DFV$ by
  the pullback square
  \csquare{{\DF^\xV}}{\xV_\otimes}{ \DF }{\xF_{*}^{\op}}{}{}{}{\CorDF^\op}
  We also define $\DFVint$ and $\DFVel$ as the pullbacks $\DFV \times_{
  \DF} \DFint$ and $\DFV \times_{\DF} \DFel$, respectively.
\end{definition}

\begin{remark}\label{rem labeled tree}
  Note that the fibre $({\DF^\xV})_{\xfc_{n}}$ is equivalent to $\xV$
  for a corolla $\xfc_{n}\in\DF$. An object in ${\DF^\xV}$ should be
  thought of as an object $([m], f)\in\DF$ together with a labelling
  of each vertex in $([m], f)$ by an object of $\xV$. Therefore, if
  $\xfc_{n}$ is a corolla in $\DF$, we write $\xfc_{n}(v)$ for the
  object in $\DF^\xV$ lying over $\xfc_{n}$ and labeled by $v\in
  \xV$. Given an arbitrary object $I\in \DF$ and objects $v_c\in \xV$,
  $c \in \CorDF(I)$, we write $I(v_c)_c$ for the object in $\DFV$
  corresponding to $(v_{c})_{c}$ under the equivalence
  $(\DFV)_{I} \simeq \mathcal{V}^{\times | \CorDF(I)| }$. We often
  write $\overline{I}$ instead of $I(v_c)_c$ if we do not want to
  explicitly introduce notation for the labeling $(v_c)_c$. In particular, we let
  $\overline{\xfe}$ denote the unique object of
  $(\DFV)_{\mathfrak{e}} \simeq *$.
\end{remark}

\begin{remark}
  More generally, instead of a symmetric monoidal \icat{}
  $\mathcal{V}$ we could consider enriching in a coCartesian fibration
  $\mathcal{M} \to \DF^{\op}$ satisfying Segal conditions in the form
  \[ \mathcal{M}_{I} \simeq \prod_{\mathfrak{c}_{n} \to
    I}\mathcal{M}_{\mathfrak{c}_{n}}\]
  (where the product is over isomorphism classes of inert maps
  $\mathfrak{c}_{n} \to I$); such an $\mathcal{M}$ is an ``internal
  \iopd{}'' in \icats{}. We expect that most of our results should
  straightforwardly generalize to this setting. However, as we do not
  wish to develop the theory of ``internal \iopds{} in $\CatI$''
  (because we are not aware of any interesting examples of these, let
  alone ones that one might want to enrich \iopds{} in) we have chosen
  to consider only enrichment in symmetric monoidal \icats{} for
  simplicity.
\end{remark}

\begin{definition}\label{def Segal presheaf}
  Let $\mathcal{V}$ be a symmetric monoidal \icat{}. We say a presheaf
  $\mathcal{O} \in \PSh(\DFV)$ is a \emph{Segal presheaf} if for every
  object $\overline{I} \in \DFV$ lying over $I$ in $\DF$, the  canonical map
  \[ \xxO(\overline I) \to \lim_{\psi \in \DFelIop}\xxO(\psi^* \overline I)\] is 
  an equivalence, where $\psi^*x\to x$ is the Cartesian lift of the
  inert map $\psi$ (corresponding to a coCartesian morphism in
  $\xV^\otimes$). 
  We write $\PSeg(\DFV)$ for the full subcategory of $\PSh(\DFV)$ spanned by the Segal presheaves. 
\end{definition}

\begin{defn}\label{def cts Seg psh}
  Let $\mathcal{V}$ be a presentably symmetric monoidal \icat{}. We
  say a presheaf $\mathcal{O} \in \PSh(\DFV)$ is a \emph{continuous
    Segal presheaf} if it is a Segal presheaf and moreover for every
  $n$, the functor
\[\xxO(\xfc_n(\blank))\colon \mathcal{V}^{\op} \simeq (\DFVop)_{\xfc_n} \to
\xS_{/\xxO(\overline \xfe)^{n+1}}\]
induced by the Cartesian lifts of the $n+1$ morphisms $\xfe\to \xfc_n$
preserves all small limits. We write $\PCS(\DFV)$ for the full
subcategory of $\PSh(\DFV)$ spanned by the continuous Segal
presheaves.  
\end{defn}

\begin{remark}
  Continuous Segal presheaves give a model of $\mathcal{V}$-enriched
  \iopds{}: Given $\mathcal{O} \in \PCS(\DFV)$ and
  $x_{1},\ldots, x_{n},y \in \mathcal{O}(\mathfrak{e})$, let
  $\mathcal{O}(\mathfrak{c}_{n} (v, {x_{1},\ldots,x_{n},y}))$ be defined
  by the pullback square \nolabelcsquare{\mathcal{O}(\mathfrak{c}_{n}(v,{x_{1},\ldots,x_{n},y}))}{\mathcal{O}(\mathfrak{c}_{n},
    v)}{\{x_{1},\ldots,x_{n},y\}}{\mathcal{O}(\mathfrak{e})^{\times
      (n+1)},} where $y$ is in the ``outgoing'' coordinate.  Then the
  presheaf
  $\mathcal{O}(\mathfrak{c}_{n} (\blank, {x_{1},\ldots,x_{n},y})) \colon
  \mathcal{V}^{\op} \to \mathcal{S}$
  is limit-preserving. Since $\mathcal{V}$ is presentable, this means
  this presheaf is representable by some object
  \[\mathcal{O}(x_{1},\ldots,x_{n}; y) \in \mathcal{V}.\] This is the
  object describing the multimorphisms from $(x_{1},\ldots,x_{n})$ to
  $y$ in the enriched \iopd{} $\mathcal{O}$.
\end{remark}
\begin{remark}
  To obtain a definition of enriched \iopds{} in this style when
  $\mathcal{V}$ is not presentable, we could consider those presheaves
  $\mathcal{O}$ on $\DFV$ such that the presheaves
  $\mathcal{O}(\mathfrak{c}_{n} (\blank, {x_{1},\ldots,x_{n},y}))$ are
  all representable. However, in such settings it is more natural to
  consider an alternative definition of enriched \iopds{}, such as
  that of \S\ref{subsec SegAlg}. Moreover, most of the results we wish to prove
  using Segal presheaves only hold in the presentable case, so there
  is no reason to consider such a generalization.
\end{remark}

We end this section by proving some equivalent reformulations of the
definitions of Segal presheaves and continuous Segal presheaves. First
we give an analogue of Proposition~\ref{propn:SegDFcond}:
\begin{propn}\label{propn:SegDFVcond}
  The following are equivalent for a presheaf $\mathcal{O} \in
  \PSh(\DFV)$:
  \begin{enumerate}[(1)]
  \item $\mathcal{O}$ is a Segal presheaf.
  \item $\mathcal{O}$ is local with respect to the morphisms
    \[ \overline{I}_{\Seg} := \colim_{J \in \DFelIop} \overline{J} \to \overline{I}\]
    for all $\overline{I} \in \DFV$.
  \item $\mathcal{O}$ is local with respect to the morphisms
    \[ \overline{I}|_{\Delta^{n}_{\Seg}} := \overline{I}|_{01} \amalg_{\overline{I}|_{1}} \overline{I}|_{12} \amalg_{\overline{I}|_{2}} \cdots
    \amalg_{\overline{I}|_{n-1}} \overline{I}|_{(n-1)n} \to \overline{I}\]
    for all $\overline{I} \in \DF$,
    \[ \coprod_{i \in \mathbf{m}}\mathfrak c_{n_i}(v_{i})
    \to ([1], \mathbf{n} \to \mathbf{m})(v_{1},\ldots,v_{m}),\]
    for all $\mathbf{n} \to \mathbf{m}$ (including $\mathbf{m} =
    \mathbf{0}$), and
    \[ \coprod_{i \in \mathbf{m}}\mathfrak e
    \to ([0], \mathbf{m}),\]
    for all $\mathbf{m}$ (including $\mathbf{m} = \mathbf{0}$).
  \item $\mathcal{O}$ is local with respect to the morphisms
    \[ \overline{I}_{\Seg}  \to \overline{I}\]
    for all $\overline{I}$ over $I = ([n],f)$ such that $f(n) = \mathbf{1}$, and
    \[ \coprod_{i \in f(n)} \overline{I}_{i} \to \overline{I} \]
    for all $I = ([n], f)$, where $I_{i} = ([n], f_{i})$ is obtained
    by taking the fibres at $i \in f(n)$.
  \item $\mathcal{O}|_{\DFVintop}$ is the right Kan extension of
    $\mathcal{O}|_{\DFVelop}$
  \end{enumerate}
\end{propn}

To prove this we use the following observation:
\begin{lemma}\label{lem:cofslice}
  Suppose $p \colon \mathcal{E} \to \mathcal{B}$ is a Cartesian
  fibration. If $\mathcal{B}'$ is a full subcategory of $\mathcal{B}$
  and $e \in \mathcal{E}$, let
  $\mathcal{E}' := \mathcal{E} \times_{\mathcal{B}} \mathcal{B}'$,
  $\mathcal{E}'_{/e} := \mathcal{E}' \times_{\mathcal{E}}
  \mathcal{E}_{/e}$
  and $\xB'_{/p(e)}:=\xB'\times_\xB\xB_{/p(e)}$. Let
  $\mathcal{E}''_{/e}$ denote the full subcategory of
  $\mathcal{E}'_{/e}$ containing only the Cartesian morphisms to
  $e$. Then $\mathcal{E}''_{/e} \to \mathcal{E}'_{/e}$ is cofinal and
  $\mathcal{E}''_{/e} \to \mathcal{B}'_{/p(e)}$ is an equivalence.
\end{lemma}
\begin{proof}
  By \cite[Theorem 4.1.3.1]{ht}, the map
  $\mathcal{E}''_{/e} \to \mathcal{E}'_{/e}$ is cofinal if and only if
  the $\infty$-category
  $\xcc:=\mathcal{E}''_{/e}\times_{
    \mathcal{E}'_{/e}}(\mathcal{E}'_{/e})_{\phi/}$ is weakly
  contractible for every object $\phi\in \mathcal{E}'_{/e}$. The
  definition of $\xcc$ implies that an object in $\xcc$ is given by a
  factorization $\phi=\beta\circ \alpha$ such that the morphism
  $\beta$ is $p$-Cartesian. Moreover, a morphism from
  $\phi=\beta_0\circ \alpha_0$ to $\phi=\beta_1\circ \alpha_1$ in
  $\xcc$ is given by a factorization $\alpha_1= \alpha'\circ\alpha_0$
  such that $\beta_1\circ\alpha'=\beta_0$.  Since $\beta_0$ and
  $\beta_1$ are Cartesian, \cite[Proposition 2.4.1.7]{ht} implies that
  $\alpha'$ is Cartesian as well.  We now prove the weak
  contractibility of $\xcc$ by showing that it has an initial object
  given by the factorization $\phi=\beta_0\circ \alpha_0$, where
  $\beta_0$ is a Cartesian lift of $p(\phi)$. If an object in $\xcc$
  is given by a factorization $\phi=\beta_1\circ \alpha_1$, then we
  have a map from $\phi=\beta_0\circ \alpha_0$ to
  $\phi=\beta_1\circ \alpha_1$ given by factoring $\alpha_1$ into a
  Cartesian lift $\alpha'$ of $p(\alpha_1)$ and a map $\alpha'$ lying
  in the fibre $\xE_{e_0}$. Then $\beta_1\circ \alpha'$ coincides with
  $\beta_0$ as both are Cartesian lifts of
  $p(\beta_1)\circ p(\alpha_1)= p(\phi)$. Since every map from
  $\phi=\beta_0\circ \alpha_0$ to $\phi=\beta_1\circ \alpha_1$ is
  necessarily induced by a Cartesian lift of $p(\alpha_1)$, we see
  that this map is essentially unique.
  
  To prove that $\mathcal{E}''_{/e} \to
  \mathcal{B}'_{/p(e)}$ is an equivalence, we first observe that in the commutative diagram
  \[\begin{tikzcd}
  &  \xE'_{/e}\ar[rr]\ar[ld]\ar[dd]& & \xE_{/e}\ar[ld]\ar[dd]\\
  \xB'_{/p(e)} \ar[rr,crossing over]\ar[dd]&  & \xB_{/p(e)}&\\
  &    \xE'\ar[rr]\ar[ld] && \xE\ar[ld]\\
  \xB'\ar[rr] & & \xB\ar[from=uu, crossing over]\\
      \end{tikzcd}\]
      the bottom, front, and back squares are Cartesian by
      definition. Hence, the top square is a pullback and
      $\xE'_{/e}\to \xB'_{/p(e)}$ is a Cartesian fibration which restricts
      to a right fibration
      $\mathcal{E}''_{/e} \to \mathcal{B}'_{/p(e)}$. By
      \cite[Proposition 2.1.3.4]{ht}, it is a trivial fibration if all
      fibres are contractible. But this is clear because the fibre
      over an object $\psi\in \mathcal{B}'_{/p(e)}$ can be identified
      with the full subcategory of the fibre $(\xE'_{/e})_\psi$
      spanned by the Cartesian lifts of $\psi$, which is contractible.
\end{proof}

\begin{proof}[Proof of Proposition~\ref{propn:SegDFVcond}]
  The proof that (1) is equivalent to (2), (3), and (4) is the same as
  in the non-enriched case. By applying Lemma~\ref{lem:cofslice} to
  the Cartesian fibration
  $p\colon \DFVint\to \simp_{\mathbb{F},\text{int}}$ and the full
  subcategory $\simp_{\mathbb{F},\xel}$, we obtain that (1) is
  also equivalent to (5).
\end{proof}

\begin{propn}\label{propn:CtsDFV}
	Let $\emptyset$ denote the initial object in $\xV$.
  The following are equivalent for a Segal presheaf $\xxO \in
  \PSh_{\Seg}(\DFV)$:
  \begin{enumerate}[(1)]
  \item $\xxO$ is continuous.
  \item For every $n$, the presheaf 
    \[\xxO(\mathfrak{c}_{n}( \blank)) \colon \mathcal{V}^{\op}
    \simeq (\DF^{\xV,\op})_{\mathfrak{c}_{n}} \to \mathcal{S}\]
    preserves limits of diagrams of the form $\phi \colon
    \mathcal{I}^{\triangleright} \to \mathcal{V}^\op$ where $\phi(\infty)
    \simeq \emptyset$, and the natural map
    $\xxO(\mathfrak{c}_{n}(\emptyset)) \to
    \prod_{n+1}\xxO(\overline{\mathfrak{e}})$ is an equivalence.
  \item $\xxO$ is local with respect to the map $\coprod_{n+1}
    \overline{\mathfrak{e}} \to \mathfrak{c}_{n}(\emptyset)$ and the map
    $\colim_{\mathcal{I}^{\triangleleft}}\mathfrak{c}_{n}(\phi) \to
    \mathfrak{c}_{n}(\colim_{\mathcal{I}^{\triangleleft}} \phi)$ for
    every diagram $\phi$ such that $\phi(-\infty) \simeq \emptyset$.
  \item For every $n$, the presheaf 
    \[\xxO(\mathfrak{c}_{n}( \blank)) \colon \mathcal{V}^{\op}
    \simeq (\DFV)^{\op}_{\mathfrak{c}_{n}} \to \mathcal{S}\]
    preserves weakly contractible limits, and the natural map
    $\xxO(\mathfrak{c}_{n}(\emptyset)) \to
    \prod_{n+1}\xxO(\overline{\mathfrak{e}})$ is an equivalence.
  \item $\xxO$ is local with respect to the map $\coprod_{n+1}
    \overline{\mathfrak{e}} \to \mathfrak{c}_{n}(\emptyset)$ and the map
    $\colim_{\mathcal{I}}\mathfrak{c}_{n}(\phi) \to
    \mathfrak{c}_{n}(\colim_{\mathcal{I}} \phi)$ for every weakly
    contractible diagram in $\mathcal{V}$.
  \end{enumerate}
\end{propn}
\begin{proof}
By definition a Segal presheaf $\xxO$ is continuous if and only if the functor 
\[ \xxO(\xfc_n,\blank)\colon \mathcal{V}^{\op} \simeq
(\DFVop)_{\xfc_n} \to \xS_{/\xxO(\xfe)^{n+1}}\]
preserves small limits. Since the limit of a functor $\phi \colon \mathcal{I}
\to \mathcal{V}^{\op}$ is the same as the limit of its right Kan
extension along $\mathcal{I} \hookrightarrow
\mathcal{I}^{\triangleright}$, which takes $\infty$ to $*$, this is
equivalent to the preservation of such \emph{conical} limits and of
the terminal object. Moreover, by \cite[Lemma 2.2.6]{anmnd} the
preservation of conical limits is equivalent to the preservation of
weakly contractible limits.

Here the preservation of the terminal object is obviously equivalent to 
$\xxO(\mathfrak{c}_{n}(\emptyset))$ coinciding with the terminal object
$\xxO(\overline{\mathfrak{e}})^{n+1}\in \xS_{/\xxO(\xfe)^{n+1}}$.  Since
a functor $\mathcal{V}^{\op} \to \xS_{/\xxO(\xfe)^{n+1}}$ preserves
weakly contractible limits if and only its composition with the
forgetful functor $\xS_{/\xxO(\xfe)^{n+1}}\to \xS$ does, we see that
conditions (1), (2), and (4) are equivalent. The equivalence between
(4) and (5) follows from the fact that a
functor $\xV^\op\to \xS$ preserves all weakly contractible limits if
and only if is local with respect to all maps
$\colim_{\mathcal{I}}\mathfrak{c}_{n}(\phi) \to 
\mathfrak{c}_{n}(\colim_{\mathcal{I}} \phi)$ of presheaves where $\phi$ is a weakly
contractible diagram in $\mathcal{V}$; the same argument also shows
that (2) and (3) are equivalent.
\end{proof}

\begin{definition}\label{defn:ssat}
  We call a class $\mathbb{S}$ of morphisms in a cocomplete $\infty$-category
  $\xcc$ \emph{strongly saturated} if
  \begin{enumerate}[(i)]
  \item it satisfies the $2$-of-$3$ property,
  \item it is stable under pushouts along any morphism in $\xcc$,
  \item the full subcategory of $\xFun(\Delta^1, \xcc)$ spanned by
    $\mathbb{S}$ is stable under small colimits.
  \end{enumerate}
  For any class of morphisms $\mathbb{S}$, there exists a \emph{smallest}
  strongly saturated class $\overline{\mathbb{S}}$ that contains $\mathbb{S}$.
\end{definition}

\begin{definition}
  If $\mathbb{S}$ is a class of morphisms in $\mathcal{C}$, then we say an
  object $c \in \mathcal{C}$ is \emph{$\mathbb{S}$-local} if
  $\Map_{\mathcal{C}}(\blank, c)$ takes the elements of $\mathbb{S}$ to
  equivalences in $\mathcal{\mathbb{S}}$. A morphism $f \colon x \to y$ in
  $\mathcal{C}$ is then an \emph{$\mathbb{S}$-equivalence} if
  $f^{*} \colon \Map_{\mathcal{C}}(y,c) \to \Map_{\mathcal{C}}(x, c)$
  is an equivalence for every $\mathbb{S}$-local object $c$. Let $\name{Eq}(\mathbb{S})$
  denote the class of $\mathbb{S}$-equivalences; if $\mathcal{C}$ is
  cocomplete, then $\name{Eq}(\mathbb{S})$ is a strongly saturated class by
  \cite[Lemma 5.5.4.11]{ht}, and $\overline{\mathbb{S}} \subseteq \name{Eq}(\mathbb{S})$.
\end{definition}

\begin{remark}\label{rem Sloc}
  If $\mathbb{S}$ is a small set of morphisms in a presentable \icat{}
  $\mathcal{C}$, then $\name{Eq}(\mathbb{S}) = \overline{\mathbb{S}}$ by
  \cite[Proposition 5.5.4.15]{ht}.  Moreover, if $\mathcal{C}_{\mathbb{S}}$
  denotes the full subcategory of $\mathcal{C}$ spanned by the
  $\mathbb{S}$-local objects, then $\mathcal{C}_{\mathbb{S}}$ is again presentable and
  the inclusion $\mathcal{C}_{\mathbb{S}} \hookrightarrow \mathcal{C}$ admits a
  left adjoint, which exhibits $\mathcal{C}_{\mathbb{S}}$ as the localization
  of $\mathcal{C}$ that inverts the morphisms in $\mathbb{S}$.
\end{remark}

\begin{defn}\label{def:ssatgen}
  If $\mathbb{S}$ is a class of morphisms in a cocomplete \icat{}
  $\mathcal{C}$, we say that the class $\name{Eq}(\mathbb{S})$ is the strongly
  saturated class \emph{generated} by $\mathbb{S}$, and call the elements of
  $\mathbb{S}$ the \emph{generators} of $\name{Eq}(\mathbb{S})$.
\end{defn}

\begin{remark}
  Our use of the term ``generated by $\mathbb{S}$'' for the class
  $\name{Eq}(\mathbb{S})$ (rather than the class $\overline{\mathbb{S}}$, which might a
  priori be smaller than $\name{Eq}(\mathbb{S})$ when $\mathbb{S}$ is not small) is
  somewhat non-standard, but will be convenient for us: In practice we
  will be interested in ``weak equivalences'' of the form
  $\name{Eq}(\mathbb{S})$ where $\mathbb{S}$ is not small, and we will proceed to find a
  small set $\mathbb{S}'$ such that
  $\name{Eq}(\mathbb{S}) = \name{Eq}(\mathbb{S}') = \overline{\mathbb{S}'}$ (which also implies
  $\name{Eq}(\mathbb{S}) = \overline{\mathbb{S}}$).
\end{remark}

\begin{definition}
	We call elements in the strongly saturated class of morphisms in $\PSh(\simp)$ generated by the set of spine inclusions $\Delta^1\amalg_{\Delta^0}\ldots\amalg_{\Delta^0}\Delta^1\to
	\Delta^n$ \emph{Segal equivalences}. By Remark~\ref{rem Segal spaces} and Remark~\ref{rem Sloc}, $\PSeg(\simp)$ is given by a localization of $\PSh(\simp)$ with respect to Segal equivalences.
\end{definition}

\begin{definition}\label{def set s}
  We say that a map $F \to G$ in $\PSh(\DFV)$ is a \emph{Segal
    equivalence} if $\Map(G, \mathcal{O}) \to \Map(F, \mathcal{O})$ is
  an equivalence for every Segal presheaf $\mathcal{O}$, and a
  \emph{continuous Segal equivalence} if this is an equivalence for
  every continuous Segal presheaf $\mathcal{O}$. These are both
  strongly saturated classes of morphisms, since the Segal
  equivalences are of the form $\name{Eq}(\mathbb{S})$ where $\mathbb{S}$ consists of
  the maps listed in (2),
  (3), and (4) of Proposition~\ref{propn:SegDFVcond}, while the
  continous Segal equivalences are the $\mathbb{S}$-equivalences for these together with the
  maps listed in (3) or (5) of
  Proposition~\ref{propn:CtsDFV}. Explicitly, we can view the
  continuous Segal equivalences as generated by the 
  following morphsims:
  \begin{enumerate}
    \item $\overline I_{\Seg} \to \overline I$ for all $\overline I\in\DFV$,
    \item $\colim_{\mathcal{I}}\mathfrak{c}_{n}(\phi) \to
    \mathfrak{c}_{n}(\colim_{\mathcal{I}} \phi)$ for every weakly
    contractible diagram in $\mathcal{V}$,
    \item $\coprod_{n}
    \overline{\mathfrak{e}} \to \mathfrak{c}_{n}(\emptyset)$.
  \end{enumerate}
  Note, however, that this set of morphisms is \emph{not} small.
\end{definition}

Continuous Segal presheaves have an obvious functoriality in
colimit-preserving symmetric monoidal functors:
\begin{lemma}\label{lem DFV functor}
  Suppose $F \colon \mathcal{V} \to \mathcal{W}$ is a symmetric
  monoidal colimit-preserving functor between presentably symmetric
  monoidal \icats{}. Then $F$ induces a functor $\DFV \to
  \DF^{\mathcal{W}}$, which we also denote $F$, and composition with
  $F^{\op}$ induces a functor $\PCS(\DF^{\mathcal{W}}) \to \PCS(\DFV)$.
\end{lemma}

\subsection{Segal Presheaves vs. Algebras}\label{subsec SegAlg}
\begin{defn}
  Given a space $X$, we write $\DFX \to \DF$ for the right fibration associated to the functor $\DF^{\op} \to \mathcal{S}$ given by the right Kan extension of the functor $\{\mathfrak{e}\} \to
  \mathcal{S}$ with value $X$ along the inclusion $\{\mathfrak{e}\}
  \hookrightarrow \DF^{\op}$. We write $I(x_i)_i$ for an object in $\DFX$ and we view $\DFX^{\op}$ as living over
  $\xF_{*}$ via the composite map \[\DFX^{\op} \to \DF^{\op}
  \xto{\CorDF} \xF_{*}.\] For $X \in \mathcal{S}$ and
  $\mathcal{V}$ a symmetric monoidal
  \icat{}, we define $\DFXV := \DFX \times_{\DF}\DFV$ and we write $I(v_c, x_i)_{c,i}$ or $\overline I(x_i)_i$ for its objects.
\end{defn}
\begin{remark}
  The right Kan extension $\DF^{\op} \to \mathcal{S}$ takes an object $I$ to a product of copies
  of $X$ indexed by the number of edges of $I$.
\end{remark}
\begin{defn}
  If $\mathcal{V}^{\otimes}\to \xF_{*}$ is a symmetric monoidal
  \icat{}, then a \emph{$\DFX^{\op}$-algebra} in $\mathcal{V}$ is a
  functor $\DFX^{\op} \to \mathcal{V}^{\otimes}$ over $\xF_{*}$
  that takes the inert morphisms lying over $\rho_{i}$ to coCartesian
  morphisms. We write $\Alg_{\DFX^{\op}}(\mathcal{V})$ for the full
  subcategory of $\Fun_{\xF_{*}}(\DFX^{\op}, \mathcal{V}^{\otimes})$
  spanned by the algebras. This is clearly contravariantly functorial in $X$, and we
  write $\AlgDFS(\mathcal{V}) \to \mathcal{S}$ for the
  Cartesian fibration associated to the functor $\mathcal{S}^{\op} \to
  \CatI$ taking $X$ to $\Alg_{\DFX^{\op}}(\mathcal{V})$.
\end{defn}

Our goal in this subsection is to construct an equivalence between the
\icats{} $\PCS(\DFV)$ and $\AlgDFS(\mathcal{V})$:
\begin{thm}\label{thm:PCSisAlgLT}
  Let $\mathcal{V}$ be a presentably symmetric monoidal \icat{}. There is an equivalence of \icats{} over $\mathcal{S}$,
  \[
  \begin{tikzcd}
    \PCS(\DFV) \arrow{dr}{\name{ev}_{\mathfrak{e}}} \arrow{rr}{\sim} &
      & \AlgDFS(\mathcal{V}) \arrow{dl} \\
      & \mathcal{S}.
  \end{tikzcd}
  \]
\end{thm}
To prove this we
will show that both sides are equivalent to an intermediate model
given by continuous $\DFXVop$-monoids, in the following sense:
\begin{defn}
  We say a presheaf $F
  \in \PSh(\DFXV)$ is a \emph{$\DFXVop$-monoid} if for every object
  $\overline{I}$ in $\DFXV$ lying over $I$ in $\DF$, the natural map
  \[ F(\overline{I}) \to \prod_{v \in \CorDF(I)}
  F(v^{*}\overline{I})\]
  is an equivalence, where $v^{*}\overline{I} \to \overline{I}$
  denotes the Cartesian morphism in $\DFXV$ over an inert map $v \colon
  \mathfrak{c}_{n} \to I$ in $\DF$ corresponding to $v \in \CorDF(I)$. We say a
  $\DFXVop$-monoid is \emph{continous} if for every
  $\tilde{\mathfrak{c}}_{n}$ in $\DFX$ over $\mathfrak{c}_{n}$ in
  $\DF$ the presheaf
  \[ \mathcal{V}^{\op} \simeq (\DFXVop)_{\tilde{\mathfrak{c}}_{n}} \to
  \mathcal{S} \] preserves limits. We write $\PM(\DFXV)$ for the full
  subcategory of $\PSh(\DFXV)$ spanned by the $\DFXVop$-monoids and
  $\PCM(\DFXV)$ for the full subcategory spanned by the continuous
  $\DFXVop$-monoids.
\end{defn}

\begin{defn}
  Given a coCartesian fibration $\mathcal{E} \to \mathcal{C}$
  corresponding to a functor $F \colon \mathcal{C} \to \CatI$, let
  $\PSh_{\mathcal{C}}(\mathcal{E}) \to \mathcal{C}$ denote the
  Cartesian fibration associated to the functor 
  \[ \mathcal{C}^{\op} \xto{F^{\op}} \CatI^{\op} \xto{\PSh}
  \LCatI.\]
\end{defn}

\begin{propn}\label{propn:PCMisAlg}
  Let $\mathcal{V}$ be a presentably symmetric monoidal \icat{}. Then
  there is an equivalence of \icats{}
  $\PCM(\DFXV) \simeq \Alg_{\DFXop}(\mathcal{V})$, natural in
  $X \in \mathcal{S}$.
\end{propn}
\begin{proof}
  Since $\DFXVop \to \DFXop$ is a coCartesian fibration, by
  \cite[Proposition 7.3]{laxcolimits} we can identify
 the \icat{} $\Fun(\DFXVop, \mathcal{S})$ with 
 \[ \Fun_{\DFXop}(\DFXop, \PSh_{\DFXop}(\DFXVop)) \simeq
 \Fun_{\xF_{*}}(\DFXop,
 \PSh_{\xF_{*}}(\mathcal{V}^{\op,\otimes})).\]
 If $M$ is a continuous $\DFXVop$-monoid, then the corresponding
 functor $\DFXop \to \PSh_{\xF_{*}}(\mathcal{V}^{\op,\otimes})$
 sends $\overline{I} \in \DFXop$ to an object in the full subcategory 
 \[ \Fun^{R}((\mathcal{V}^{\op})^{\times |\CorDF(I)|}, \mathcal{S})
 \hookrightarrow \Fun((\mathcal{V}^{\op})^{\times |\CorDF(I)|},
 \mathcal{S}) \simeq
 \PSh_{\xF_{*}}(\mathcal{V}^{\op,\otimes})_{\CorDF(I)}\]
 of functors that preserve limits. Since $\mathcal{V}$ is presentable,
 this \icat{} can be identified with $\mathcal{V}^{\times
   |\CorDF(I)|}$ under the Yoneda embedding, and the full subcategory
 of $\PSh_{\xF_{*}}(\mathcal{V}^{\op,\otimes})$ spanned by these
 objects for all $\langle n \rangle \in \xF_{*}$ can be identified with
 $\mathcal{V}^{\otimes}$. Furthermore, under this equivalence the full subcategory
 $\PCM(\DFXV)$ of $\PSh(\DFXV)$ is identified with 
 the full subcategory $\Alg_{\DFXop}(\mathcal{V})$ of $\Fun_{\xF_{*}}(\DFXop,
 \PSh_{\xF_{*}}(\mathcal{V}^{\op,\otimes}))$.
\end{proof}

\begin{propn}\label{propn:SegisMon}
  Let $\mathcal{E} \to \mathcal{S}$ denote the Cartesian fibration for
  the functor $\mathcal{S}^\op \to \CatI$ taking $X$ to
  $\PSh(\DFXV)$, and let $\mathcal{E}_{\name{CtsMon}}$ denote the full
  subcategory of $\mathcal{E}$ spanned by the continuous
  $\DFXVop$-monoids for all $X \in \mathcal{S}$. There is an equivalence
\[
\begin{tikzcd}
  \PCS(\DFV) \arrow{dr}{\name{ev}_{{\overline{\mathfrak{e}}}}}
  \arrow{rr}{\sim} & & \mathcal{E}_{\name{CtsMon}} \arrow{dl}\\
  & \mathcal{S}.
\end{tikzcd}
\]
\end{propn}

For the proof we need the following two lemmas:
\begin{lemma}\label{lem:pbadj}
  Suppose
  \[ u : \mathcal{D} \rightleftarrows \mathcal{C} : \eta \]
  is an adjunction with unit transformation $\alpha \colon \id \to
  \eta u$, and let the \icat{} $\mathcal{E}$ be defined by
  the pullback square
  \csquare{\mathcal{E}}{\mathcal{D}^{\Delta^{1}}}{\mathcal{C}}{\mathcal{D}.}{G}{}{\name{ev}_{1}}{\eta}
  \begin{enumerate}[(i)]
  \item The functor $G$ has a left adjoint $F$, which sends $X \to Y$ to
    $(u(Y), X \to Y \xto{\alpha_{Y}} \eta u(Y))$.
  \item The composite functor
    $\mathcal{E} \xto{G} \mathcal{D}^{\Delta^{1}} \xto{\name{ev}_{0}}
    \mathcal{D}$
    has a left adjoint, given by
    \[X \mapsto (u(X), X \xto{\alpha_{X}} \eta u(X)).\] Moreover, this
    is fully faithful.
  \end{enumerate}
\end{lemma}
\begin{proof}
  Since $\name{ev}_{1}$ is a coCartesian fibration, (i) is a special
  case of (the dual of) \cite[Lemma 4.14]{spans}. The functor
  $\name{ev}_{0}$ is given by composition with the inclusion
  $i_{0} \colon \{0\} \hookrightarrow \Delta^{1}$; it is right adjoint
  to the constant diagram functor
  $c \colon \mathcal{D} \to \mathcal{D}^{\Delta^{1}}$, since this can
  be described as left Kan extension along $i_{0}$. The composite
  $\name{ev}_{0}G$ is therefore right adjoint to $F c$, which indeed
  takes $X \in \mathcal{D}$ to $(u(X), X \xto{\alpha_{X}} \eta u(X))$. To see
  that $Fc$ is fully faithful, we must show that the unit map
  $X \to \name{ev}_{0}GFcX$ is an equivalence, which is clear.
\end{proof}

\begin{lemma}\label{lem:monLKEseg}
  Let $X$ be a space, and let $\pi$ denote the projection
  $\DFXVop \to \DFVop$. Then a presheaf
  $F \colon \DFXVop \to \mathcal{S}$ such that $F(\xfe(x)) \simeq *$
  for all $x \in X$ is a continuous $\DFXVop$-monoid \IFF{} the left
  Kan extension $\pi_{!}F$ is a continuous Segal presheaf.
\end{lemma}
\begin{proof}
  We first show that $\pi_!$ gives an identification between
  $\DFXVop$-monoids and Segal presheaves. At the end we will see that
  this restricts to an equivalence of continuous objects.

  Suppose $F \colon \DFXVop \to \mathcal{S}$ is a functor such that
  $F(\xfe(x)) \simeq *$ for all $x \in X$. For $I\in \DF$, let
  $E_{\DF}(I)$ denote the set of edges of $I$.  Since $\pi$ is a
  cocartesian fibration, we can compute the left Kan extension along
  $\pi$ fibrewise; for $\overline{I}$ in $\DFVop$ this gives
  \[\pi_{!}F(\overline{I}) \simeq \colim_{\overline{I}(x_{i})_{i} \in
      (\DFXVop)_{\overline{I}}} F(\overline{I}(x_{i})_{i}) \simeq
    \colim_{(x_{i})_{i} \in X^{\times |E_{\DF}(I)|}} F(\overline{I}(x_{i})_{i}).\]
  In particular, we have that
  \[\pi_{!}F(\xfe) \simeq \colim_{x \in X} * \simeq X.\]
  Moreover, the natural transformation from $F$ to the constant
  functor with value $*$ induces natural morphisms
  $\pi_{!}F(\overline{I}) \to X^{\times |E_{\DF}(I)|}$. Hence, there is a commutative square
  \nolabelcsquare{\pi_!F(\overline I)}{\lim_{\phi\in (\DFelI)^\op}
    \pi_!F(\phi^* \overline I)}{X^{\times |E_{\DF}(I)|}}{\lim_{J \to I
      \in
      (\DFelI)^\op} X^{\times |E_{\DF}(J)|},}
  where the bottom horizontal map is an equivalence, since the set
  $E_{\DF}(I)$ is isomorphic to
  \[\colim_{J \to I \in (\DFelI)^\op}
  E_{\DF}(J).\]

  It follows that the top horizontal morphism in the square is an
  equivalence \IFF{} it gives an equivalence on all fibres over
  $X^{\times |E_{\DF}(I)|}$. The functor $\colim \colon \Fun(T, \mathcal{S})
  \to \mathcal{S}_{/T}$ is an equivalence for any space $T$, with
  inverse given by taking fibres, so the fibre of $\pi_{!}F(\overline{I})
  \simeq \colim_{(x_{i})_{i} \in X^{\times |E_{\DF}(I)|}}
  F(\overline{I}(x_{i})_{i})$ at $(x_{i})_{i}$ is naturally equivalent
  to $F(\overline{I}(x_{i})_{i})$. Since limits commute, it follows
  that the map on fibres at $(x_{i})_{i}$ in the square is the natural
  map
  \[ F(\overline{I}(x_{i})_{i}) \to \lim_{\phi\in (\DFelI)^\op}
    F(\phi^{*}\overline{I}(x_{i})_{i}) \simeq \prod_{v \in V_{\DF}(I)}
    F(v^{*}\overline{I}(x_{i})_{i}),\]
  where the equivalence follows from the assumption that $F(\xfe(x))$
  is the terminal object for all $ x \in X$.

  It follows that $\pi_{!}F$ is a Segal presheaf \IFF{} the above
  squares are all pullbacks, and this in turn holds \IFF{} $F$ is a
  monoid.

By definition, $\pi_{!}F$ is continuous \IFF{} for all
$n$, the functor \[\pi_{!}F(\mathfrak{c}_{n}( \blank)) \colon
\mathcal{V}^{\op} \simeq (\DFVop)_{\mathfrak{c}_{n}} \to
\mathcal{S}_{/X^{\times (n+1)}} \simeq \Fun(X^{\times (n+1)},
\mathcal{S})\]
preserves limits. Limits in functor \icats{} are computed
objectwise, so this holds \IFF{} for all $(x_{i})_{i=1,\ldots,n+1} \in X^{\times
  (n+1)}$ the composite of this functor with evaluation at $(x_{i})_i$
preserves limits. But this composite can be identified with
$F(\mathfrak{c}_{n}(\blank, (x_{i})_i)) \colon \mathcal{V}^{\op} \to
\mathcal{S}$; this preserves limits for all $n$ and $(x_{i})_i \in
X^{\times (n+1)}$ precisely if $F$ is a continuous $\DFXVop$-monoid.
\end{proof}

\begin{proof}[Proof of Proposition~\ref{propn:SegisMon}]
  The functor
  $\name{ev}_{\overline{\mathfrak{e}}} \colon \PSh(\DFV) \to
  \mathcal{S}$
  has a right adjoint, which takes $X \in \mathcal{S}$ to the presheaf
  $\overline{\imath}_{X} \colon \DFVop \to \DF^{\op} \xto{i_{X}}
  \mathcal{S}$
  corresponding to the right fibration $\DFXV \to \DFV$.  Let
  $\mathcal{E}' \to \mathcal{S}$ be the Cartesian fibration for the
  functor $X \mapsto \PSh(\DFV)_{/\overline{\imath}_{X}}$. Then we can
  apply Lemma~\ref{lem:pbadj} to the pullback diagram
  \nolabelcsquare{\mathcal{E}'}{\PSh(\DFV)^{\Delta^{1}}}{\mathcal{S}}{\PSh(\DFV)}
  to conclude that the forgetful functor $\mathcal{E}' \to \PSh(\DFV)$
  has a fully faithful left adjoint, which takes a presheaf $F$ to the
  adjunction unit
  $F \to \overline{\imath}_{F(\overline{\mathfrak{e}})}$. The image of
  $\PSh(\DFV)$ consists of precisely those maps $G \to
  \overline{\imath}_{X}$ such that the restriction $G(\xfe) \to X$ is
  an equivalence.

  By \cite[Corollary 9.8]{laxcolimits}, left Kan extension along
  $\DFXVop \to \DFVop$ gives an equivalence between 
  $\PSh(\DFV)_{/\overline{\imath}_{X}} \simeq \xE'_{X}$ and
  $\PSh(\DFXV) \simeq \mathcal{E}_{X}$. Since this equivalence is
  natural in
  $X$, it gives an equivalence over $\xS$ between $\mathcal{E}$
  and $\mathcal{E}'$. Moreover, the image of $\PSh(\DFV)$ in $\mathcal{E}'$
  corresponds under this equivalence to the full subcategory of
  functors $F \colon \DFXVop \to \mathcal{S}$ such that
  $F(\mathfrak{e}(x)) \simeq *$ for all $x \in X$.

  It thus remains to show that this equivalence restricts further to
  $\PCS(\DFV)\simeq \mathcal{E}_{\name{CtsMon}}$; this is exactly the
  statement of Lemma~\ref{lem:monLKEseg}.
\end{proof}

\begin{proof}[Proof of Theorem~\ref{thm:PCSisAlgLT}]
  Combine Propositions~\ref{propn:SegisMon} and \ref{propn:PCMisAlg}.
\end{proof}

We record the functoriality of $\AlgDFS(\mathcal{V})$ in the symmetric
monoidal \icat{} $\mathcal{V}$:
\begin{propn}\ 
  \begin{enumerate}[(i)]
  \item If $F \colon \mathcal{V} \to \mathcal{W}$ is a lax monoidal functor,
    then $F$ induces a functor \[F_{*} \colon \AlgDFS(\mathcal{V}) \to \AlgDFS(\mathcal{W}).\]
  \item If $F \colon \mathcal{V} \to \mathcal{W}$ is a symmetric
    monoidal left adjoint, with (lax monoidal) right adjoint $G$, then
    there is an adjunction
    \[ F_{*} \colon \AlgDFS(\mathcal{V}) \rightleftarrows \AlgDFS(\mathcal{W}) : G_{*}.\]
  \item If $L \colon \mathcal{V} \to \mathcal{W}$ is a symmetric
    monoidal localization with (lax monoidal) fully faithful right
    adjoint $i$, then the right adjoint
    $i_{*} \colon \AlgDFS(\mathcal{W}) \to \AlgDFS(\mathcal{V})$ is
    fully faithful with image those algebras
    $A \colon \DFX^{\op} \to \mathcal{V}^{\otimes}$ such that for
    every corolla $\tilde{\mathfrak{c}}_{n}$ in $\DFX$ over
    $\mathfrak{c}_{n} \in \DF$, the image $A(\tilde{\mathfrak{c}}_{n})$ lies in $i(\mathcal{W})$.
  \end{enumerate}
\end{propn}
\begin{proof}
  (i) is obvious from the definition of $\AlgDFS(\mathcal{V})$, (ii)
  follows from \cite[Corollary 7.3.2.7]{ha}, and (iii) is immediate
  from (ii).
\end{proof}

We can restate this in terms of continuous Segal presheaves:
\begin{cor}\label{cor:laxmonftr}\ 
  \begin{enumerate}[(i)]
  \item If $F \colon \mathcal{V} \to \mathcal{W}$ is a lax monoidal functor,
    then $F$ induces a functor $F_{*} \colon \PCS(\DFV) \to
    \PCS(\DF^{\mathcal{W}})$.
  \item If $F \colon \mathcal{V} \to \mathcal{W}$ is a symmetric
    monoidal left adjoint, with (lax monoidal) right adjoint $G$, then
    there is an adjunction
    \[ F_{*} \colon \PCS(\DFV) \rightleftarrows  \PCS(\DF^{\mathcal{W}}) : G_{*}.\]
    Moreover, the functor $G_{*}$ can be identified with $F^{*}$.
  \item If $L \colon \mathcal{V} \to \mathcal{W}$ is a symmetric
    monoidal localization with (lax monoidal) fully faithful right
    adjoint $i$, then the right adjoint
    $i_{*} \simeq L^{*} \colon \PCS(\DF^{\mathcal{W}}) \to \PCS(\DFV)$
    is fully faithful with image those $\mathcal{O} \in \PCS(\DFV)$
    where the presheaves $\mathcal{O}(\mathfrak{c}_{n}(
   \blank, x_{1},\ldots,x_{n},y))$ are representable by objects in $\mathcal{W}$.
  \end{enumerate}
\end{cor}
\begin{proof}
  The only claim that requires proof is the identification of $G_{*}$
  with $F^{*}$ in (ii). For $\mathcal{O} \in \PCS(\DF^{\mathcal{W}})$,
  let $A$ be the corresponding object of $\AlgDFS(\mathcal{W})$. We have a natural equivalence
  \[ \Map_{\mathcal{W}^{|\CorDF(I)|}}((w_{c})_c, A(I(x_{i})_i)) \simeq \mathcal{O}(I(w_c,x_{i})_{c,i}).\]
  Then we get 
  \[ F^{*}\mathcal{O}(I(v_c,x_{i})_{c,i}) \simeq \mathcal{O}(I(Fv_{c},x_{i})_{c,i}) \simeq
  \Map_{\mathcal{W}^{|\CorDF(I)|}}(F(v_{c}), A(I(x_{i})_i)) \simeq
  \Map_{\mathcal{V}^{|\CorDF(I)|}}(v_{c}, G A(I(x_{i})_i)),\]
  i.e. $F^{*}\mathcal{O}$ corresponds to $G_{*}A$ under the
  identification of $\PCS(\DFV)$ with $\AlgDFS(\mathcal{V})$.
\end{proof}

\subsection{Enrichment in Presheaves}\label{subsec psh}
If $\mathcal{U}$ is a small symmetric monoidal \icat{}, then by
\cite[Proposition 4.8.1.12]{ha} the $\infty$-category $\PSh(\xU)$ of
presheaves on $\mathcal{U}$ has a unique symmetric monoidal structure
such that the tensor product preserves colimits in each variable and
the Yoneda embedding $y \colon \mathcal{U} \to \PSh(\mathcal{U})$ is
symmetric monoidal; this is the \emph{Day convolution} \cite{Glasman}.
This induces a fully faithful functor
$\DFU \to \DF^{\PSh(\mathcal{U})}$, which we also denote $y$, and our
goal in this section is to show that composition with $y$ induces an
equivalence of \icats{} between $\PCS(\DF^{\PSh(\mathcal{U})})$ and
the \icat{} $\PSeg(\DFU)$ of Segal presheaves on $\DFU$. In
particular, taking $\mathcal{U}$ to be a point, this implies that
continuous Segal presheaves on $\DF^{\mathcal{S}}$ are equivalent to
Segal presheaves on $\DF$.

We first prove that the functor
$y^{*} \colon \PSh(\DF^{\PSh(\mathcal{U})}) \to \PSh(\DFU)$ given by
composition with $y$ has a right adjoint given by right Kan extension,
despite $\DF^{\PSh(\mathcal{U})}$ being a large \icat{}:
\begin{propn}
  The restriction
  \[ \PSh(\DF^{\PSh(\mathcal{U})}) \xto{y^{*}} \PSh(\DFU) \]
  has a fully faithful right adjoint $y_{*}$, given by right Kan extension.
\end{propn}
\begin{proof}
  For a presheaf $F\in \PSh(\DFU)$, the right Kan extension $y_{*}F$
  along $y$ exists if for every $F \in \PSh(\DFU)$ and
  $\overline I = (I, (\phi_{c})_{c \in \CorDF(I)}) \in
  \DF^{\PSh(\mathcal{U})}$, the diagram
  \[ (\DFUop)_{\overline I/} := \DFUop
  \times_{\DF^{\PSh(\mathcal{U}),\op}}
  (\DF^{\PSh(\mathcal{U}),\op})_{\overline I/} \to \DFUop
  \xto{F}\mathcal{S} \]
  has a limit in $\mathcal{S}$. Since $\mathcal{S}$ has all small
  limits, it is enough to show the \icat{} $(\DFUop)_{\overline I/}$
  is essentially small. As a pullback of the left fibration
  $(\DF^{\PSh(\mathcal{U}),\op})_{\overline I/} \to
  (\DF^{\PSh(\mathcal{U}),\op}) $,
  the map $(\DFUop)_{\overline I/}\to \DFUop$ is a left fibration
  whose target is essentially small since  $\xU$ is essentially small. Hence, it
  suffices to show that the fibre over any object
  $\overline J=(J,(\psi_d)_{d\in\CorDF(J)})\in \DFUop$ is small. It
  follows from the definition of $\DF^{\PSh(\mathcal{U})}$ that this fibre
  can be identified with
  \[\xMap_{\DF^{\PSh(\mathcal{U})}}(y(\overline J),\overline I)\simeq
  \xMap_{\DF}(J,I)\times_{\Map_{\xF_{*}^{\op}}(\CorDF^\op(J),\CorDF^\op(I))}\xMap_{\PSh(\xU)_\otimes}((\psi_d)_d,(\phi_c)_c).\]
  It suffices to show that
  $\xMap_{\PSh(\xU)^\otimes}((\psi_d)_d,(\phi_c)_c)$ is small as the
  other two mapping spaces in the pullback are obviously small. The Cartesian fibration $\PSh(\xU)_\otimes\to \xF_{*}^{\op}$ induces a map
  \[\xMap_{\PSh(\xU)_\otimes}((\psi_d)_d,(\phi_c)_c)\to
  \Map_{\xF_{*}^{\op}}(\CorDF^\op(J),\CorDF^\op(I))\]
  whose fibre over $\alpha$ is given by
  \[\prod_{d\in \CorDF(J)}\xMap_{\PSh(\xU)}(\psi_d, \bigotimes_{\alpha(c)=d}\phi_c).\]
  The \icat{} $\PSh(\mathcal{U})$ is locally small by \cite[Example
  5.4.1.8]{ht}, so these mapping spaces are small.
\end{proof}

\begin{thm}\label{thm enr psh}
  Let $\xU^\otimes$ be a small symmetric monoidal \icat{}.  The fully
  faithful functor
  $y_{*} \colon \PSh(\DFU) \hookrightarrow
  \PSh(\DF^{\PSh(\mathcal{U})})$
  restricts to an equivalence
  $\PSeg(\DFU) \isoto \PCS(\DF^{\PSh(\mathcal{U})})$.
\end{thm}

Before we turn to the proof of this theorem, we prove the preliminary
result that the functors $y^{*}$ and $y_{*}$ restrict to functors
between the subcategories of Segal objects. We first consider the easy
case of $y^{*}$:
\begin{lemma}\label{lem y* pres segal}
  The left adjoint functor
  $y^* \colon \PSh(\DF^{\PSh(\mathcal{U})})\to
  \PSh(\DFU)$
  takes continuous Segal presheaves to Segal presheaves.
\end{lemma}
\begin{proof}
  Consider an object $\overline I\in \DFU$ and an inert map $\psi\colon \mathfrak c_n\to I$ in $\DF$. Since the Yoneda embedding is symmetric monoidal, we have $\psi^*y(\overline I)\simeq \mathfrak c_n(yv_c)\simeq y(\psi^*\overline I)$ for some $v_c\in \xV$.
Hence, for every continuous Segal presheaf
  $\xxO\in \PSh_{\xCS}(\DF^{\PSh(\mathcal{U})})$ we have
  \[y^*\xxO(\overline I)=\xxO(y(\overline I))\simeq \lim_{\psi \in
    (\simp^\op_{\mathbb{F},\xel})_{I/}} \xxO(\psi^{*} y(\overline
  I))\simeq \lim_{\psi \in (\simp^\op_{\mathbb{F},\xel})_{ I/}} \xxO(
  y(\psi^{*}\overline I)) =\lim_{\psi \in (\simp^\op_{\mathbb{F},\xel})_{ I/}} y^{*}\xxO(\psi^{*}\overline I),\]
  where the first equivalence follows from the assumption that $\xxO$
  is a continuous Segal presheaf (and so in particular a Segal presheaf).
\end{proof}

Now we consider the opposite direction:
\begin{proposition}\label{prop pres segal equivalences}
  The functor
  $y_{*} \colon\PSh(\DFU) \hookrightarrow
  \PSh(\DF^{\PSh(\mathcal{U})})$
  takes objects in $\PSeg(\DFU) $ to objects in
  $\PCS(\DF^{\PSh(\mathcal{U})})$.
\end{proposition}
The proof of this requires some preliminary lemmas:
\begin{lemma}\label{lem:y*Ipreswccolim}
  For $I \in \DF$, the functor 
  \[y^{*}I(\blank,\ldots,\blank) \colon \PSh(\mathcal{U})^{\times |\CorDF(I)|} \to
  \PSh(\DFU) \]
  preserves weakly contractible colimits in each variable.
\end{lemma}
\begin{proof}
  Pick a vertex $c' \in \CorDF(I)$. We fix the labels $X_{c} \in \PU$
  for all vertices $c \neq c'$, and write
  $\overline{I}_{c'}(\blank) \colon \PU \to \PSh(\DFPU)$ for the
  functor $X_{c'} \mapsto I(X_{c})_{c \in \CorDF(\mathcal{U})}$. Then
  we wish to show that the functor $y^{*}\overline{I}_{c'}$ preserves weakly
  contractible colimits, which is equivalent to the natural map
  \[ \Map_{\PSh(\DFU)}(\mathcal{F}, \colim
  y^{*}\overline{I}_{c'}(\phi)) \to \Map_{\PSh(\DFU)}(\mathcal{F},
  y^{*}\overline{I}_{c'}(\colim\phi)) \]
  being an equivalence of spaces for all $\mathcal{F} \in \PSh(\DFU)$
  and all diagrams $\phi \colon \mathcal{I} \to \PU$ with
  $\mathcal{I}$ weakly contractible. It suffices to check this map is
  an equivalence when $\mathcal{F}$ is (the Yoneda image of) an object
  $\overline{J} = J(v_{x})_{x \in \CorDF(J)}$ in $\DFU$. In this case we have $\Map_{\PSh(\DFU)}(\overline J, \colim
  y^{*}\overline{I}_{c'}(\phi))\simeq  \colim\Map_{\PSh(\DFU)}(\overline{J},
  y^{*}\overline{I}_{c'}(\phi))$ and therefore it suffices to prove that
  the natural map
  \[ \colim\Map_{\PSh(\DFU)}(\overline{J},
  y^{*}\overline{I}_{c'}(\phi)) \to \Map_{\PSh(\DFU)}(\overline{J},
  y^{*}\overline{I}_{c'}(\colim\phi)) \]
  is an equivalence. This map can be identified with the horizontal
  map in the commutative triangle
  \[
  \begin{tikzcd}
    \colim\Map_{\DFPU}(y(\overline{J}),
  \overline{I}_{c'}(\phi)) \arrow{rr} \arrow{dr} & &
  \Map_{\DFPU}(y(\overline{J}),
  \overline{I}_{c'}(\colim\phi)) \arrow{dl}\\
   & \Map_{\DF}(J, I).
  \end{tikzcd}
  \]
  It therefore suffices to show that this map gives an equivalence on
  the fibres over every map $\phi \colon J \to I$ in $\DF$. Since
  $\DFPU \to \DF$ is a Cartesian fibration and colimits of spaces are
  preserves by pullbacks, the map on fibres at $\phi$ is
  \[ \colim \Map_{(\DFPU)_{J}}(y(\overline{J}),
  \phi^{*}\overline{I}_{c'}(\phi)) \to
  \Map_{(\DFPU)_{J}}(y(\overline{J}),
 \phi^* \overline{I}_{c'}(\colim \phi)).\]
  Under the equivalence $(\DFPU)_{J} \simeq \PSh(\mathcal{U})^{\times
    |\CorDF(J)|}$, the object $y(\overline{J})$ corresponds to
  $y(v_{x})_{x \in \CorDF(J)}$ and the object
  $\phi^{*}\overline{I}_{c'}(X_{c'})$ corresponds to $(\bigotimes_{c
    \in \CorDF(\phi)^{-1}(x)} X_{c})_{x \in \CorDF(J)}$, so we have an equivalence
 \[
 \begin{split}
   \Map_{(\DFPU)_{J}}(y(\overline{J}),
   \phi^{*}\overline{I}_{c'}(X_{c'})) & \simeq \prod_{x \in \CorDF(J)}
   \Map_{\PU}\left(y(v_{x}), \bigotimes_{c
     \in \CorDF(\phi)^{-1}(x)} X_{c}\right) \\ & \simeq \prod_{x \in \CorDF(J)}
   \left(\bigotimes_{c  \in \CorDF(\phi)^{-1}(x)} X_{c}\right)(v_{x}).
 \end{split}
 \]
 There are two cases to consider: If $\CorDF(\phi)$ takes $c'$ to the
 base point, then $X_{c'}$ does not appear in this expression, and so
 the functor
 $\Map_{(\DFPU)_{J}}(y(\overline{J}), \phi^*\overline{I}_{c'}(\blank))$ is
 constant. The map
 \[ \colim \Map_{(\DFPU)_{J}}(y(\overline{J}),
 \phi^{*}\overline{I}_{c'}(\phi)) \to
 \Map_{(\DFPU)_{J}}(y(\overline{J}),\phi^* \overline{I}_{c'}(\colim \phi)).\]
 is therefore an equivalence because $\mathcal{I}$ is weakly
 contractible (which implies that the colimit cocone of a constant
 diagram is constant). (Note that in this case we do \emph{not} have
 an equivalence if $\mathcal{I}$ fails to be weakly contractible.) On
 the other hand, if $\CorDF(\phi)$ takes $c'$ to a corolla $x' \in
 \CorDF(I)$ then $\Map_{(\DFPU)_{J}}(y(\overline{J}),
 \phi^*\overline{I}_{c'}(\blank))$ can be identified with
 \[ \left( (\blank) \otimes \bigotimes_{\substack{c \in \CorDF(\phi)^{-1}(x')\\ c
       \neq c'}} X_{c} \right)(v_{x'}) \times \prod_{\substack{x \in \CorDF(J)\\ x \neq x'}}
 \left(\bigotimes_{c  \in \CorDF(\phi)^{-1}(x)} X_{c}\right)(v_{x}).
 \]
 This functor preserves colimits, since the tensor product on $\PU$
 preserves colimits in each variable, colimits in presheaves are
 computed objectwise, and the Cartesian product of spaces preserves
 colimits in each variable.
\end{proof}

\begin{defn}
  For $I \in \DF$ and $c' \in \CorDF(I)$, let $\Sub_{c'}(I)$ denote
  the partially ordered set of inert maps $I' \to I$ such that $c'$
  is not in the image of $I'$, and inert maps between them. We write $I\setminus c'$ for the
  colimit in $\PSh(\DF)$ of $I'$ over $I' \to I$ in $\Sub_{c'}(I)$.
\end{defn}

\begin{lemma}\label{lem empty}
  Suppose $\overline{I}$ is an object of $\DF^{\PSh(\mathcal{U})}$
  lying over $I \in \DF$ such that the corolla $c'$ is labelled by
  $\emptyset$. Let $\overline{I}\setminus c'$ denote the colimit in
  $\PSh(\DF^{\PSh(\mathcal{U})})$ of $\overline{I}'$ over all $I' \to
  I$ in $\Sub_{c'}(I)$, where
  $\overline{I}' \to \overline{I}$ is the Cartesian morphism over this
  map with target $\overline{I}$. Then $y^{*}(\overline{I}\setminus c') \to
  y^{*}\overline{I}$ is an equivalence.
\end{lemma}
\begin{proof}
  It suffices to show that the induced map
  \[ \Map_{\PSh(\DFU)}(\overline{J}, y^{*}(\overline{I}\setminus c'))
  \to \Map_{\PSh(\DFU)}(\overline{J}, y^{*}(\overline{I}))\]
  is an equivalence for all $\overline{J} \in \DFU$. This map can be
  identified with the top horizontal map in the commutative triangle
  \[
  \begin{tikzcd}
    \colim_{I' \to I}\Map_{\DFPU}(y(\overline{J}), \overline{I'}) \arrow{rr} \arrow{dr} & &
  \Map_{\DFPU}(y(\overline{J}), \overline{I}) \arrow{dl}\\
   & \Map_{\DF}(J, I).
  \end{tikzcd}
  \]
  It suffices to show that this map is an equivalence on the fibre
  over every map $\phi \colon J \to I$ in $\DF$. There are two cases
  to consider, according to whether the map $\CorDF(\phi)$ takes the
  corolla $c'$ to a corolla in $\CorDF(J)$ or to the base point. If
  this map sends $c'$ to a corolla
  $x'$ in $\CorDF(J)$, then the fibre of the left-hand map is
  empty (since the map to $\Map_{\DF}(J,I)$ factors through the subset
  $\colim_{I' \to I} \Map_{\DF}(J, I')$). On the other hand, if
  $\overline{J} := J(v_{x})_{x \in \CorDF(J)}$ and $\overline{I} :=
  I(X_{c})_{c \in \CorDF(I)}$ (with $X_{c'} = \emptyset$), then we can
  identify the fibre of the right-hand map with
  \[
  \left( \emptyset \otimes \bigotimes_{\substack{c \in
        \CorDF(\phi)^{-1}(x')\\ c \neq c'}} X_{c} \right)(v_{x'}) \times
  \prod_{\substack{x \in \CorDF(J)\\ x \neq x'}} \left(\bigotimes_{c
      \in \CorDF(\phi)^{-1}(x)} X_{c}\right)(v_{x}),\]
  as in the proof of Lemma~\ref{lem:y*Ipreswccolim}. Since the tensor
  product on $\PU$ preserves colimits, colimits in presheaves are
  computed objectwise, and the Cartesian product of spaces preserves
  colimits, this is also empty.

  It remains to consider the case where $\CorDF(\phi)$ takes $c'$ to
  the base point in $\CorDF(J)$. Observe that, since the maps $\psi
  \colon I' \to  I$ are inert, there is at most one map $\phi' \colon
  J \to I'$ such that $\phi = \psi \circ \phi'$. If the active-inert
  factorization of $\phi$ is $J \xto{\alpha} K \xto{\beta} I$, then
  such a map $\phi'$ exists \IFF{} $\beta$ factors through $I'$ (with
  such a factorization being unique if it exists). We can therefore
  identify the functor
  \[ (I' \to I) \mapsto \Map_{\DFPU}(y(\overline{J}),
  \overline{I'})_{\phi} \] 
  with the left Kan extension of its restriction to
  $\Sub_{c'}(I)_{K/}$. Moreover, we can identify this restriction with
  the constant functor with value $\Map_{(\DFPU)_{J}}(y(\overline{J}),
  \alpha^{*}\overline{K})$. This space is also equivalent to
  $\Map_{(\DFPU)_{J}}(y(\overline{J}), \phi^{*}\overline{I})$, which is
  the fibre of $\Map_{\DFPU}(y(\overline{J}), \overline{I})$ at
  $\phi$. Since the category $\Sub_{c'}(I)_{K/}$ is
  weakly contractible, and fibre products of spaces commute with
  colimits, this means that the map
  \[ \colim_{I' \to I}\Map_{\DFPU}(y(\overline{J}), \overline{I'}) \to
  \Map_{\DFPU}(y(\overline{J}), \overline{I})\]
  is an equivalence on the fibre over $\phi$.  
\end{proof}

\begin{proof}[Proof of Proposition~\ref{prop pres segal equivalences}]
  To prove that $y_{*}$ takes Segal presheaves to continuous Segal
  presheaves it suffices to show that the left adjoint $y^{*}$ takes a
  collection of generating continuous Segal equivalences in
  $\PSh(\DF^{\PSh(\mathcal{U})})$ to Segal equivalences in
  $\PSh(\DFU)$. We consider the generators
  \[ \coprod_{n+1} \overline{\mathfrak{e}} \to (\mathfrak{c}_{n},
  \emptyset), \quad (n \geq 0),\]
  \[ \colim_{\mathcal I} \mathfrak{c}_{n}(\phi) \to
  \mathfrak{c}_{n}(\colim_{\mathcal{I}} \phi), \quad (\phi \colon
  \mathcal{I} \to \PU\text{, $\mathcal{I}$ weakly contractible}),\]
  \[ \overline{I}_{\Seg} \to \overline{I} \quad (\overline{I} \in
  \DFPU).\]
  As special cases of Lemmas~\ref{lem empty} and
  \ref{lem:y*Ipreswccolim}, respectively, we have that $y^{*}$ takes
  the first two types of maps to equivalences. It then remains to show
  that $y^{*}\overline{I}_{\Seg}\to y^{*}\overline{I}$ is a Segal
  equivalence for all $\overline{I}$ in $\DFPU$. Writing $\overline{I} :=
  I(X_{c})_{c \in \CorDF(I)}$, we will prove this in several steps for
  increasingly general labels $X_{c} \in \PU$.

  \emph{Step 1}: First suppose all the labels $X_{c}$ lie in the essential image of
  the Yoneda embedding. Then $\overline{I}_{\Seg}\to \overline{I}$ is
  the image under $y$ of a generating Segal equivalence in
  $\PSh(\DFU)$ and so is obviously taken to a Segal equivalence by
  $y^{*}$.

  \emph{Step 2}: Next consider the case where the labels
  $X_{c}$ are either the initial object $\emptyset$ or in the
  essential image of the Yoneda embedding. We induct on the number of
  corollas labelled by $\emptyset$ (the case where this is zero being
  Step 1). Suppose the corolla $c'$ is labelled by $\emptyset$ and
  consider the commutative square
  \nolabelcsquare{(\overline{I}\setminus
    c')_{\Seg}}{\overline{I}\setminus c'}{\overline{I}_{\Seg}}{\overline{I},} where
  $(\overline{I}\setminus c')_{\Seg} := \colim_{I' \to I \in
    \Sub_{c'}(I)} \overline{I}'_{\Seg}$. This colimit can be
  identified with the colimit of the same diagram as
  $\overline{I}_{\Seg}$, except that the corolla $\mathfrak{c}_{n}(\emptyset)$ corresponding to $c'$ is replaced by $\mathfrak{c}_{n}(\emptyset)\setminus c' \simeq \coprod_{n+1}
  \overline{\mathfrak{e}}$. Thus by Lemma~\ref{lem empty}, the functor
  $y^{*}$ takes the vertical maps in the square to equivalences. By
  the inductive hypothesis, $y^{*}(\overline{I}\setminus c')_{\Seg}
  \to y^{*}(\overline{I}\setminus c')$ is a Segal equivalence, since
  it is a colimit of $\overline{I}'_{\Seg} \to \overline{I}'$ where
  the number of corollas in $\overline{I'}$ labelled by $\emptyset$ is at least one less than
  in $\overline{I}$. By the 2-of-3 property, we conclude that
  $y^{*}\overline{I}_{\Seg}\to  y^{*}\overline{I}$ is also a Segal
  equivalence.

  \emph{Step 3}: Now consider the case where the labels $X_{c}$ are
  all finite coproducts of elements in the essential image of the
  Yoneda embedding. We induct on the size of these coproducts (with
  the case where they are all of size 0 or 1 covered by Step
  2). Suppose the corolla $c'$ is labelled by $v \coprod Y$ with $v
  \in \mathcal{U}$. Then by Lemma~\ref{lem:y*Ipreswccolim} the map
  $y^{*}\overline{I}_{\Seg} \to y^{*}\overline{I}$ can be identified
  with the pushout
  \[ y^{*}\overline{I}_{c'}(v)_{\Seg}
  \amalg_{y^{*}\overline{I}_{c'}(\emptyset)_{\Seg}}
  y^{*}\overline{I}_{c'}(Y)_{\Seg} \to y^{*}\overline{I}_{c'}(v)
  \amalg_{y^{*}\overline{I}_{c'}(\emptyset)}
  y^{*}\overline{I}_{c'}(Y).\]
  This is a pushout of maps that are Segal equivalences by the
  inductive hypothesis, and hence this is also a Segal equivalence.

  \emph{Step 4}: Any presheaf on $\mathcal{U}$ can be written as a
  sifted (hence weakly contractible) colimit of finite coproducts of
  elements in $\mathcal{U}$. By Lemma~\ref{lem:y*Ipreswccolim} the map
  $y^{*}\overline{I}_{\Seg} \to y^{*}\overline{I}$ for general labels $X_{c}$ is
  therefore a colimit of the maps we proved were Segal equivalences in
  Step 3. This map is therefore also a Segal equivalence.
\end{proof}

It is now easy to complete the proof of Theorem~\ref{thm enr psh}:
 \begin{proof}[Proof of Theorem~\ref{thm enr psh}]
   Since $y_{*}$ is fully faithful, Proposition~\ref{prop pres segal
     equivalences} implies that it restricts to a fully faithful
   functor
   $\PSh_{\Seg}(\DFU) \to
   \PCS(\DF^{\PSh(\mathcal{U})})$.
   It therefore only remains to prove that this restricted functor is
   essentially surjective as well. Suppose
   $\xxO \in \PCS(\DF^{\PSh(\mathcal{U})})$; we will show
   that the unit map $u \colon \xxO \to y_{*} y^{*}(\xxO)$ is an
   equivalence. As $y^{*}\xxO$ lies in $\PSeg(\DFU)$ by
   Lemma~\ref{lem y* pres segal}, this is a map of continuous Segal presheaves
   by Proposition~\ref{prop pres segal equivalences}. To show that $u$
   is an equivalence it is therefore enough to show it gives an
   equivalence when evaluated at $\overline{\mathfrak{e}}$ and
   $\mathfrak{c}_{n}(\phi)$ with
   $\phi \in \PSh(\mathcal{U})$. Since $y_{*}$ is fully
   faithful, we have $y^{*}\circ y_{*} \simeq \id$ and so for
   $X \in \DFU$ we get
  \[y_{*}y^{*}(\xxO)(y(X)) \simeq \Map(y(X), y_{*} y^{*}(\xxO))
  \simeq \Map(X, y^{*} y_{*} y^{*}(\xxO)) \simeq \Map(X, y^{*}\xxO) \simeq
  \xxO(y(X)).\]
  In particular, this applies to $\mathfrak{e}$ and
  $(\mathfrak{c}_{n}, y(v))$ with $v \in \mathcal{U}$. But for a
  general $\phi \in \PSh(\mathcal{U})$ there exists a small diagram $f \colon \mathcal{C} \to
  \mathcal{U}$ with colimit $\phi$, and so a Segal
  equivalence $\colim_{\mathcal{C}^{\triangleleft}} (\mathfrak{c}_{n},
 f^\triangleleft ) \to \mathfrak{c}_{n}(\phi)$, which means $u$ is an equivalence
  also when evaluated at $\mathfrak{c}_{n}(\phi)$.
\end{proof}

\begin{corollary}\label{theo P(V)Seg = enriched V operads)}
There is an equivalence of \icats{}
\[  \PSeg(\DF)\simeq\PCS(\DF^{\mathcal{S}}),
\]
where the \icat{} $\mathcal{S}$ of spaces is equipped with the
Cartesian monoidal structure.
\end{corollary}
\begin{proof}
  Applying Theorem~\ref{thm enr psh} to the trivial \icat{} $*$ we get
  an equivalence
  $\PSeg(\DF^{*}) \simeq
  \PCS(\DF^{\mathcal{S}})$, where
  $\PSeg(\DF^{*})$ coincides with $\PSeg(\DF)$
  by definition.
\end{proof}

\subsection{Enrichment in Localizations of Presheaves and
  Presentability}\label{subsec enr loc}
In this subsection we will show that if
$\mathcal{V}$ is the symmetric monoidal
localization of $\PSh(\mathcal{U})$ at a set of
morphisms, then the \icat{} of continuous Segal presheaves on
$\DFV$ is equivalent to a certain
localization of  $\PSeg(\DF^{\mathcal{U}})$. This will also allow us
to show that $\PCS(\DFV)$ is presentable when $\mathcal{V}$ is
presentably symmetric monoidal.

Consider a small symmetric monoidal $\infty$-category $\xU$ and a set
$\mathbb{S}$ of morphisms in $\PSh(\xU)$ which is compatible with the
symmetric monoidal structure on $\PSh(\xU)$, in the sense that the
strongly saturated class generated by $\mathbb{S}$ is closed under
tensor products. We write $\PSh_{\mathbb{S}}(\mathcal{U})$ for the
full subcategory of $\PSh(\mathcal{U})$ spanned by the
$\mathbb{S}$-local objects. Then by \cite[Proposition 2.2.1.9]{ha} the
\icat{} $\PSh_{\mathbb{S}}(\mathcal{U})$ inherits a symmetric monoidal
structure such that the localization
$\PU \to \PSh_{\mathbb{S}}(\mathcal{U})$ is a symmetric monoidal
functor. 

\begin{definition}\label{def S Segal presheaves}
  If $\mathcal{U}$ and $\mathbb{S}$ are as above, then a presheaf in
  $\PSh(\DFPU)$ is a \emph{continuous $\mathbb{S}$-Segal presheaf} if
  it is a continuous Segal presheaf and it is local with respect to
  the maps $\mathfrak{c}_n(s)$ where $s$ is in
  $\mathbb{S}$. Similarly, a presheaf in $\PSh(\DFU)$ is an
  \emph{$\mathbb{S}$-Segal presheaf} if it is a Segal presheaf and is
  local with respect to the maps $y^{*}\mathfrak{c}_n(s)$ where $s$ is in
  $\mathbb{S}$. We write $\PCSS(\DFPU)$ and $\PSSeg(\DFU)$ for the
  full subcategories of $\PSh(\DFPU)$ and $\PSh(\DFU)$ spanned by the
  continous $\mathbb{S}$-Segal presheaves and the $\mathbb{S}$-Segal
  presheaves, respectively. These are by definition the localizations
  of $\PSh(\DFPU)$ and $\PSh(\DFU)$ at the \emph{continuous
    $\mathbb{S}$-Segal equivalences} and the \emph{$\mathbb{S}$-Segal
    equivalences}, these being the morphisms in the strongly saturated
  classes generated (in the sense of Definition~\ref{def:ssatgen}) by
  the continuous Segal equivalences together with
  the maps $\mathfrak{c}_n(s)$ and the Segal equivalences together
  with the maps  $y^{*}\mathfrak{c}_n(s)$, respectively.
\end{definition}

\begin{remark}\label{rmk S-seg repr}
  A Segal presheaf $\mathcal{O} \in \PSeg(\DFU)$ is an
  $\mathbb{S}$-Segal presheaf \IFF{} the presheaf
  $\mathcal{O}(\mathfrak{c}_{n}(\blank)) \colon \mathcal{U}^{\op} \to
  \mathcal{S}$
  is $\mathbb{S}$-local for all $n$. Similarly, a continuous Segal
  presheaf $\mathcal{O} \in \PCS(\DFPU)$ is a continuous
  $\mathbb{S}$-Segal presheaf \IFF{} the presheaf
  $\mathcal{O}(\mathfrak{c}_{n}(\blank,x_{1},\ldots,x_{n},y)) \colon
  \PU^{\op} \to \mathcal{S}$
  is representable, and the representing object in $\PU$ is
  $\mathbb{S}$-local: for $s \colon v \to v'$ in $\mathcal{V}$ we have
  a commutative diagram
  \[
    \begin{tikzcd}
      \Map(\mathfrak{c}_{n}(v'), \mathcal{O}) \arrow{rr}{\mathfrak{c}_{n}(s)^{*}} \arrow{dr} &
      & \Map(\mathfrak{c}_{n}(v), \mathcal{O}) \arrow{dl}\\
      & \Map(\coprod_{n+1}\mathfrak{e}, \mathcal{O}),
    \end{tikzcd}
  \]
  and $\mathfrak{c}_{n}(s)^{*}$ is an equivalence \IFF{} it gives an
  equivalence on fibres over each point of
  $\Map(\coprod_{n+1}\mathfrak{e}, \mathcal{O}) \simeq
  \mathcal{O}(\mathfrak{e})^{n+1}$, and over $(x_{1},\ldots,x_{n},y)$ we
  get the map
  \[ \Map_{\mathcal{V}}(v', \mathcal{O}(x_{1},\ldots,x_{n};y))
    \xto{s^{*}} \Map_{\mathcal{V}}(v,
    \mathcal{O}(x_{1},\ldots,x_{n};y)),\]
  where $\mathcal{O}(x_{1},\ldots,x_{n};y)$ is the object representing
  $\Map(\mathfrak{c}_{n}(\blank), \mathcal{O})_{(x_{1},\ldots,x_{n},y)}$.
\end{remark}

\begin{corollary}\label{cor PSeg=PCS}
  Let $\mathcal{U}$ and $\mathbb{S}$ be as above.
  \begin{enumerate}[(i)]
  \item The equivalence $y^{*} \colon \PCS(\DFPU) \isoto \PSeg(\DFU)$
    restricts to an equivalence
    \[ \PCSS(\DFPU) \isoto \PSSeg(\DFU).\]
  \item Let $L_{\mathbb{S}}$ denote the localization $\PU \to
    \PSh_{\mathbb{S}}(\mathcal{U})$. This is symmetric monoidal, and
    so induces a functor $\DFPU \to
    \DF^{\PSh_{\mathbb{S}}(\mathcal{U})}$, which we also denote
    $L_{\mathbb{S}}$. Then the functor $L_{\mathbb{S}}^{*} \colon
    \PSh(\DF^{\PSh_{\mathbb{S}}(\mathcal{U})}) \to \PSh(\DFPU)$
    given by composition with $L_{\mathbb{S}}^{\op}$ restricts to an
    equivalence
    \[ \PCS(\DF^{\PSh_{\mathbb{S}}(\mathcal{U})}) \isoto
    \PCSS(\DFPU).\]
  \end{enumerate}
  In particular, we have an equivalence
  $\PCS(\DF^{\PSh_{\mathbb{S}}(\mathcal{U})}) \isoto \PSSeg(\DFU)$, which
  describes the continuous Segal presheaves on
  $\DF^{\PSh_{\mathbb{S}}(\mathcal{U})}$ as a localization of
  $\PSh(\DFU)$.
\end{corollary}
\begin{proof}
  Part (i) is obvious, since $\PCSS(\DFPU)$ and $\PSSeg(\DFU)$ are,
  respectively, the localization of $\PCS(\DFPU)$ at a certain set of
  maps and the localization of $\PSeg(\DFU)$ at the image of this
  set under the equivalence $y^{*}$ of Theorem~\ref{thm enr psh}.

  To prove (ii), first recall that
  the functor $L_{\mathbb{S}}^{*}$ restricts to a functor
  \[ \PCS(\DF^{\PSh_{\mathbb{S}}(\mathcal{U})}) \to \PCS(\DFPU)\]
  by Lemma~\ref{lem DFV functor}. Now observe that by
  Corollary~\ref{cor:laxmonftr}, this functor is fully faithful, and
  by Remark~\ref{rmk S-seg repr} its image is precisely
  $\PCSS(\DFPU)$.
\end{proof}

We now consider a key special case of this corollary:
\begin{defn}\label{defn:kappacts}
  Let $\mathcal{U}$ be a small symmetric monoidal \icat{}, and suppose
  $\kappa$ is a regular cardinal such that $\mathcal{U}$ has all
  $\kappa$-small colimits and the tensor product preserves these in each variable separately. A
  Segal presheaf $\mathcal{O} \in \PSeg(\DFU)$ is
  \emph{$\kappa$-continuous} if for every $n$,
  \begin{enumerate}[(1)]
  \item  the functor
    $\mathcal{O}(\mathfrak{c}_{n}(\blank)) \colon \mathcal{U}^{\op} \to
    \mathcal{S}$ preserves $\kappa$-small weakly contractible limits,
  \item $\mathcal{O}(\mathfrak{c}_{n}(\emptyset)) \to
    \prod_{n+1}\mathcal{O}(\mathfrak{e})$ is an equivalence.
  \end{enumerate}
  We write $\PkSeg(\DFU)$ for the full subcategory of $\PSh(\DFU)$
  spanned by the $\kappa$-continuous Segal presheaves.
\end{defn}

\begin{propn}\label{propn:kctsS}
  Let $\mathcal{U}$ be as in Definition~\ref{defn:kappacts}. Take
  $\mathbb{S}_{\kappa}$ to be the class of morphisms of the form
  $\colim y(\phi) \to y(\colim \phi)$ where $\phi$ is a $\kappa$-small
  diagram in $\mathcal{U}$ and $y$ is the Yoneda embedding; we take
  $\mathbb{S}'_{\kappa} \subseteq \mathbb{S}_{\kappa}$ to be a subset
  corresponding to a set of representatives for equivalence classes of
  pushouts and $\kappa$-small coproducts in $\mathcal{U}$. Then the
  following are equivalent for $\mathcal{O} \in \PSeg(\DFU)$:
  \begin{enumerate}[(i)]
  \item $\mathcal{O}$ is a $\kappa$-continuous Segal presheaf,
  \item $\mathcal{O}$ is an $\mathbb{S}_{\kappa}$-Segal presheaf,
  \item $\mathcal{O}$ is an $\mathbb{S}'_{\kappa}$-Segal presheaf.
  \end{enumerate}
\end{propn}
\begin{proof}
  By definition, $\mathcal{O}$ is an $\mathbb{S}_{\kappa}$-Segal
  presheaf if for every $\kappa$-small diagram $\phi \colon
  \mathcal{I} \to \mathcal{U}$, the morphism
  \[\Map(y^{*}\mathfrak{c}_{n}(y(\colim \phi)), \mathcal{O}) \to \Map(y^{*}\mathfrak{c}_{n}(\colim
   y(\phi)), \mathcal{O}) \]
  is an equivalence. This holds \IFF{} it holds for all $\kappa$-small
  weakly contractible diagrams and the initial object, since these
  generate the $\kappa$-small colimits. For $\kappa$-small weakly
  contractible colimits, we have $y^{*}\mathfrak{c}_{n}(\colim
    y(\phi))) \simeq \colim y^{*}\mathfrak{c}_{n}(y(\phi))$ by
    Lemma~\ref{lem:y*Ipreswccolim}, so we can identify the morphism
    with $\mathcal{O}(\mathfrak{c}_{n}(\colim \phi)) \to \lim
    \mathcal{O}(\mathfrak{c}_{n}(\phi))$, corresponding to condition
    (1) in Definition~\ref{defn:kappacts}. For the initial object, we
    instead have $y^{*}\mathfrak{c}_{n}(\emptyset) \simeq
    \coprod_{n+1}\mathfrak{e}$ by Lemma~\ref{lem empty}, corresponding
    to condition (2). This proves that (i) and (ii) are
    equivalent. The equivalence of (ii) and (iii) is immediate, since
    $\kappa$-small colimits are generated by pushouts and
    $\kappa$-small coproducts.
\end{proof}

\begin{cor}\label{cor:kappaseg}
  Let $\mathcal{U}$ be a small symmetric monoidal \icat{}, and suppose
  $\kappa$ is a regular cardinal such that $\mathcal{U}$ has all
  $\kappa$-small colimits and the tensor product preserves these. Then
  the Yoneda embedding induces an equivalence
  \[ \PCS(\DF^{\Ind_{\kappa}\mathcal{U}}) \isoto \PkSeg(\DFU).\]
\end{cor}
\begin{proof}
  Apply Corollary~\ref{cor PSeg=PCS} with $\mathbb{S}$ being the set
  $\mathbb{S}'_{\kappa}$ from Proposition~\ref{propn:kctsS}.
\end{proof}

\begin{remark}\label{rmk PCS Ind pres}
  Since $\mathbb{S}'_{\kappa}$ is a \emph{small} set, the \icat{}
  $\PkSeg(\DFU)$ is the full subcategory of $\PSh(\DFU)$
  of presheaves that are local with respect to a small set of
  maps. It follows that $\PkSeg(\DFU)$ is an accessible localization
  of $\PSh(\DFU)$. In particular, the \icat{}
  $\PCS(\DF^{\Ind_{\kappa}\mathcal{U}}) \simeq \PkSeg(\DFU)$ is
  presentable. We can use this result to see that $\PCS(\DFV)$ is
  presentable for any presentably symmetric monoidal \icat{} $\mathcal{V}$:
\end{remark}

\begin{cor}\label{cor PCSDFV pres}
  Suppose $\mathcal{V}$ is a presentably symmetric monoidal
  \icat{}. Then $\PCS(\DFV)$ is a presentable \icat{}.
\end{cor}

This follows from Remark~\ref{rmk PCS Ind pres} and the following
technical observation:

\begin{propn}\label{propn Vkappa symmon}
  Suppose $\mathcal{V}^\otimes$ is a presentably symmetric monoidal
  \icat{}. Then there exists a regular cardinal $\kappa$ such that
  $\mathcal{V}$ is $\kappa$-presentable and the full subcategory
  $\mathcal{V}^{\kappa}$ of $\kappa$-compact objects is closed under
  the tensor product.
\end{propn}

To prove Proposition~\ref{propn Vkappa symmon} it is convenient to
prove a slightly more general result, for accessibly symmetric monoidal
\icats{} in the following sense:
\begin{definition}\label{def accessible monoidal}
  Given a regular cardinal $\kappa$, we say a symmetric monoidal
  $\infty$-category $\xV$ is \emph{$\kappa$-accessibly symmetric
    monoidal} if the underlying \icat{} $\mathcal{V}$ is
  $\kappa$-accessible and the tensor product preserves
  $\kappa$-filtered colimits separately in each variable. We say
  $\mathcal{V}$ is \emph{accessibly symmetric monoidal} if it is
  $\kappa$-accessibly symmetric monoidal for some cardinal $\kappa$.
\end{definition}

\begin{lemma}\label{lem compact obj}
  Let $\kappa$ and $\kappa'$ be two regular cardinals such that
  $\kappa\ll\kappa'$ (in the sense of \cite[Definition 5.4.2.8]{ht}). If $\xcc$ is a $\kappa$-accesible
  $\infty$-category, then an object in $\xcc$ is $\kappa'$-compact if
  and only if it is a $\kappa'$-small $\kappa$-filtered colimit of
  $\kappa$-compact objects.
\end{lemma}

\begin{proof}
  Let $\xcc^{\kappa}$ and $\xcc^{\kappa'}$ denote the full
  subcategories of $\xcc$ spanned by the $\kappa$-compact and
  $\kappa'$-compact objects, respectively. If we write $\xcc'$ for the
  full subcategory of $\xcc$ spanned by the colimits of all
  $\kappa'$-small $\kappa$-filtered diagrams in $\xcc^{\kappa}$, then
  $\xcc'$ is essentially small. This is because $\xcc$ is locally
  small and the collection of all equivalence classes of
  $\kappa'$-small $\kappa$-filtered diagrams is bounded. Since
  $\xcc^{\kappa'}$ is closed under $\kappa'$-small colimits by
  \cite[Corollary 5.3.4.15]{ht}, the $\infty$-category $\xcc'$ is a
  full subcategory of $\xcc^{\kappa'}$.
  
  The proof of \cite[Proposition 5.4.2.11]{ht} shows that $\xcc'$
  generates $\xcc$ under small $\kappa'$-filtered colimits. According
  to \cite[Lemma 5.4.2.4]{ht}, the $\infty$-category
  $\xcc^{\kappa'}\subseteq\xcc=\xInd_{\kappa'}(\xcc')$ is given by the
  idempotent completion of $\xcc'$. Since $\xcc'$ is already
  idempotent complete by \cite[Proposition 4.4.5.15]{ht}, the
  $\infty$-categories $\xcc'$ and $\xcc^{\kappa'}$ coincide.
\end{proof}

\begin{definition}\label{def Vkappa}
  Let $\xV^\otimes \to \xF_{*}$ be a symmetric monoidal
  $\infty$-category.  For a cardinal $\kappa$, let $\xV^\kappa$ denote
  the full subcategory of $\xV$ spanned by $\kappa$-compact objects in
  $\xV$. We write $\xV^{\kappa,\otimes}$ for the full subcategory of
  $\xV^\otimes$ spanned by objects lying in
  $(\xV^\kappa)^n\subseteq (\xV)^n\simeq\xV_{\langle n\rangle}$ for
  all $n\geq 1$.
\end{definition}
\begin{remark}
  We emphasize that $\xV^{\kappa,\otimes}$ is in general \emph{not}
  the full subcategory of the $\infty$-category $\xV^\otimes$ spanned
  by $\kappa$-compact objects. Moreover, it is in general also not a
  symmetric monoidal \icat{}, though it is an \iopd{}.
\end{remark}

\begin{proposition}\label{prop kappa}
  Suppose $\xV$ is a $\kappa$-accessibly symmetric monoidal
  $\infty$-category. Then there exists a regular cardinal $\kappa'$
  such that the full subcategory $\xV^{\kappa'}$ is a symmetric
  monoidal $\infty$-category.
\end{proposition}
\begin{proof}
  Since $\xV$ is $\kappa$-accessibly symmetric monoidal, the
  $\infty$-category $\xV^{\kappa,\otimes}$ is essentially small and we
  can choose a cardinal $\kappa'$ such that $\kappa'\gg\kappa$,
  $\bbone\in\xV^{\kappa'} $ and $v\otimes w\in \xV^{\kappa'}$ for any
  pair of objects $v,w\in \xV^{\kappa}$.
  
  If $w\in \xV^{\kappa'}$, then, by Lemma~\ref{lem compact obj}, there
  exists a $\kappa'$-small $\kappa$-filtered colimit diagram
  $f\colon \xI^\triangleright\to \xV$ such that $f(i)\in \xV^{\kappa}$
  and $w=\colim_{i\in I} f(i)$. Since $\mathcal{V}$ is
  $\kappa$-accessibly symmetric monoidal, the tensor product preserves
  $\kappa$-filtered colimits in each variable. Therefore, for every
  $v\in \xV^{\kappa}$, we have that
  $v\otimes (\colim_{i\in I} f(i))\simeq \colim_{i\in I} (v\otimes
  f(i))$
  is a $\kappa'$-small colimit of $\kappa'$-compact objects, and is
  thus $\kappa'$-compact by \cite[Corollary 5.3.4.15]{ht}.  Thus we
  see that $v\otimes w\in \xV^{\kappa'}$, if $v\in \xV^{\kappa'}$ and
  $w\in\xV^{\kappa}$. Applying the same argument to the other variable, we deduce that
  $\xV^{\kappa'}$ is closed under tensor products, and so it is a
  symmetric monoidal subcategory of $\mathcal{V}$.
\end{proof}

Since every presentably symmetric monoidal \icat{} is by definition
$\kappa$-accessibly symmetric monoidal for some $\kappa$,
Proposition~\ref{propn Vkappa symmon} follows as a special case.

\subsection{Inner Anodyne Maps}\label{subsec inner anodyne}
Recall that a morphism in $\PSh(\simp)$ is \emph{inner anodyne} if it
is in the weakly saturated class generated by the inner horn
inclusions $\Lambda^{n}_{i} \to \Delta^{n}$ (viewed as discrete
simplicial spaces). In this subsection we will consider an analogous
class of inner anodyne maps in $\PSh(\DFV)$ and show that these are
Segal equivalences.

\begin{definition}
  Let $\xcc$ be a cocomplete $\infty$-category.
  \begin{enumerate}
    \item  We say a class of morphisms in $\xcc$ is \emph{weakly saturated} if it is closed under pushouts, transfinite compositions and retracts. 
    \item Given a class $\mathbb{S}$ of morphisms in $\xcc$, we write
      $\langle \mathbb{S}\rangle$ for the smallest weakly saturated
      class of morphisms in $\xcc$ containing $\mathbb{S}$ and we say
      that $\langle \mathbb{S}\rangle$ is \emph{generated} by
      $\mathbb{S}$.
  \end{enumerate}
\end{definition}

\begin{remark}
  In the model category literature, the 1-categorical analogue of a
  weakly saturated class is usually just called a \emph{saturated}
  class. Following \cite{ht}, we use the term \emph{weakly} saturated
  to avoid confusion with the \emph{strongly} saturated classes
  relevant to accessible localizations of presentable \icat{} (\cf{}
  Definition~\ref{defn:ssat}). It is then clear that from the definition that strongly saturated classes are also weakly saturated.
\end{remark}

\begin{definition}\label{def partical/lambda x}
Let $p$ denote the Cartesian
  fibration $\DF \to \simp$ and let $I$ be an object of $\DF$ with $p(I) =
  [n]$. For $1\leq k\leq n-1$, we define $\Lambda^{n}_k I$ by the pullback
  \nolabelcsquare{\Lambda^{n}_kI}{
    I}{p^{*}\Lambda^{n}_k}{p^*(\Delta^n).}
  If $\mathcal{V}$ is a symmetric monoidal \icat{} and $\overline{I}$
  is an object of $\DFV$ over $I$, we similarly define
  $\Lambda^{n}_k\overline{I}$ by the pullback
  \nolabelcsquare{\Lambda^{n}_k\overline{I}}{
    \overline{I}}{(pq)^{*}\Lambda^{n}_k}{(pq)^*(\Delta^n),}
  where $q$ denotes the Cartesian fibration $\DFV \to \DF$.
  We refer to the inclusions
  $\Lambda^{n}_{k}\overline{I} \hookrightarrow \overline{I}$ as
  \emph{inner horn inclusions} and we say a map in $\PSh(\DFV)$ is \emph{inner anodyne} if it lies in the weakly saturated class generated by the
  inner horn inclusions (for all $\overline{I}$ in $\DFV$).
\end{definition}

\begin{remark}
  In $\PSh(\simp)$, a presheaf is a Segal space \IFF{} it has the
  right lifting property with respect to the inner anodyne maps. The
  analogous statement for Segal presheaves on $\DFV$
  is not true for the inner anodyne maps as defined here,
  since they clearly do not generate the maps
  \[ \coprod_{i \in \mathbf{m}}\mathfrak c_{n_i}(v_i)
  \to ([1], \mathbf{n} \to \mathbf{m})(v_{1},\ldots,v_{m}),\]
  \[ \coprod_{i \in \mathbf{m}}\mathfrak e 
  \to ([0], \mathbf{m}).\]
  It should be possible to enlarge the class of inner anodyne maps to
  include ``inner horns'' related also to these maps, but we will not
  consider this issue here.
\end{remark}

Our goal is now to prove the following:
\begin{propn}\label{propn:iaDFV}
  The inner anodyne maps in $\PSh(\DFV)$ are Segal equivalences.
\end{propn}

We will prove this at the end of this subsection after making some
formal observations. The starting point for the proof is the following
reinterpretation of the proof of \cite[Lemma
3.5]{JoyalTierney}:
\begin{propn}[Joyal--Tierney]\label{propn:JoyalTierney}
  Every inner horn inclusion
  $\Lambda^{n}_{k} \hookrightarrow \Delta^{n}$ can be constructed from the
  spine inclusions $\Delta^{i}_{\Seg} \to \Delta^{i}$ ($i \leq n$) by
  a finite sequence of compositions, pushouts, and right cancellations.
\end{propn}

\begin{proposition}\label{prop Segal equ = inner horn}
  In $\PSh(\simp)$ inner horn inclusions generate the strongly saturated class of Segal equivalences.
\end{proposition}
\begin{proof}
  Inner horn inclusions are Segal equivalences as they are generated
  by the spine inclusions according to the previous
  proposition. Hence, it suffices to show that the spine inclusions, which
  generate the Segal equivalences, lie in the strongly saturated class
  generated by the inner horn inclusions. This follows from
  \cite[Proposition 2.13]{Joyal_Quasicat} which says that spine
  inclusions are inner anodyne, and so even lie in the weakly
  saturated class generated by the inner horn inclusions.
\end{proof}

Since pullbacks in an $\infty$-topos preserve colimits, we get the
following:
\begin{propn}\label{propn:pbwsat}
  Suppose $\mathbb{S}$ is a set of morphisms in an $\infty$-topos
  $\mathcal{X}$. If \csquare{X}{Y}{X'}{Y'}{f}{\xi}{\eta}{f'} is a
  pullback square, let $\eta^*\mathbb{S}$ denotes the class of
  morphisms $\eta^{*}A \to \eta^{*}B$ given by pullbacks
  \[
    \begin{tikzcd}
      \eta^{*}A \arrow{r} \arrow{d} & \eta^{*}B \arrow{r} \arrow{d} &
      Y \arrow{d}{\eta}\\
      A \arrow{r}{s} & B \arrow{r} & Y',
    \end{tikzcd}
  \]
  for all $s \in \mathbb{S}$ and all maps $B \to Y'$.
  \begin{enumerate}[(i)]
  \item If $f'$ is a (possibly transfinite) composite of morphisms in
    $\mathbb{S}$, then $f$ is a (transfinite) composite of
    morphisms in $\eta^{*}\mathbb{S}$.
  \item If $f'$ is a cobase change of a morphism in $\mathbb{S}$,
    then $f$ is a cobase change of a morphism in $\eta^*\mathbb{S}$.
  \item If $f'$ is a retract of a morphism in $\mathbb{S}$, then
    $f$ is a retract of a morphism in $\eta^{*}\mathbb{S}$.
  \item If $f'$ is obtained from $\mathbb{S}$ by right cancellation
    (\ie{} there exists $g'$ in $\mathbb{S}$ such that $f'g'$ is in
    $\mathbb{S}$) then $f$ is obtained from $\eta^{*}\mathbb{S}$ by
    right cancellation.
  \end{enumerate}
\end{propn}

Combining this observation with the small object argument, we have:
\begin{cor}\label{cor:wsatpb}
  Suppose $\mathbb{S}$ is a set of morphisms in an $\infty$-topos
  $\mathcal{X}$. If \csquare{X}{Y}{X'}{Y'}{f}{\xi}{\eta}{f'} is a
  pullback square and $f'$ lies in the weakly saturated class
  $\angled{\mathbb{S}}$ generated by $\mathbb{S}$, then $f$ lies in
  the weakly saturated class $\angled{\eta^{*}\mathbb{S}}$. \qed
\end{cor}

\begin{lemma}\label{lem cart lift}
  Let $p\colon \xE\to \xcc$ be a Cartesian fibration and let
  $p^*\colon \PSh(\xcc)\to \PSh(\xE)$ denote the functor given by
  composition with $p$. If
  $\overline{s}\to \overline{t}$ is a $p$-Cartesian morphism lying over
  $s\to t$ in $\xcc$, then there is a pullback square in $\PSh(\xE)$
\[
\begin{tikzcd}
    \overline{s} \ar[r] \ar[d]    &   \overline{t} \ar[d]  \\
    p^* s \ar[r]             &   p^* t   ,  
\end{tikzcd}
\]
  where the vertical maps are the adjunction units $\overline{s}\to
        p^*p_!\overline{s}\simeq  p^* s$ and $\overline{t}\to
        p^*p_!\overline{t}\simeq  p^* t$, respectively.
\end{lemma}
\begin{proof}
  We have to show that for every presheaf $F\in \PSh(\xE)$, the commutative square
\[
\begin{tikzcd}
    \xMap_{\PSh(\xE)}(F,\overline{s}) \ar[r] \ar[d]    &   \xMap_{\PSh(\xE)}(F,\overline{t}) \ar[d]  \\
    \xMap_{\PSh(\xE)}(F,p^* s) \ar[r]     &
                \xMap_{\PSh(\xE)}(F,p^* t)  
\end{tikzcd}
\]
  is a pullback square of spaces. Since every object in the presheaf category $\PSh(\xE)$ is given by a colimit of representable objects we can assume without loss of generality that $F$ is represented by an object $\overline x\in\xE$ lying over $x\in \xcc$. In this case, the adjunction $(p_!,p^*)$ implies that the above square is of the form 
\[
\begin{tikzcd}
    \xMap_{\xE}(\overline{x},\overline{s}) \ar[r] \ar[d]    &    \xMap_{\xE}(\overline{x},\overline{t}) \ar[d]  \\
    \xMap_{\xcc}(x,s) \ar[r]     &    \xMap_{\xcc}(x,t)  ,
\end{tikzcd}
\]
  which is a pullback of spaces because the morphism $\overline{s}\to \overline{t}$ was $p$-Cartesian.
\end{proof}

\begin{defn}\label{def simple}
  Let $p\colon \xE\to \xcc$ be a Cartesian fibration between small
  $\infty$-categories and let $p^*\colon \PSh(\xcc)\to \PSh(\xE)$
  denote the functor on presheaves given by composition with $p$.  For
  $F \in \PSh(\xE)$ and $G \in \PSh(\xcc)$, we say a morphism
  $F \to p^*(G)$ is \emph{simple} if for every map $\sigma\colon c \to G$ where $c$
  is representable, in the pullback
  \nolabelcsquare{X}{F}{p^*(c)}{p^*(G)} the presheaf $X$ is
  representable, and the adjoint map $p_{!}X \to c$ is an equivalence
  (i.e.\ $X$ is representable by an object whose image under $p$ is
  $c$). In this case we often write $F_\sigma$ for the presheaf $X$.
\end{defn}

\begin{remark}\label{rem simple}
  The previous definition and Lemma~\ref{lem cart lift} immediately
  imply that for any $\overline t\in \xE$ the counit map
  $\overline t\to p^*t$ is simple and a pullback of a simple map is
  simple.
\end{remark}

\begin{notation}
  Let $\mathbb{S}$ be a class of morphisms $f\colon K\to L$ in
  $\PSh(\xcc)$ such that $L$ is representable. We define
  $p^*\mathbb{S}:=\{p^*s, s\in \mathbb{S}\}$ and we write
  $\mathbb{S}_{\mathcal{E}}$ for the class of morphisms $Y\to X$ in
  $\PSh(\xE)$ which fit in a pullback square
  \nolabelcsquare{Y}{X}{p^*K}{p^*L,} where $X$ is representable, the
  adjoint map $p(X)\simeq p_!X\to L $ is an equivalence and the bottom
  horizontal map lies in $p^*\mathbb{S}$.
\end{notation}

\begin{lemma}\label{lem ws}
  Let $p^*\colon \PSh(\xcc)\to \PSh(\xE)$ denote the functor induced
  by composition with a Cartesian fibration $p\colon \xE\to \xcc$
  between small $\infty$-categories and let $\mathbb{S}$ be a class of
  morphisms in $\PSh(\xcc)$ whose codomains are representable. Suppose
  there is a simple map $f\colon X\to p^*L$ and a map
  $K\to L\in \langle\mathbb{S}\rangle$, then the natural map
  $Y := p^{*}K \times_{p^{*}L}X \to X$ lies in
  $\langle\mathbb{S}_{\mathcal{E}}\rangle$.
\end{lemma}
\begin{proof}
  Since the functor
  $p^*$ is left adjoint to the functor $p_*$ given by right Kan
  extension, it preserves colimits. Therefore, by the small object
  argument it takes every element
  in $\langle\mathbb{S}\rangle$ to $\langle p^*\mathbb{S}\rangle$. In
  particular, the bottom horizontal morphism of the pullback square
  \nolabelcsquare{Y}{X}{p^*K}{p^*L} lies in
  $\langle p^*\mathbb{S}\rangle$.  Then Proposition~\ref{cor:wsatpb}
  implies that $Y\to X$ lies in $\langle f^*p^*\mathbb{S}\rangle$ and
  each generator $f^*p^*s$, $s\in \mathbb{S}$, is given by two
  pullback squares
  \[
  \begin{tikzcd}
  f^{*}p^{*}A \arrow{r}{f^*p^*s} \arrow{d} & f^{*}p^{*}B \arrow{r} \arrow{d} &
  X \arrow{d}{f}\\
  p^{*}A \arrow{r}{p^*s} & p^{*}B \arrow{r} & p^{*}L.
  \end{tikzcd}
  \]
  Therefore, to prove the lemma, it suffices to show that each of
  these maps $f^*p^*s$ lie in $\mathbb{S}_{\mathcal{E}}$. Since
  $s\in\mathbb{S}$ and its codomain $B$ is representable, the
  assumption that $f\colon X\to p^*L$ is simple implies that $f^*p^*B$
  is representable and $p(f^*p^*B)\simeq p^*B$. Hence, the map
  $f^*p^*s$ lies in $\mathbb{S}_{\mathcal{E}}$ by definition.
\end{proof}

\begin{proof}[Proof of Proposition~\ref{propn:iaDFV}]
  It's enough to show that every inner horn inclusion
  $\Lambda^{n}_{k}\overline{I}\to \overline{I}$ in $\DFV$ is a Segal
  equivalence: this implies that the class of inner anodyne maps,
  which is the weakly saturated class generated by these maps, must
  be contained in the strongly
  saturated class of Segal equivalences.  Given an inner horn
  inclusion $j \colon \Lambda^{n}_{k}\overline{I} \to \overline{I}$, by
  definition it comes with a pullback diagram
  \csquare{\Lambda^{n}_{k}\overline{I}}{\overline{I}}{(pq)^{*}\Lambda^{n}_{k}}{(pq)^{*}\Delta^{n}}{\overline{\jmath}}{}{\eta}{(pq)^{*}j}
  Let $\mathbb{S}$ denote the set of spine inclusions
  $\Delta^m_\Seg\to \Delta^m$, then
  Propositions~\ref{propn:JoyalTierney}
  and \ref{propn:pbwsat} imply that $\overline{\jmath}$ is obtained
  from $\eta^{*}\mathbb{S}$ by a finite sequence of pushouts,
  compositions, and right cancellations. It follows that
  $\overline{\jmath}$ is contained in the strongly saturated class
  generated by $\eta^{*}\mathbb{S}$. 

  Since the right vertical map $\eta \colon \overline I\to (pq)^{*}\Delta^{n}$ is
  simple by Remark~\ref{rem simple}, as in the proof of Lemma~\ref{lem
    ws} we see that the maps in $\eta^{*}\mathbb{S}$ are among the
  generating Segal equivalences of
  Proposition~\ref{propn:SegDFVcond}(3). Hence $\overline{\jmath}$
  must be contained in the strongly saturated class of Segal
  equivalences, as required.
\end{proof}

By applying Lemma~\ref{lem ws} to the set
$\mathbb{S}=\{\Lambda^n_k\to \Delta^n, n\geq 0, 0<k<n \}$ of inner
horn inclusions in $\PSh(\simp)$ we also obtain the following:
\begin{corollary}\label{cor pb of inner anodyne}
  Let $p\colon \DF \to \simp$ and $q\colon \DFV \to \DF$ denote the
  Cartesian fibrations given by the obvious projections.  If
  $f\colon X\to (pg)^*L$ is a simple map in $\PSh(\DFV)$ and $K\to L$
  is an inner anodyne map in $\PSh(\simp)$, then the natural map
  $Y := (pq)^{*}K \times_{(pq)^{*}L}X \to X$ in $\PSh(\DFV)$ is inner
  anodyne (and so in particular a Segal equivalence, by
  Proposition~\ref{propn:iaDFV}). \qed
\end{corollary}

\subsection{Tensoring with Segal Spaces}\label{sec tensor product}
Our goal in this section is to prove that for $\mathcal{V}$ a
presentably symmetric monoidal \icat{}, the \icat{} $\PCS(\DFV)$ is
a module over the symmetric monoidal \icat{} $\PSh_{\Seg}(\simp)$ of
Segal spaces. This will be useful in the next section as it allows us
to reduce a number of proofs to the case of Segal spaces.

\begin{defn}\label{def tensor}
  Let $\mathcal{U}$ be a small symmetric monoidal \icat{}, and let
  $p \colon \DFU \to \simp$ denote the composite of the two
  Cartesian fibrations $\DFU \to \DF$ and $\DF \to
  \simp$. Precomposition with $p^\op$ gives a functor
  $p^{*}\colon \PSh(\simp) \to \PSh(\DFU)$. This functor preserves products
  (since it is a right adjoint), and so can be viewed as a
  symmetric monoidal functor, with both \icats{} equipped with their
  Cartesian symmetric monoidal structures. In other words, $p^{*}$ is a
  morphism of commutative algebra objects in $\CatI$ (or even in
  $\PrL$) and so induces on $\PSh(\DFU)$ the structure of a module
  over $\PSh(\simp)$ (e.g.\ by \cite[Corollary 3.4.1.7]{ha}), given by the functor
  $\PSh(\DFU) \times \PSh(\simp) \to \PSh(\DFU)$ that takes
  $(\overline I, K)$ to $\overline I \times p^*(K)$. Since the
  Cartesian product in $\mathcal{S}$ preserves colimits in each
  variable, as does $p^{*}$, the functor $\blank \times p^{*}(\blank)$
  preserves colimits in each variable.
\end{defn}

Our main goal in this section is to prove the following result:
\begin{theorem}\label{theo tensor}
  For $\mathcal{U}$ a small symmetric monoidal \icat{}, the $\PSh(\simp)$-module structure on
  $\PSh(\DFU)$ induces a $\PSeg(\simp)$-module structure on
  $\PSeg(\DFU)$ where the tensoring
  \[\otimes \colon \PSeg(\DFU) \times \PSeg(\simp) \to \PSeg(\DFU)\] is
  given by
  $\mathcal{O} \otimes \mathcal{X} = L(\mathcal{O} \times
  p^{*}\mathcal{X})$,
  where $L$ denotes the localization $\PSh(\DFU) \to \PSeg(\DFU)$. In
  particular, the tensoring preserves colimits in each variable.
\end{theorem}

By \cite[Proposition 2.2.1.9]{ha} to prove Theorem~\ref{theo tensor}
it suffices to show that the module structure on $\PSh(\DFU)$ is
compatible with the Segal equivalences in the following sense:
\begin{proposition}\label{propn Seg tensor}
  Suppose $f \colon F \to G$ is a Segal
  equivalence in $\PSh(\DFU)$ and
  $g \colon K \to L$ is a Segal equivalence in
  $\PSh(\simp)$. Then
  $f \times p^*(g) \colon F \times
  p^*(K) \to G \times p^{*}G$ is a Segal
  equivalence in $\PSh(\DFU)$.
\end{proposition}

As already mentioned in Definition~\ref{def tensor} the tensor functor
$\blank \times p^*(\blank)$ preserves colimits in each variable. This
allows us to verify the claim of the proposition by reducing it to a
few key special cases of Segal equivalences.

\begin{propn}\label{propn prelim Seg tensor}
  Given an object $\overline I \in \DFU$ lying over $[n] \in \simp$ and a Segal equivalence $U \to \Delta^{n}$, we write $f\colon \overline I| _{U}\to \overline I$
  for the map $\overline I \times_{p^*(\Delta^n)}p^* (U)\to \overline{I}$. For every Segal equivalence $g\colon K \to L$ in $\PSh(\simp)$, the map
  \[f \times p^*(g) \colon  \overline I| _{U} \times
  p^*(K) \to \overline I \times p^* (L)\] is a Segal equivalence.
\end{propn}
\begin{proof}
  We first observe that we have a commutative diagram
  \[
  \begin{tikzcd}
  \overline I| _{U} \times p^*(K) \arrow{r}{f \times p^*(g)} \arrow{d} & {\overline I} \times p^* (L) \arrow{r} \arrow{d} &
  \overline I \arrow{d}\\
  p^*(U) \times p^* (K) \arrow{r} & p^*(\Delta^n) \times p^* (L)\arrow{r} & p^{*}(\Delta^n)
  \end{tikzcd}
  \]
  consisting of two pullback squares where the middle vertical map is
  simple in the sense of Definition~\ref{def simple}
  by Remark~\ref{rem simple}. We first assume that the maps $U\to \Delta^n$ and $K\to L$ are inner horn inclusions. Then \cite[Corollary 2.3.2.4]{ht} says that the product $U\times K\to \Delta^n\times L$ of inner horn inclusions is inner anodyne and Corollary~\ref{cor pb of inner anodyne} shows that the left upper horizontal map $f \times p^*(g)$ is inner anodyne and hence a Segal equivalence by Proposition~\ref{propn:iaDFV}.
  If $U\to \Delta^n$ is an arbitrary Segal equivalence in $\PSh(\simp)$, then according to Proposition~\ref{prop Segal equ = inner horn} it is strongly generated by inner horn inclusions. Since the class of morphisms $f$ such that $f\times p^*(\id_K)$ is a Segal equivalence is clearly strongly saturated, we see that $f\times p^*(\id_K)$ is a Segal equivalence by the assumption that $f$ lies over a Segal equivalence $U\to \Delta^n$. By fixing the object $\overline I\in \DFU$, the same argument shows that $\id_{\overline I}\times p^*(g)$ is a Segal equivalence for every Segal equivalence $g$. Hence, the map $f \times p^*(g)$ which is the composite of $f\times p^*(\id_K)$ and $\id_{\overline{I}}\times p^*(g)$ is a Segal equivalence.
\end{proof}

Before we complete the proof of Proposition~\ref{propn Seg tensor} we
need to understand more explicitly what the result of tensoring a
corolla with $\Delta^{1}$ looks like, for which it is convenient to
introduce some notation:
\begin{defn}\label{def X pm}
  For an object ${\overline I} \in \DFU$ lying over $[1] \in \simp$,
  let ${\overline I}^{+},{\overline I}^{-} \to {\overline I}$ denote
  the Cartesian morphisms lying over $s_{0},s_{1} \colon [2] \to [1]$,
  respectively. If ${\overline I} = \mathfrak{c}_n(v)$ then we
  also write ${\overline I}^{\pm}$ as $\mathfrak{c}_{n}^{\pm}(v)$
  --- these objects are clearly functorial in
  $v \in \mathcal{U}$. The objects
  $\mathfrak{c}_{n}^{\pm}(v)$ thus lie over
  $([2], \mathbf{n} \xto{\id} \mathbf{n} \to \mathbf{1})$ and
  $([2], \mathbf{n} \to \mathbf{1} \xto{\id} \mathbf{1})$,
  respectively, with the unary vertices labelled by $\bbone$ and the
  $n$-ary vertex by $v$ in both cases.
\end{defn}

\begin{lemma}\label{lem decomposing}
  For an object ${\overline I} \in \DFU$ lying over $[1] \in \simp$,
  there is an equivalence
  \[ {\overline I}^+ \amalg_{{\overline I}} {\overline I}^- \to
  {\overline I} \times p^*(\Delta^{1}).\]
  In particular, for ${\overline I}:= \xfc_n(v)$, we have an
  equivalence
  $\mathfrak{c}_{n}^{+}(v) \amalg_{\mathfrak{c}_n(v)}
  \mathfrak{c}_{n}^{-}(v) \to \mathfrak{c}_n(v) \times
  p^*(\Delta^{1}),$ natural in $v$.
\end{lemma}
\begin{proof}
  Let
  $\sigma^{\pm} \colon \Delta^{2} \to \Delta^{1} \times \Delta^{1}$
  denote the two non-degenerate 2-simplices of
  $\Delta^{1} \times \Delta^{1}$, taking $(0,1,2)$ to
  $((0,0), (0,1), (1,1))$ and $(0,0), (1,0), (1,1))$, respectively,
  and let $\delta \colon \Delta^{1} \to \Delta^{1} \times \Delta^{1}$
  denote the diagonal map. Then the map
  $\Delta^{2} \amalg_{\Delta^{1}} \Delta^{2} \to \Delta^{1}\times
  \Delta^{1}$
  induced by $\sigma^{\pm}$ and $\delta$ is an equivalence. Since
  pullbacks in $\PSh(\DFU)$ preserve colimits, we have a natural
  pullback square \csquare{({\overline I} \times
    p^*(\Delta^{1}))_{\sigma^{+}} \amalg_{({\overline I} \times
      p^*(\Delta^{1}))_{\delta}} ({\overline I} \times
    p^*(\Delta^{1}))_{\sigma^{-}}}{{\overline I} \times
    p^*(\Delta^{1})}{p^*(\Delta^2) \amalg_{p^*(\Delta^1)}
    p^*(\Delta^2)}{p^*(\Delta^1) \times
    p^*(\Delta^1).}{\sim}{}{}{\sim} 
The right vertical map is clearly the pullback of $\overline I\to p^*(\Delta^1)$ along the projection $p^*(\Delta^1)\times p^*(\Delta^1)\to p^*(\Delta^1)$ and therefore simple by Remark~\ref{rem simple}. This implies that
  $({\overline I} \times p^*(\Delta^{1}))_{\sigma^{\pm}} \simeq
  {\overline I}^{\pm}$
  and
  $({\overline I} \times p^*(\Delta^{1}))_{\delta} \simeq {\overline
    I}$, which completes the proof.
\end{proof}

\begin{proof}[Proof of Proposition~\ref{propn Seg tensor}]
  It remains to prove that the map
  $ F \times p^*(K)\to G\times
  p^*(K)$
  is a Segal equivalence in $\PSh(\DFU)$ for every
  $K \in \PSh(\simp)$ and every Segal equivalence
  $f \colon F \to G$. 
  Since Segal equivalences are closed under colimits and
  $\blank\times p^*(\blank)$ preserves colimits in each variable, we
  may assume that $K$ is a representable presheaf
  $\Delta^{n}$. Then, by Proposition~\ref{propn prelim Seg tensor}, the two vertical morphisms in the
  commutative diagram \nolabelcsquare{F \times
    p^{*}\Delta^n_{\Seg}}{G \times
    p^{*}\Delta^n_{\Seg}}{F \times
    p^*(\Delta^n)}{G \times p^*(\Delta^n)}
   are Segal equivalences. It therefore suffices to
  prove that the top horizontal morphism is a Segal equivalence,
  and using the definition of $\Delta^n_{\Seg}$ as a colimit we see
  that for this it is enough to show that $f \times p^*(\Delta^{1})$ is a Segal
  equivalence.

  It suffices to check this for $f$ in a class of generating Segal
  equivalences; by Proposition~\ref{propn:SegDFVcond} we can consider
  the morphisms
  \begin{itemize}
  \item $\overline{I}|_{\Delta^{n}_{\Seg}} := \overline{I}|_{01} \amalg_{\overline{I}|_{1}} \overline{I}|_{12} \amalg_{\overline{I}|_{2}} \cdots
  \amalg_{\overline{I}|_{n-1}} \overline{I}|_{(n-1)n} \to \overline{I}$
  for all $\overline{I} \in \DF$,
\item 
  $ \coprod_{i \in \mathbf{m}}([1], \mathbf{n_{i}} \to \mathbf{1}, v_{i})
  \to ([1], \mathbf{n} \to \mathbf{m}, v_{1},\ldots,v_{m}),$
  for all $\mathbf{n} \to \mathbf{m}$ (including $\mathbf{m} =
  \mathbf{0}$), 
\item $\coprod_{i \in \mathbf{m}}([0], \mathbf{1}) 
  \to ([0], \mathbf{m}),$
  for all $\mathbf{m}$ (including $\mathbf{m} = \mathbf{0}$).
  \end{itemize}
  For the first class of maps the result was proved in
  Proposition~\ref{propn prelim Seg tensor}. For $f \colon \coprod
  \overline{I}_{i} \to \overline{I}$ in the second class
  of maps, by Lemma~\ref{lem decomposing} we can identify $f \times
  p^{*}\Delta^{1}$ with 
  \[\coprod_i({\overline I}_i^+\amalg_{{\overline I}_i}  {\overline
    I}_i^-) \to {\overline I}^+\amalg_{\overline I} {\overline I}^-.\]
  Since colimits commute, it suffices to show that the maps
  $\coprod_i {\overline I}_i^{\pm}\to {\overline I}^\pm$ are Segal
  equivalences. We consider the case of ${\overline I}^+$; the proof
  of the other case is the same. We have a commutative square
  \nolabelcsquare{\coprod_i {\overline I}_i^+|
    _{\Delta^2_\Seg}}{{\overline I}^+| _{\Delta^2_\Seg}}{\coprod_i
    {\overline I}_i^+}{{\overline I}^+,} where the vertical maps are
  Segal equivalences. By the $2$-of-$3$ property, we only need to see
  that the upper horizontal map given by
  $\coprod_i({({\overline I}^+_i)}_{01}\amalg_{{({\overline
        I}^+_i)}_1} {({\overline I}^+_i)}_{12})\to {{\overline
      I}^+_{01}}\amalg_{{\overline I}_1^+} {\overline I}^+_{12}$
  is a Segal equivalence. It follows from Definition~\ref{def X pm}
  that
  ${({\overline I}^+_i)}_{k,k+1} \simeq {({\overline I}_{k,k+1}^+)}_i$
  for $0\leq k\leq 1$ and that there exists $\mathbf{n_i}$ such that
  $\mathbf{n_1}\amalg\ldots\amalg\mathbf{n_m}\simeq \mathbf{n}$ and
  ${({\overline I}^+_i)}_1\simeq \coprod_{j\in \mathbf{n_i}}\overline{\xfe}$.
  Therefore, the upper horizontal map of the square is a pushout of
  maps
  $\coprod_{j\in \mathbf{n_i}}\overline{\xfe}\to {\overline I}_1\simeq
  ([0],\mathbf{n})$
  and
  $\coprod_i {({\overline I}_{k,k+1}^+)}_i\to {\overline I}_{k,k+1}^+$
  which are Segal equivalences.
    
  It remains to consider the third class of maps. The canonical equivalence
  $\sigma\colon\Delta^1\hookrightarrow \Delta^0\times \Delta^1$
  induces a pullback square
  \csquare{([0],\mathbf{n})_\sigma}{([0],\mathbf{n})\times
    p^*(\Delta^1)}{p^*(\Delta^1)}{p^*(\Delta^0)\times
    p^*(\Delta^1)}{\sim}{}{}{\sim} where the horizontal maps are
  equivalences.  
  By Remark~\ref{rem simple}, the right vertical map, which is a pullback of $([0], \mathbf{n})\to p^*(\Delta^0)$ along the projection $p^*(\Delta^0)\times p^*(\Delta^1)\to p^*(\Delta^0)$, is simple. Therefore, the
  presheaf $([0],\mathbf{n})_\sigma$ is represented by
  $\tau^*([0],\mathbf{n})$ given by the Cartesian lift
  $\tau^* ([0],\mathbf{n})\to ([0],\mathbf{n})$ of the projection
  $\tau\colon \Delta^1\to \Delta^0$. Hence, the the presheaf
  $([0],\mathbf{n})\times p^*(\Delta^1)$ is represented by
  $I(\bbone_c)_{c\in \CorDF(I)}$, where
  $I=([1],\mathbf{n}\xto{\id}\mathbf{n})$ and $\bbone$ denotes the
  unit in $ \xU$. In particular, we have
  $\overline{\xfe}\times p^*(\Delta^1)\simeq \xfc_1(\bbone)$. Therefore, it is
  clear that the map
  \[ \coprod_{i\in \mathbf{n}} \overline{\mathfrak{e}}\times p^*(\Delta^1)\simeq
  \coprod_{i\in \mathbf{n}}\xfc_1(\bbone)\to I(\bbone_c)_{c\in
    \CorDF(I)}\simeq ([0], \mathbf{n})\times p^*(\Delta^1)\]
  is a Segal equivalence.    
\end{proof}

This completes the proof of Theorem~\ref{theo tensor}. As a
consequence, we get:
\begin{cor}\label{cor tensor S}
 Let $\mathcal{U}$ be a small symmetric monoidal \icat{} and
  $\mathbb{S}$ a set of morphisms in $\PSh(\mathcal{U})$ compatible
  with the symmetric monoidal structure. Then the
  $\PSeg(\simp)$-module structure on $\PSeg(\DFU)$ induces a
  $\PSeg(\simp)$-module structure on $\PSSeg(\DFU)$.
\end{cor}
\begin{proof}
  We must show that for $s \colon X \to Y$ in $\mathbb{S}$, the map
  \[ y^{*}\mathfrak{c}_n(s) \times p^{*}K \colon  y^{*}\mathfrak{c}_n(X) \times p^{*}K \to
  y^{*}\mathfrak{c}_{n}(Y) \times p^{*}K\]
  is an $\mathbb{S}$-Segal equivalence for all $K$ in
  $\PSh(\simp)$. As in the proof of Proposition~\ref{propn Seg
    tensor}, we can use Proposition~\ref{propn prelim Seg tensor}
  to reduce to the case where $K = \Delta^{1}$. 

  Using Lemma~\ref{lem decomposing} we can (since
  $y^{*}\mathfrak{c}_n(X)$ is a colimit of the objects considered
  there) write the map $y^{*}\mathfrak{c}_n(s) \times
  p^{*}\Delta^{1}$ as
  \[ y^{*}\mathfrak{c}_{n}^{+}(X) \amalg_{y^{*}(\mathfrak{c}_{n},
    X)} y^{*}\mathfrak{c}_{n}^{-}(X) \to y^{*}\mathfrak{c}_{n}^{+}(Y) \amalg_{y^{*}\mathfrak{c}_{n}(Y)} y^{*}\mathfrak{c}_{n}^{-}(Y).\]
  It then suffices to show that the morphisms
  $y^{*}\mathfrak{c}_{n}^{\pm}(X) \to y^{*}\mathfrak{c}_{n}^{\pm}(Y)$
  are both $\mathbb{S}$-Segal equivalences. We consider the case of
  $\mathfrak{c}_{n}^{+}$; the proof for $\mathfrak{c}_{n}^{-}$ is the
  same. We then have a commutative diagram
  \nolabelcsquare{\left(\coprod_{n}\mathfrak{c}_{1}(\bbone)\right)
    \amalg_{\coprod_{n} \overline{\mathfrak{e}}} y^*\mathfrak{c}_{n}(X)}{\left(\coprod_{n} \mathfrak{c}_{1}(\bbone)\right)
    \amalg_{\coprod_{n} \overline{\mathfrak{e}}} y^*\mathfrak{c}_{n}(Y)}{y^*\mathfrak{c}^{+}_{n}(X)}{y^*\mathfrak{c}^{+}_{n}(Y).}
  Here the top horizontal map is an $\mathbb{S}$-Segal equivalence, as
  it is a pushout of
  $y^{*}\mathfrak{c}_n(X) \to y^{*}\mathfrak{c}_{n}(Y)$.
  Since the Yoneda embedding $y\colon \xU\to \PSh(\xU)$ is symmetric monoidal, the vertical maps are of the form $y^*\xfc^+_n(Z)_\Seg\to y^*\xfc^+_n(Z)$. In particular, they are Segal equivalences and 
by the 2-of-3 property, the bottom horizontal map
  is also an $\mathbb{S}$-Segal equivalence.
\end{proof}

As a special case, we get:
\begin{cor}\label{cor:PCStensor}
  Let $\mathcal{V}$ be a  presentably symmetric monoidal \icat{}. Then
  $\PCS(\DFV)$ has a $\PSeg(\simp)$-module structure such that the
  tensoring preserves colimits in each variable.
\end{cor}

\begin{proof}
  Since $\mathcal{V}$ is presentably symmetric monoidal, by
  Proposition~\ref{propn Vkappa symmon} we can choose
  a regular cardinal $\kappa$ such that $\mathcal{V}^{\kappa}$ is a
  symmetric monoidal subcategory of $\mathcal{V}$, and $\mathcal{V}
  \simeq \Ind_{\kappa}\mathcal{V}^{\kappa}$ as a symmetric monoidal
  \icat{}. The result is then a special case of
  Corollary~\ref{cor tensor S} applied to $\mathcal{V}^{\kappa}$ with
  $\mathbb{S}$ as in the proof of Corollary~\ref{cor:kappaseg}.
\end{proof}

Applying the adjoint functor theorem, we get:
\begin{corollary}\label{cor Alg(-,-)}
  Let $\mathcal{V}$ be a presentably symmetric monoidal \icat{}. The tensor product
  $\otimes\colon \PCS(\DFV)\times \PSh_{\Seg}(\simp)\to
  \PSh_{\xCS}({\DF^\xV})$ induces 
  \[\xAlg^\xV_{(\blank)}(\blank)\colon (\PSh_{\xCS}({\DF^\xV}))^\op\times
  \PSh_{\xCS}({\DF^\xV}) \to \PSh_{\Seg}(\simp)\] such that
  \[ \Map_{\PSeg(\simp)}(\mathcal{X}, \Alg^{\mathcal{V}}_{\mathcal{O}}(\mathcal{P}))
  \simeq \Map_{\PCS(\DFV)}(\mathcal{O} \otimes \mathcal{X},
  \mathcal{P})\]
  and a cotensor product 
  \[(\blank)^{(\blank)}\colon
  \PSh_{\xCS}({\DF^\xV})\times (\PSh_{\Seg}(\simp))^\op\to
  \PSh_{\xCS}({\DF^\xV})\]
  such that
  \[\Map_{\PCS(\DFV)}(\mathcal{O}, \mathcal{P}^{\mathcal{X}}) \simeq
  \Map_{\PCS(\DFV)}(\mathcal{O} \otimes \mathcal{X}, \mathcal{P}).\]
  Moreover, both of these functors preserve limits in each variable.
\end{corollary}

\begin{remark}
  Since $\PCS(\DFV)$ is a module over $\PSeg(\simp)$, for
  $\mathcal{O} \in \PCS(\DFV)$ and $\mathcal{X}, \mathcal{Y} \in
  \PSeg(\simp)$ we have a
  natural equivalence 
  \[ (\mathcal{O} \otimes \mathcal{X}) \otimes \mathcal{Y} \simeq
  \mathcal{O} \otimes (\mathcal{X} \times \mathcal{Y}).\]
  This induces a natural equivalence
  \[ \mathcal{O}^{\mathcal{X} \times \mathcal{Y}} \simeq (\mathcal{O}^{\mathcal{X}})^{\mathcal{Y}}.\]
\end{remark}

\subsection{The Underlying Enriched $\infty$-Category}\label{sec under
  icat}
In this subsection we define the underlying enriched \icat{} of an
enriched \iopd{}, by extracting from a (continuous) Segal presheaf on
$\DFV$ a (continuous) Segal presheaf on $\simp^{\mathcal{V}}$. Here we
assume without further comment that the analogues of the results of
\S\ref{subsec psh}--\ref{sec tensor product} also hold for Segal
presheaves on $\simp^{\mathcal{V}}$, by simpler versions of the
arguments given above.

\begin{defn}
  Let $u \colon \simp \to \DF$ denote the functor taking $[n]$
  to
  $([n], \mathbf{1} \xto{\id} \mathbf{1} \to \cdots \to \mathbf{1})$.
  We write $\overline{u} \colon \simp^\xV\hookrightarrow \DFV $ for
  the map given by the pullback
  \[
  \begin{tikzcd}
    \simp^\xV \ar[hookrightarrow]{r}{\overline{u}} \arrow{d} & \DFV \arrow{d}\\
    \simp \arrow[hookrightarrow]{r}{u} &\DF.
  \end{tikzcd}
  \]
\end{defn}

\begin{propn}
  Let $\mathcal{U}$ be a small symmetric monoidal \icat{}. The functors $\overline{u}_{!}$ and $\overline{u}^{*}$ satisfy:
  \begin{enumerate}[(i)]
  \item $\overline{u}_{!} \colon \PSh(\simp^{\mathcal{U}}) \to
    \PSh(\DFU)$ preserves Segal presheaves.
  \item $\overline{u}_{!}$ preserves Segal equivalences.
  \item
    $\overline{u}^{*} \colon \PSh(\DFU) \to \PSh(\simp^{\mathcal{U}})$
    preserves Segal presheaves.
  \item The adjunction $\overline{u}_{!} \dashv \overline{u}^{*}$
    restricts to an adjunction between the \icats{} of Segal presheaves.
  \end{enumerate}
\end{propn}
\begin{proof}
  For a Segal presheaf
  $F \colon \simp^{\mathcal{U},\op} \to \mathcal{S}$, we must show
  that $\overline{u}_{!}F$ is a Segal presheaf on $\DFU$. This is clear, since
  $\overline{u}_{!}F(\overline{I}) \simeq \emptyset$ for any $I$ not in the image
  of $u$. Moreover, the images of the generating Segal equivalences in
  $\PSh(\simp^{\mathcal{U}})$ under $\overline{u}_{!}$ are clearly
  among the generating Segal equivalences in $\PSh(\DFU)$. This proves
  (ii), and (iii) is an immediate consequence of (ii). Then (iv)
  follows since both functors preserve the full subcategories of Segal
  presheaves.
\end{proof}

The \icat{} $\PSeg(\simp^{\mathcal{U}})$ is tensored and cotensored
over $\PSeg(\simp)$, with $\mathcal{C} \otimes \xX$ for $\mathcal{C} \in
\PSeg(\simp^{\mathcal{U}})$ and $\xX \in \PSeg(\simp)$ given by the
localization of $\mathcal{C} \times q^{*}\xX$ for $q$ the projection
$\simp^{\mathcal{U}} \to \simp$. This is compatible with the tensoring
of $\PSeg(\DFU)$:
\begin{propn}
  For $\mathcal{C} \in \PSeg(\simp^{\mathcal{U}})$ and $\xX \in \PSeg(\simp)$
  there is a natural equivalence
  \[\overline{u}_{!}(\mathcal{C} \otimes \xX) \isoto (\overline{u}_{!}\mathcal{C})
  \otimes \xX.\]
\end{propn}
\begin{proof}
  For $\mathcal{O}$ a Segal presheaf on $\DFU$ and $\xX \in
  \PSeg(\simp)$, the object $\mathcal{O} \otimes \xX$ is defined as the
  localization of $\mathcal{O} \times p^{*}\xX$, where $p$ is the composite $\DFU\to \simp^\xU\overset{q}{\to}\simp$ of the obvious projections. We clearly have
  $\overline{u}^{*}(\mathcal{O} \times p^{*}\xX) \simeq
  \overline{u}^{*}\mathcal{O} \times q^{*}\xX$. Moreover,
  $\overline{u}^{*}(\mathcal{O} \otimes \xX)$ is a Segal presheaf since
  $\overline{u}^{*}$ preserves these. Hence, the map
  $\overline{u}^{*}(\mathcal{O} \times p^{*}\xX) \to \overline{u}^{*}(\mathcal{O} \otimes \xX)$
  factors through the localization $(\overline{u}^{*}\mathcal{O}) \otimes \xX$ to
  give a natural map $(\overline{u}^{*}\mathcal{O}) \otimes \xX \to
  \overline{u}^{*}(\mathcal{O} \otimes \xX)$. Applying this with $\mathcal{O} =
  \overline{u}_{!}\mathcal{C}$ and combining it with the counit for the
  adjunction we get a natural map \[\mathcal{C} \otimes \xX \to
  \overline{u}^{*}\overline{u}_{!}\mathcal{C} \otimes \xX \to
  \overline{u}^{*}(\overline{u}_{!}\mathcal{C} \otimes \xX),\]
  which induces the required natural map
  \[ \overline{u}_{!}(\mathcal{C} \otimes \xX) \to
  \overline{u}_{!}\mathcal{C} \otimes \xX\]
  by adjunction. Since $\overline{u}_{!}$ preserves Segal equivalences
  and colimits, as do the two tensor products in each variable, it
  suffices to check that the map
  \[ \overline{u}_{!}([1] (v) \otimes \Delta^{1}) \to
  \overline{u}_{!}[1] (v) \otimes \Delta^{1} \simeq
  \mathfrak{c}_{1}(v) \otimes \Delta^{1}\]
  is an equivalence, which is clear from Lemma~\ref{lem decomposing}
  and its analogue for $\simp^{\mathcal{U}}$.
\end{proof}

\begin{cor}
  There is a natural equivalence $\overline{u}^{*}(\mathcal{O}^{\xX})
  \simeq (\overline{u}^{*}\mathcal{O})^{\xX}$.
\end{cor}
\begin{proof}
  We have natural equivalences 
  \[\Map(\mathcal{C}, \overline{u}^{*}\mathcal{O}^{\xX}) \simeq
  \Map(\overline{u}_{!}\mathcal{C} \otimes \xX, \mathcal{O}) \simeq
  \Map(\overline{u}_{!}(\mathcal{C} \otimes \xX), \mathcal{O}) \simeq
  \Map(\mathcal{C}, (\overline{u}^{*}\mathcal{O})^{\xX}).\qedhere\]
\end{proof}

The preceding results are compatible with localization, in the
following sense:
\begin{cor}\label{cor:PSu}
  Suppose $\mathcal{U}$ is a small symmetric monoidal \icat{} and
  $\mathbb{S}$ is a set of morphisms in $\PSh(\mathcal{U})$ compatible
  with the symmetric monoidal structure. Then:
  \begin{enumerate}[(i)]
  \item The adjunction $\overline{u}_{!} \dashv \overline{u}^{*}$
    restricts to an adjunction
    \[ \overline{u}_{!} : \PSSeg(\simp^{\mathcal{U}}) \rightleftarrows
    \PSSeg(\DFU) : \overline{u}^{*}.\]
  \item The functor $\overline{u}_{!}$ is compatible with the
    tensoring with $\PSeg(\simp)$ in the sense that there is a natural
    equivalence $\overline{u}_{!}(\mathcal{C}
    \otimes \xX) \simeq \overline{u}_{!}(\mathcal{C}) \otimes \xX$ for $\mathcal{C} \in
    \PSSeg(\simp^{\mathcal{U}})$ and $\xX \in \PSeg(\simp)$.
  \item There is a natural equivalence
    $\overline{u}^{*}(\mathcal{O}^{\xX}) \simeq
    (\overline{u}^{*}\mathcal{O})^{\xX}$ for $\mathcal{O} \in
    \PSSeg(\DFU)$ and $\xX \in \PSeg(\simp)$.    
  \end{enumerate}
\end{cor}

As usual, this extends to the presentable case by choosing a suitable
subcategory of $\kappa$-compact objects, giving:
\begin{cor}\label{cor:PCSu}
  Suppose $\mathcal{V}$ is a presentably symmetric monoidal \icat{}. Then
  \begin{enumerate}[(i)]
  \item The adjunction $\overline{u}_{!} \dashv \overline{u}^{*}$
    restricts to an adjunction
    \[ \overline{u}_{!} : \PCS(\simp^{\mathcal{V}}) \rightleftarrows
    \PCS(\DFV) : \overline{u}^{*}.\]
  \item The functor $\overline{u}_{!}$ is compatible with the
    tensoring with $\PSeg(\simp)$ in the sense that there is a natural
    equivalence $\overline{u}_{!}(\mathcal{C}
    \otimes \xX) \simeq \overline{u}_{!}(\mathcal{C}) \otimes \xX$ for $\mathcal{C} \in
    \PCS(\simp^{\mathcal{V}})$ and $\xX \in \PSeg(\simp)$.
  \item There is a natural equivalence
    $\overline{u}^{*}\mathcal{O}^{\xX} \simeq
    (\overline{u}^{*}\mathcal{O})^{\xX}$ for $\mathcal{O} \in
    \PCS(\DFV)$ and $\xX \in \PSeg(\simp)$.
  \end{enumerate}
\end{cor}

\begin{warning}
  Note that here the functors $\overline{u}_{!}$ and
  $\overline{u}^{*}$ for continuous Segal presheaves for $\mathcal{V}$
  refer to those obtained from the analogoues functors for Segal
  presheaves on $\mathcal{V}^{\kappa}$ for a suitable cardinal
  $\kappa$. In particular, $\overline{u}_{!}$ is \emph{not} just given
  by a left Kan extension along $\simp^{\mathcal{V},\op} \to \DFVop$.
\end{warning}

\begin{remark}\label{rem underlying cat}
  Since the $\infty$-category $\PCS(\simp^\xV)$ models $\xV$-enriched
  $\infty$-categories by \cite[Theorem 4.5.3]{enriched}, we call
  $\overline{u}^*(\xxO)$ the \emph{underlying $\xV$-enriched \icat{} of $\xxO$}
  for every continuous Segal presheaf $\xxO\in \PCS(\DFV)$.
\end{remark}

\section{The Completion Theorem}\label{sec ffes}
So far we have considered the ``algebraic'' theory of enriched
\iopds{}. However, just as in the case of Segal spaces, to produce the
correct homotopy theory of enriched \iopds{} we need to invert the
class of fully faithful and essentially surjective moprhisms. 
In \cite{Rezk} Rezk showed that the localization of Segal spaces at
the fully faithful and essentially surjective morphisms is given by
the full subcategory of \emph{complete} objects, which are those Segal
spaces whose space of objects is equivalent to their ``classifying
space of equivalences''. Our main goal in this section is to prove the
analogous result for enriched $\infty$-operads, in the form of
continuous Segal presheaves on $\DFV$.

In \S\ref{subsec css} we review the definition of complete objects in
Segal spaces and in enriched \icats{}, and in \S\ref{subsec cspsh} we
extend this to define complete continuous Segal presheaves on
$\DFV$. Then we introduce fully faithful and essentially surjective
functors in \S\ref{subsec ffes}, and a notion of
``pseudo-equivalences'' in \S\ref{subsec pseudo}; these are morphisms
with a ``pseudo-inverse'', i.e.\ an inverse up to natural
equivalence. Lastly, in \S\ref{subsec compl} we prove our analogue of
Rezk's completion theorem.

\subsection{Completeness for Segal Spaces and Enriched
  $\infty$-Categories}\label{subsec css}
To fix notation, we recall the definition of complete Segal spaces
from \cite{Rezk} and the analogous definition of complete enriched
\icats{} from \cite{enriched}.

\begin{definition}\label{def En}
  Let $E^{n}$ denote the (contractible) category with $n+1$ objects
  $0,1,\ldots,n$ and a unique morphism between any pair of
  objects.  We
  also denote the Segal space corresponding to this category by $E^{n}$.
\end{definition}
\begin{remark}
  The category $E^{1}$ is the ``generic isomorphism'', so giving a morphism
  of Segal spaces $E^{1} \to X$ corresponds to giving two objects of
  $X$ and an equivalence between them. Similarly, giving a map $E^n\to
  X$ amounts to specifying 
  $n+1$ equivalent objects in $X$.
\end{remark}

\begin{defn}
  A Segal space $F$ is \emph{complete} if the map
  \[F_{0} \to \Map(E^{1}, F)\]
  induced by the map $E^{1} \to E^{0} \simeq *$, is an equivalence in $\xS$.  We
  write $\PCoS(\simp)$ for the full subcategory of $\PSeg(\simp)$
  spanned by the complete Segal spaces.
\end{defn}

\begin{defn}
  We can view the category $E^{n}$ as enriched in the initial monoidal
  \icat{} $*$ and so by transferring the enrichment we get for any
  symmetric monoidal \icat{} $\mathcal{V}$ an enriched \icat{}
  $E^{n}_{\mathcal{V}}$ with morphisms
  $E^{n}_{\mathcal{V}}(i,j) \simeq \bbone_{\mathcal{V}}$ for all
  $i,j$. We will usually denote this simply as $E^{n}$, leaving
  $\mathcal{V}$ implicit. A $\mathcal{V}$-enriched \icat{} is
  \emph{complete} if it is local with respect to $E^{1} \to
  E^{0}$. Viewing enriched \icats{} as continuous Segal presheaves on
  $\simp^{\mathcal{V}}$ (for $\mathcal{V}$ presentably symmetric
  monoidal) we write $\PCCS(\simp^{\mathcal{V}})$ for the full
  subcategory of $\PSh(\simp^{\mathcal{V}})$ spanned by the complete
  continuous Segal presheaves.
\end{defn}

\begin{definition}
  For every $\xcc\in \PSh_{\xCS}(\simp^\xV)$, let $\iota_n(\xcc)$
  denote the space $\xMap_{\PSh_{\xCS}(\simp^\xV)}(E^n,\xcc) $. We
  call a map $E^1\to \xcc$ in $\PSh_{\xCS}(\simp^\xV)$ an
  \emph{equivalence in $\xcc$} and the space $\iota_1(\xcc)$ the
  \emph{space of equivalences} of $\xcc$.  We write
  $\iota \mathcal{C}$ for the colimit of the simplicial object
  $ \iota_\bullet(\xcc) =
  \xMap_{\PSh_{\xCS}(\simp^\xV)}(E^\bullet,\xcc)$
  and call this the \emph{classifying space of equivalences} in
  $\mathcal{C}$.
\end{definition}

\begin{remark}
  We refer the reader to \cite[\S\S 5.1--5.2]{enriched} for a much
  more substantial treatment of equivalences in enriched \icats{}.
\end{remark}

\subsection{Complete Segal Presheaves}\label{subsec cspsh}
\begin{defn}
  If $\mathcal{O}$ is a Segal presheaf on $\DF$, then the underlying
  \icat{} $u^{*}\mathcal{O}$ (see \S\ref{sec under icat}) is a Segal
  space. We say $\mathcal{O}$ is \emph{complete} if its underlying
  Segal space $u^{*}\mathcal{O}$ is a complete Segal
  space. Similarly, $\mathcal{O} \in \PCS(\DFV)$ is \emph{complete} if
  $\overline{u}^{*}\mathcal{O}$ is a complete object in
  $\PCS(\simp^{\mathcal{V}})$; we write $\PCCS(\DFV)$ for the full
  subcategory of $\PCS(\DFV)$ spanned by the complete objects.
\end{defn}

\begin{notation}
  We will also write $\OpdIV$ for the \icat{} $\PCCS(\DFV)$ when we do
  not wish to emphasize the specific implementation of enriched
  \iopds{} as continuous Segal presheaves.
\end{notation}
 
\begin{notation}
  For $\xxO\in \PSh_{\xCS}(\DFV )$ we abbreviate $\iota_{n}\xxO :=
  \iota_n \overline{u}^*(\mathcal{O})$. Then $\mathcal{O}$ is complete
  \IFF{} the canonical map $\iota_1 \xxO\to \iota_0 \xxO$ is an
  equivalence in $\mathcal{S}$.
\end{notation}

\begin{lemma}\label{lem equivalence}
  For every continuous Segal presheaf $\xxO\in \PSh_{\xCS}(\DFV)$, the
  objects $\iota_0 \xxO$ and $\xxO(\overline{\xfe})$ are equivalent
  in $\xS$.
\end{lemma}
\begin{proof}
  For every object $\xxO\in \PSh_{\xCS}(\DFV)$, we have the following
  equivalences:
  \[
  \iota_0 \xxO= \xMap(E^0,\overline{u}^*(\xxO))\simeq\xMap(\overline{u}_!E^0,\xxO)\simeq \xMap(\overline\xfe,\xxO)\simeq \xxO(\overline\xfe).\qedhere
  \] 
\end{proof}
\begin{remark}
  This implies that the functor $\iota_0$ preserves colimits.
\end{remark}

\begin{remark}\label{rem completeness}
  For a presentably symmetric monoidal $\infty$-category
  $\xV^\otimes$, there exists a unique colimit-preserving functor
  $F_0\colon\xS\to \xV$ extending $\ast\mapsto\bbone$. This is
  symmetric monoidal, and it has a lax monoidal right adjoint $G
  \colon \mathcal{V} \to \mathcal{S}$ given by $\xMap_{\mathcal{V}}(\bbone, \blank)$.
  By \cite[Theorem 4.4.7, Theorem
  4.5.3]{enriched}, the right adjoint induces a map
  $G_*\colon \PSh_{\xCS}(\simp^\xV)\to\PSh_{\Seg}(\simp)$ which
  carries a Segal presheaf to its underlying Segal space. Moreover,
  \cite[Proposition 5.1.11]{enriched} implies that this functor
  detects complete objects, meaning $\mathcal{C}\in \PSh_{\xCS}(\simp^\xV)$ is
  complete if and only if $G_*(\mathcal{C})$ is a complete Segal space.
\end{remark}
\begin{definition}
   A morphism in $\PSh_{\xCS}(\DFV)$ is called a
  \emph{local equivalence} if it lies in the strongly saturated class
  of maps generated by $u_!(s^0)$, where  $s^0$ denotes the canonical map
  $E^1\to E^0$ in $\PSh_{\xCS}(\DFV)$. Equivalently, a morphism $f
  \colon \mathcal{O} \to \mathcal{O}'$ is a local equivalence \IFF{}
  for every complete object $\mathcal{P}$ the induced map
  $\Map(\mathcal{O}', \mathcal{P}) \to \Map(\mathcal{O}, \mathcal{P})$
  is an equivalence.
\end{definition}

\subsection{Fully Faithful and Essentially Surjective
  Functors}\label{subsec ffes}

\begin{definition}\label{def ffes}
Given two objects $\xxO,\xxP\in\PSh_{\xCS}(\DFV)$, we say that a morphism $f\colon \xxO\to \xxP$ is 
\begin{itemize}
  \item \emph{fully faithful} if the commutative square
  \csquare{\xxO(\overline \xfc_n)}{\xxP(\overline \xfc_n) }{\xxO(\overline\xfe)^{n+1}}{\xxP(\overline\xfe)^{n+1}   }{f(\overline \xfc_n)}{}{}{f(\overline \xfe)^{n+1}}
  is Cartesian for every $\overline \xfc_n\in\DFV$ lying over a
        corolla $\mathfrak{c}_{n}$ in $\DF$.
  \item \emph{essentially surjective} if the induced functor $\overline u^*(f)\colon \overline u^*(\xxO)\to\overline  u^*(\xxP)$ of the underlying $\xV$-enriched $\infty$-categories is essentially surjective, i.e. the map $\pi_0(\iota f)\colon \pi_0(\iota \xxO)\to \pi_0(\iota \xxP)$ is a surjection of sets.
\end{itemize}
\end{definition}

\begin{remark}
  A morphism $f \colon \mathcal{O} \to \mathcal{P}$ of Segal
  presheaves on $\DFV$ is fully faithful \IFF{} for all
  $\overline \xfc_n \in \DFV$ and
  $x_{1},\ldots,x_{n},y\in \xxO(\overline\xfe)$ the natural map
  $\mathcal{O}(\overline \xfc_n(x_{1},\ldots,x_{n},y)) \to
  \mathcal{P}(\overline \xfc_n(f x_{1},\ldots,f x_{n},f y))$
  is an equivalence. Since these are both representable presheaves in
  $v$, this is equivalent to the representing morphism
  $\mathcal{O}(x_{1},\ldots,x_{n};y) \to \mathcal{P}(f
  x_{1},\ldots, f x_{n}; f y)$
  being an equivalence in $\mathcal{V}$, which is what we would expect
  the notion of fully faithful to mean.
\end{remark}

\begin{remark}\label{rem surj on pi0}
  By \cite[Lemma 5.3.4]{enriched}, a functor of Segal presheaves is
  essentially surjective in our sense if and only for every object
  $x\in \iota_0\xxP$ there exists an equivalence (i.e. a functor of
  $\xV$-enriched $\infty$-categories $E^1 \to \overline{u}^*(\xxP)$)
  connecting $x$ to an object lying in the image of $\iota_0(f)$ in
  $\iota_0 \xxP$.
\end{remark}
\begin{lemma}\label{lem ff = Cartesian morphism}
  Evaluation at $\overline \xfe\in \DFV$ gives a map
  $\xev_{\overline{\mathfrak{e}}}\colon \PSh_{\xCS}(\DFV)\to \xS$,
 which is a Cartesian
  fibration by Proposition~\ref{propn:SegisMon}. A morphism $f\colon
  \xxO\to \xxP$ in $\PSh_{\xCS}(\DFV)$ is
  $\xev_{\overline{\mathfrak{e}}}$-Cartesian \IFF{} is fully faithful.
\end{lemma}
\begin{proof}
  Suppose $f$ lies over
  $\xev_{\overline{\mathfrak{e}}}(f)\colon {X}\to {Y}\in \xS$.
  We factor the morphism $f$ in
  $\PSh_{\xCS}(\DFV)$ into a morphism
  $g\colon \xxO\to \xev_{\overline{\mathfrak{e}}}(f)^*(\xxP)$ lying
  over $\id_{X}$ followed by an
  $\xev_{\overline{\mathfrak{e}}}$-Cartesian morphism
  $h\colon \xev_{\overline{\mathfrak{e}}}(f)^*(\xxP)\to \xxP$ lying
  over $\xev_{\overline{\mathfrak{e}}}(f)$. If $f$ is fully faithful,
  then we wish to show that $g$ is an equivalence in
  $\PSh_{\xCS}(\DFV)$, which is equivalent to requiring $g$ to be an
  equivalence in the fibre $\PSh_{\xCS}(\DFV)_{X}$.  Hence, we need to
  show that
  $g(\overline I)\colon \xxO(\overline I)\to
  \xev_{\overline{\mathfrak{e}}}(f)^*(\xxP)(\overline I)$
  is an equivalence for every object $\overline I\in \DFV$. For every
  object $(\overline \xfc_n)$ in $\DFV$ lying over a corolla, the
  $\infty$-groupoid
  $\xev_{\overline{\mathfrak{e}}}(f)^*(\xxP)(\overline \xfc_n)$ is
  given by the pullback
  ${X}^{n+1}\times_{{Y}^{n+1}} \xxP(\overline \xfc_n)$ which is
  equivalent to $\xxO(\overline \xfc_n)$ since $f$ is fully faithful.
  The Segal condition in the definition of $\PSh_{\xCS}(\DFV)_{X}$
  then implies that $\xxO(\overline I)$ is equivalent to
  $\xev_{\overline{\mathfrak{e}}}(f)^*(\xxP)(\overline I)$ for every
  $\overline I\in \DFV$.
  
  Conversely, if $f$ is an $\xev_{\overline{\mathfrak{e}}}$-Cartesian
  morphism over $\xev_{\overline{\mathfrak{e}}}(f)\colon {X}\to {Y}$,
  then $\xxO(\overline \xfc_n)$ is equivalent to
  ${X}^{n+1}\times_{{Y}^{n+1}} \xxP(\overline \xfc_n)$, and thus $f$
  is fully faithful.
\end{proof}
\begin{proposition}\label{prop 2 of 3}
  Fully faithful and essentially surjective morphisms in
  $\PSh_{\xCS}(\DFV)$ satisfy the $2$-of-$3$ property.
\end{proposition}
\begin{proof}
  As \cite[Proposition 5.3.9]{enriched}.
\end{proof}

\begin{proposition}\label{prop ffes are equ}
A fully faithful and essentially surjective morphism in $\PSh_{\xCS}(\DFV)$ between complete objects is an equivalence. 
\end{proposition}
\begin{proof}
  Let $f\colon \xxO\to \xxP$ be a fully faithful and essentially surjective
  morphism in $\PSh_{\xCS}(\DFV)$. It follows
  from Definition~\ref{def ffes} that the map $\overline{u}^*(f)$ between the
  underlying $\xV$-enriched $\infty$-categories is fully faithful and
  essentially surjective, and it is then an equivalence by
  \cite[Corollary 5.3.8]{enriched}. This implies that the Cartesian
  fibration
  $\xev_{\overline{\mathfrak{e}}}\colon \PSh_{\xCS}(\DFV)\to \xS$ induced
  by the evaluation at $\overline\xfe$ carries $f$ to an equivalence
 in $\xS$. Since the fully faithful map $f$ is then an
  $\xev_{\overline{\mathfrak{e}}}$-Cartesian lift of the equivalence $\xev_{\overline{\mathfrak{e}}}(f)$ by Lemma~\ref{lem
    ff = Cartesian morphism}, it has to be an equivalence as well.
\end{proof}

\subsection{Pseudo-Equivalences}\label{subsec pseudo}

In this subsection we consider morphisms that admit an inverse up to
natural equivalence, which we call \emph{pseudo-equivalences}. We will
show that these are both local equivalences and fully faithful and
essentially surjective, which will be used to prove the completion
theorem in the next subsection.

\begin{definition}\label{def nat equivalences}
  Let $d^0,d^1\colon E^0\to E^1$ be the maps induced by the two
  inclusions $\mathbf{1}\hookrightarrow\mathbf{2}$ of the sets of
  objects. A \emph{natural equivalence} of morphisms from
  $\mathcal{O}$ to $\mathcal{P}$ is a morphism
  $h\colon \xxO\otimes E^1\to \xxP$ ($\xxO\times E^1\to \xxP$ in case
  of $\PSh_{\Seg}(\simp)$). We say that $f$ and $g$ are
  \emph{naturally equivalent} if there exists a natural equivalence
  $h$ such that $h \circ (\id \otimes d^{0}) \simeq f$ and
  $h \circ (\id \otimes d^{1}) \simeq g$.
\end{definition}

\begin{definition}\label{def pseudoequivalence}
  For a morphism $f\colon \xxO\to \xxP$ in $\PSh_{\xCS}(\DFV)$ (or $\PSh_{\Seg}(\simp)$), a \emph{pseudo-inverse} of $f$ is a morphism $g\colon \xxP\to \xxO$ such that there exist natural equivalences $\phi\colon \id_\xxO\to g\circ f$ and $\psi\colon f\circ g\to \id_\xxP$. We call a morphism $f$ a \emph{pseudo-equivalence} if it has a pseudo-inverse $g$ and we call the quadruple $(f,g,\phi,\psi)$ a \emph{pseudo-equivalence datum}. 
\end{definition}

The following proposition is an operadic variant of \cite[Proposition
5.5.3]{enriched}.
\begin{proposition}\label{prop pseudoequ are ffes}
  If $f\colon \xxO\to \xxP$ is a pseudo-equivalence, then it is fully faithful and essentially surjective.
\end{proposition}
\begin{proof}
  Let $(f,g,\phi,\psi)$ be a pseudo-equivalence datum associated to
  the pseudo-equivalence $f$. By Remark~\ref{rem surj on pi0}, to
  prove the essential surjectivity it suffices to find for every
  object $y\in \iota_{0} \xxP$ an object $x$ and an equivalence $f(x)\simeq y$. But by defining $x:= g(y)$, we get an equivalence $f\circ g(y)\simeq y$ induced by the natural equivalence $\psi$.
  
  For the full faithfulness of $f$, we have to show that for every corolla $\xfc_n\in \DFV$ and every object $\bar x=(x_1,\ldots,x_n, x)\in \xxO(\mathfrak{e})^{n+1}$
  the map $f(\overline\xfc_n)\colon \xxO(\overline
  \xfc_n)\to\xxP(\overline \xfc_n)$ restricts to an equivalence of
  fibres $f_{\bar x}\colon \xxO(\overline \xfc_n(\bar x))\to
  \xxP(\overline \xfc_n(f\bar x))$. Using the natural equivalence $\phi\colon \id_\xxO\to g\circ f$, the map $g(\overline\xfc_n)$ restricts to a map $g_{(f\bar x,f\bar y)}$ of fibres $\xxP(\overline \xfc_n(f\bar x))\to \xxO(\overline \xfc_n(gf\bar x))$ and we have
  $\id\simeq g_{f\bar x}\circ f_{\bar x}$.
  It remains to see that $f_{\bar x}\circ g_{f\bar x}\simeq \id$. The natural equivalence $\psi\colon f\circ g\to \id_\xxP$ induces an equivalence $f_{gf\bar x}\circ g_{f\bar x}\simeq \id$ and $\phi\colon \id_\xxO \to g\circ f$ gives an equivalence $f_{gf\bar x} \simeq f_{\bar x}$. Hence, we have $f_{\bar x}\circ g_{f\bar x}\simeq f_{gf\bar x}\circ g_{f\bar x}\simeq  \id$.
\end{proof}

\begin{corollary}\label{cor cat equ betw cp obj are equ}
  Pseudo-equivalences between complete objects are equivalences in $\PCS(\DFV)$.
\end{corollary}
\begin{proof}
  By the previous proposition, pseudo-equivalences in $\PCS(\DFV)$ are
  fully faithful and essentially surjective, and Proposition~\ref{prop
    ffes are equ} implies that such maps between complete objects are
  equivalences in $\PCS(\DFV)$.
\end{proof}
Similarly to \cite[Lemma 5.5.7]{enriched}, we have the following result:
\begin{lemma}\label{lem OC to OD is cat equ}
For every object $\xxO\in \PSh_{\xCS}(\DFV)$ and every pseudo-equivalence $f\colon \xcc\to \xdd$ in $\PSh_{\Seg}(\simp)$, the cotensor product $\xxO^f\colon \xxO^\xdd\to \xxO^\xcc$ is a pseudo-equivalence in $\PSh_{\xCS}(\DFV)$.
\end{lemma}
\begin{proof}
  If $(f,g,\phi,\psi)$ is a pseudo-equivalence datum associated to $f$, then the natural equivalences $\phi\colon \xcc\times E^1\to \xcc$ and $\psi\colon \xdd\times E^1\to \xdd$ in $\PSh_{\Seg}(\simp)$ induce maps $\xxO^\xcc\to \xxO^{\xcc\times E^1}\simeq (\xxO^\xcc)^{E^1}$ and $\xxO^\xdd\to \xxO^{\xdd\times E^1}\simeq (\xxO^\xdd)^{E^1}$ in $\PSh_{\xCS}(\DFV)$, respectively. If  
  $\phi^f\colon \xxO^\xcc\otimes E^1\to \xxO^\xcc$ and $\psi^f\colon \xxO^\xdd\otimes E^1\to \xxO^\xdd$ denote the corresponding adjoint maps, then $\phi^f$ and $\psi^f$ are natural equivalences and one readily checks that $(\xxO^f,\xxO^g, \phi^{f},\psi^f)$ is a pseudo-equivalence datum associated to $\xxO^f$.
\end{proof}
\begin{lemma}\label{lem Os0 is equ}
  For every complete object $\xxO\in\PCCS(\DFV)$, the map $\xxO^{s_0}\colon \xxO\simeq \xxO^{E^0}\to \xxO^{E^1}$ induced by the unique map $s^0\colon E^1\to E^0$ in $\PSh_{\Seg}(\simp)$ is an equivalence.
\end{lemma}
\begin{proof}
  By \cite[Definition 5.5.6]{enriched}, the map $s_0$ is a
        pseudo-equivalence in $\PSh_{\Seg}(\simp)$ and Lemma~\ref{lem
          OC to OD is cat equ} therefore implies that $\xxO^{s_0}$ is a pseudo-equivalence in $\PSh_{\xCS}(\DFV)$. Since $\xxO^{E^0}\simeq \xxO$ is complete by assumption, it suffices to show that $\xxO^{E^1}$ is also complete by Corollary~\ref{cor cat equ betw cp obj are equ}.
  
  The adjunctions $\overline{u}_!\dashv\overline{u}^*$ and $\blank\otimes E^1\dashv(\blank)^{E^1}$ provide a chain of equivalences  
\[\begin{split}
\iota_0(\xxO^{E^1})=\xMap(E^0,\overline{u}^*(\xxO^{E^1}))\simeq \xMap(\overline{u}_!(E^0\otimes E^1),\xxO)\\
\simeq \xMap(\overline{u}_!(E^1),\xxO)\simeq\xMap(E^1,\overline{u}^*(\xxO)).
\end{split}
\]
  Similarly, we have the equivalence $\iota_1(\xxO^{E^1})\simeq
        \xMap(E^1\otimes E^1,\overline{u}^*(\xxO))$. It follows that the map
        $\iota_0(\xxO^{E^1})\to \iota_1(\xxO^{E^1})$ can be
        identified with $\xMap(\id_{E^1}\otimes
        s^0,\overline{u}^*(\xxO))\colon\xMap(E^1,\overline{u}^*(\xxO))\to \xMap(E^1\otimes
        E^1,\overline{u}^*(\xxO))$. Since $\xxO$ is complete and
        $\id_{E^1}\otimes s^0$ is a local equivalence by \cite[Lemma 5.4.7]{enriched}, the map $\iota_0(\xxO^{E^1})\to \iota_1(\xxO^{E^1})$ is an equivalence, and hence $\xxO^{E^1}$ is complete.
\end{proof}
\begin{lemma}\label{lem id otimes s0 is loc equ}
  For every object $\xxO\in \PSh_{\xCS}(\DFV)$, the map $\id_\xxO\otimes s^0\colon \xxO\otimes E^1\to \xxO\otimes E^0\simeq \xxO$ induced by $s^0\colon E^1\to E^0$ is a local equivalence. 
\end{lemma}
\begin{proof}
The map $\id_\xxO\otimes s^0$ is a local equivalence if and only if the induced map $$(\id_\xxO\otimes s^0)^*\colon\xMap(\xxO,\xxP)\to \xMap(\xxO\otimes E^1,\xxP)$$ is an equivalence for every complete object $\xxP$. By adjunction, this is equivalent to requiring $\xMap(\xxO,\xxP)\to \xMap(\xxO,\xxP^{E^1})$ to be an equivalence, which is true for every complete object $\xxP$ by the previous lemma.
\end{proof}

\begin{proposition}\label{prop pseudoequ are loc equ}
  Every pseudo-equivalence is a local equivalence.
\end{proposition}
\begin{proof}
  Let $f\colon \xxO\to \xxP$ be a pseudo-equivalence in $\PSh_{\xCS}(\DFV)$ and let $(f,g,\phi,\psi)$ be a corresponding pseudo-equivalence datum. We want to show that the map $f^*\colon \xMap(\xxP,\xxQ )\to \xMap(\xxO, \xxQ)$ is an equivalence for every complete object $\xxQ$. It follows from Definition~\ref{def pseudoequivalence} that 
\[
  f^*g^*\simeq (\id\otimes d^0)^* \phi^* \text{ and }
  \id^*\simeq (\id\otimes d^1)^* \phi^* .
\]
  Similarly, we have
\[
  g^* f^*\simeq (\id\otimes d^1)^* \psi^*\text{ and }\id^*\simeq (\id\otimes d^0)^* \psi^*.
\]
  Hence, we only need to show that the morphisms $(\id\otimes d^0)^*$ and $(\id\otimes d^1)^*$ are equivalent in $\xS$.
  
        For every $\xxO'\in \PSh_{\xCS}(\DFV)$, the map  \[
  (\id\otimes s^0)^*\circ(\id\otimes d^i)^*\colon \xMap(\xxO',\xxQ)\to \xMap(\xxO',\xxQ)
\]
  is equivalent to the identity. Since $(\id\otimes s^0)$ is a local equivalence by Lemma~\ref{lem id otimes s0 is loc equ}, the map $(\id\otimes s^0)^*$ is an equivalence for every complete object $\xxQ$. Therefore, the maps $(\id\otimes d^0)^*$ and $(\id\otimes d^1)^*$ are equivalent, because both are right inverses of the equivalence $(\id\otimes s^0)^*$.
\end{proof}

\begin{proposition}\label{prop tensor}
  The tensor product
  $\otimes\colon \PSh_{\xCS}({\DF^\xV})\times \PSh_{\Seg}(\simp)\to
  \PSh_{\xCS}({\DF^\xV})$
  of Corollary~\ref{cor:PCStensor} induces a tensor product on the
  complete objects
  \[ \otimes\colon \OpdIV \times \xCat_\infty\to \OpdIV\]
  (i.e.\ $\PCCS(\DFV) \times \PCoS(\simp) \to \PCCS(\DFV)$),
  which preserves colimits in each variable.
\end{proposition}
\begin{proof}
  It suffices to show that for $\mathcal{O} \in \PCS(\DFV)$,
  the map $\mathcal{O} \otimes E^{1} \to \mathcal{O} \otimes E^{0}$ is
  a local equivalence, and for $\mathcal{C} \in \PSeg(\simp)$ the map
  $\overline{u}_{!}E^{1} \otimes \mathcal{C} \to \overline{u}_{!}E^{0} \otimes \mathcal{C}$
  is a local equivalence.

  By adjunction, the first claim is equivalent to
  $ \mathcal{P} \simeq \mathcal{P}^{E^0}\to \mathcal{P}^{E^1}$ being
  an equivalence for every complete object $\mathcal{P}$, which
  follows from Lemma~\ref{lem Os0 is equ}.

  By Corollary~\ref{cor:PCSu}(ii) we have equivalences
  $\overline{u}_!(E^1_\xV)\otimes \xcc\simeq \overline{u}_!(E^1_\xV\otimes \xcc)$ and
  $\overline{u}_!(E^0_\xV)\otimes \xcc\simeq \overline{u}_!(E^0_\xV\otimes \xcc)$. The map
  $E^1_\xV\otimes \xcc\to E^0_\xV\otimes \xcc$ is then a local
  equivalence in $\PSh_\Seg(\simp)$ by \cite[Proposition 5.5.9]{enriched} and the claim
  follows from the fact that $\overline{u}_!$ obviously preserves local
  equivalences.
\end{proof}

Applying the adjoint functor theorem, we get:
\begin{corollary}
  The tensor product
  $\otimes\colon \OpdIV \times \CatI \to \OpdIV$ induces 
  \[\xAlg^\xV_{(\blank)}(\blank)\colon (\OpdIV)^\op\times \OpdIV \to
  \CatI\] such that
  \[ \Map_{\CatI}(\mathcal{X}, \Alg^{\mathcal{V}}_{\mathcal{O}}(\mathcal{P}))
  \simeq \Map_{\OpdIV}(\mathcal{O} \otimes \mathcal{X},
  \mathcal{P})\]
  and a cotensor product 
  \[(\blank)^{(\blank)}\colon
  \OpdIV \times \CatI^\op\to \OpdIV\]
  such that
  \[\Map_{\OpdIV}(\mathcal{O}, \mathcal{P}^{\mathcal{X}}) \simeq
  \Map_{\OpdIV}(\mathcal{O} \otimes \mathcal{X}, \mathcal{P}).\]
  Moreover, both of these functors preserve limits in each variable.
\end{corollary}

\subsection{Completion}\label{subsec compl}
Our goal in this subsection is to prove that $\PCCS(\DFV)$ is the
localization of $\PCS(\DFV)$ at the fully faithful and essentially
surjective morphisms. We do this by the strategy introduced by
Rezk~\cite{Rezk} and applied to enriched \icats{} in
\cite[5.5]{enriched}: We define a functor that takes every continuous
Segal presheaf $\mathcal{O}$ to a complete object
$\widehat{\mathcal{O}}$ with a natural map
$\mathcal{O} \to \widehat{\mathcal{O}}$, and check that this map is
both a local equivalence and fully faithful and essentially
surjective. This functor is defined as follows:
\begin{definition}\label{def completion}
  Given an object $\xxO\in\PSh_{\xCS}(\DFV)$, we write
  $\widehat{\xxO}$ for the colimit in $\PCS(\DFV)$ of the simplicial
  object $ \xxO^{E^\bullet}$, and
  $l_{\mathcal{O}} \colon \mathcal{O} \to \widehat{\mathcal{O}}$ for
  the natural map from $\mathcal{O} \simeq \mathcal{O}^{E^{0}}$ to
  this colimit.
\end{definition}

For the rest of this section we choose a regular cardinal $\kappa$
such that $\PCS(\DFV) \simeq \PkSeg(\DFVk)$. The key observation that
makes the proof work is that the colimit $\widehat{\mathcal{O}}$ can
be computed in presheaves on $\DFVk$:
\begin{propn}\label{propn F Segal}
  For $\xxO \in \PkSeg(\DFVk)$, the geometric realization
  $\widehat{\xxO} := | \xxO^{E^{\bullet}}| $, computed in $\PSh(\DFVk)$, is a
  $\kappa$-continuous Segal presheaf.
\end{propn}
\begin{proof}
  For every object $(\overline\xfc_n)\in \DFVk$ lying over a corolla
  and every morphism $[m]\to [n]$ in $\simp^\op$, there is a
  commutative diagram
  \nolabelcsquare{\xxO^{E^m}(\overline\xfc_n)}{\xxO^{E^n}(\overline\xfc_n)}{\xxO^{E^m}(\overline{\xfe})^{k+1}
  }{ \xxO^{E^n}(\overline{\xfe})^{k+1}.}  Since $E^n\to E^m$ is a
  pseudo-equivalence by \cite[Corollary 5.5.6]{enriched},
  Lemma~\ref{lem OC to OD is cat equ} and Proposition~\ref{prop
    pseudoequ are ffes} imply that the functor
  $ \xxO^{E^m}\to \xxO^{E^n}$ is fully faithful, which by definition
  means this commutative square is Cartesian. In other words, the
  natural transformation
  $\tau\colon \xxO^{E^\bullet} (\overline\xfc_n)\to \xxO^{E^\bullet}
  (\overline{\xfe})^{k+1}$ between the two simplicial diagrams is
  Cartesian in the sense of \cite[Definition 6.1.3.1]{ht}.  Since
  $\xS$ is an $\infty$-topos, by \cite[Theorem 6.1.3.9]{ht} the
  commutative square
  \nolabelcsquare{\xxO^{E^0}(\overline{\xfc}_n)}{{|}\xxO^{E^\bullet}(\overline{\xfc}_n){|}}{\xxO^{E^0}(\overline{\xfe})^{k+1}}{{|}\xxO^{E^\bullet}
    (\overline{\xfe})^{k+1}{|}} is also Cartesian. The surjectivity of
  the bottom horizontal map on connected components and the pullback
  condition then imply that each fibre of
  $ | \xxO^{E^\bullet} (\overline\xfc_n)| \to | \xxO^{E^\bullet}
  (\overline{\xfe})^{k+1}| $ is equivalent to one of
  $ \xxO^{E^0}(\overline\xfc_n)\to \xxO^{E^0}(\overline{\xfe})^{k+1}$.
	
  Using this we first check that $\widehat{\mathcal{O}}$ is
  $\kappa$-continuous (in the sense of
  Definition~\ref{defn:kappacts}), starting with condition (2). To see
  that the map $|\mathcal{O}^{E^{\bullet}}(\xfc_{n}(\emptyset))| \to
  |\mathcal{O}^{E^{\bullet}}(\mathfrak{e})|^{\times (n+1)}$, we
  observe that this factors as
  \[ |\mathcal{O}^{E^{\bullet}}(\xfc_{n}(\emptyset))| \to
    |\mathcal{O}^{E^{\bullet}}(\mathfrak{e})^{\times (n+1)}| \to
    |\mathcal{O}^{E^{\bullet}}(\mathfrak{e})|^{\times (n+1)},\] where
  the first map is an equivalence since it is a colimit of
  equivalences (as $\mathcal{O}^{E^{n}}$ is a $\kappa$-continuous
  Segal presheaf for each $n$) and the second map is an equivalence
  since simplicial colimits commute with products.

  We now check condition (1) from Definition~\ref{defn:kappacts}. Fix a $\kappa$-small weakly
  contractible diagram $q\colon K \to \mathcal{V}^{\kappa}$.  We want
  to show that the canonical map
  $\widehat{\mathcal{O}}(\mathfrak{c}_{k}, \colim_{K} q) \to \lim_{K}
  \widehat{\mathcal{O}}(\mathfrak{c}_{k}, q) $ is an equivalence in
  $\xS_{/ | \xxO^{E^\bullet}(\xfe)^{k+1}| }$. To see this, it suffices
  to verify that the front square in the following commutative diagram
  in $\xS$
\[
\begin{tikzcd}
{}	&   \xxO^{E^0}(\mathfrak{c}_{k}, \colim_{K} q)  \ar[rr] \ar[dd]\ar[ld]  &  &   \lim_{K} \xxO^{E^0}(\mathfrak{c}_{k}, q)\ar[dd] \ar[ld] \\
{|}\xxO^{E^\bullet}(\mathfrak{c}_{k}, \colim_{K} q){|} \ar[rr, crossing
over] \ar[dd]  &  &  \lim_{K}  {|} \xxO^{E^\bullet} (\mathfrak{c}_{k}, q) {|} \ar[dd]   &  \\ 
	&  	\xxO^{E^0}(\xfe)^{k+1}\ar[rr,"\id" near start] \ar[ld] &  & \xxO^{E^0}(\xfe)^{k+1} \ar[ld]  \\
	{|}   \xxO^{E^\bullet} (\xfe)^{k+1}{|}  \ar[rr,"\id"]   &  &       {|} \xxO^{E^\bullet} (\xfe)^{k+1} {|} \ar[from=uu, crossing over], &\\
\end{tikzcd}\]
is a pullback square. The horizontal maps of the back square are equivalences by the
assumption that $ \xxO^{E^0}(\simeq \xxO)$ is a continuous Segal
presheaf. We also saw above that the square on the left side is
Cartesian. Moreover, since $K$ is weakly contractible (so the limit
cone on a constant diagram is constant) and limits commute with
pullbacks, the square on the right side is also Cartesian. This
implies that the front square is a pullback too, and so
$| \xxO^{E^\bullet}| $ satisfies the required condition. The Segal
condition for $| \xxO^{E^{\bullet}}|$ holds by a completely analogous
argument, so $| \xxO^{E^{\bullet}}| $ is indeed a $\kappa$-continuous
Segal presheaf.
\end{proof}

\begin{cor}\label{cor u* complete}
  For every $\xxO\in \PCS(\DFV)$, the canonical map
  \[\widehat{\overline{u}^{*}\mathcal{O}} \to
  \overline{u}^*\widehat{\mathcal{O}} \]
  is an equivalence, where by $\widehat{\overline{u}^{*}\mathcal{O}}$
  we mean the colimit $|(\overline{u}^{*}\mathcal{O})^{E^{\bullet}}|$
  in $\PCS(\simp^{\mathcal{V}})$.
\end{cor}
\begin{proof}
  Since there is a natural equivalence
  $\overline{u}^{*}(\mathcal{O}^{E^{\bullet}}) \simeq
  (\overline{u}^{*}\mathcal{O})^{E^{\bullet}}$ by
  Corollary~\ref{cor:PCSu}, there is a natural map from the colimit
  $\widehat{\overline{u}^{*}\mathcal{O}} \simeq
  |(\overline{u}^{*}\mathcal{O})^{E^{\bullet}}|$ to
  $\overline{u}^{*}\widehat{\mathcal{O}}$.

  Let $F$ denote the colimit of $\mathcal{O}^{E^{\bullet}}$ in
  $\PSh(\DFVk)$. Then $\widehat{\mathcal{O}}$ is the localization of $F$ at the
  $\kappa$-continuous Segal equivalences, but by Proposition~\ref{propn F Segal} the presheaf $F$
  is already a $\kappa$-continuous Segal presheaf, and so
  $\widehat{\mathcal{O}} \simeq F$. The functor $\overline{u}^{*}
  \colon \PSh(\DFVk) \to \PSh(\simp^{\mathcal{V}^{\kappa}})$ preserves
  colimits, so $\overline{u}^{*}F$ is the colimit of
  $\overline{u}^{*}(\mathcal{O}^{E^{\bullet}}) \simeq
  (\overline{u}^{*}\mathcal{O})^{E^{\bullet}}$. Moreover,
  $\overline{u}^{*}$ preserves $\kappa$-continuous Segal presheaves by
  Corollary~\ref{cor:PSu}, and so
  $\overline{u}^{*}F$ is a $\kappa$-continuous Segal presheaf. Thus $
  \overline{u}^{*}F$ is the colimit
  $\widehat{\overline{u}^{*}\mathcal{O}}$ in
  $\PCS(\simp^{\mathcal{V}})$, as required.
\end{proof}

\begin{cor}\label{cor:Ohatcomplete}
  For any $\mathcal{O}$ in $\PCS(\DFV)$, the object
  $\widehat{\mathcal{O}}$ is complete and the map $l_{\mathcal{O}}
  \colon \mathcal{O} \to \widehat{\mathcal{O}}$ is a local
  equivalence.
\end{cor}
\begin{proof}
  By Corollary~\ref{cor u* complete} the underlying enriched \icat{}
  $\overline{u}^{*}\widehat{\mathcal{O}}$ is equivalent to
  $\widehat{\overline{u}^{*}\mathcal{O}}$, which is complete by 
  \cite[Theorem 5.6.2]{enriched}; by definition, this means
  $\widehat{\mathcal{O}}$ is also complete.

  As the class of local equivalences is strongly saturated, it is
  closed under colimits of morphisms. Therefore, to see that $l_{\mathcal{O}}$ is a
  local equivalence it suffices to show that for every map
  $[m] \to [n]$ in $\simp$ the map $\xxO^{E^n}\to \xxO^{E^m}$ is a
  local equivalence. The map $E^{m}\to E^n$ is a pseudo-equivalence in
  $\PSeg(\simp)$ by \cite[Corollary 5.5.6]{enriched}. Then
  Lemma~\ref{lem OC to OD is cat equ} implies that
  $\xxO^{E^n}\to \xxO^{E^m}$ is a pseudo-equivalence and so a local
  equivalence by Proposition~\ref{prop pseudoequ are loc equ}.
\end{proof}

\begin{cor}\label{cor Ohat adj}
  The functor $\widehat{(\blank)} \colon \PCS(\DFV) \to \PCCS(\DFV)$
  is left adjoint to the inclusion $\PCCS(\DFV) \hookrightarrow \PCS(\DFV)$.
\end{cor}
\begin{proof}
  By \cite[Proposition 5.2.7.8]{ht}, the functor $\widehat{(\blank)}$
  is left adjoint to the inclusion
  $\PCCS(\DFV)\subseteq\PSh_{\xCS}(\DFV)$ if and only if for every
  $\xxO\in \PSh_{\xCS}(\DFV)$ and every $\xxP\in\PCCS(\DFV)$ the
  canonical map $l_{\xxO} \colon \mathcal{O}  \to \widehat{\mathcal{O}}$
  induces an equivalence
  \[\xMap_{\PCCS(\DFV)}(\widehat{\xxO}, \xxP)\to \xMap_{\PSh_{\xCS}(\DFV)}(\xxO,\xxP).\]
  Since $\xxP$ is complete, this follows from $l_{\mathcal{O}}$ being
  a local equivalence, which we just saw in Corollary~\ref{cor:Ohatcomplete}.
\end{proof}

\begin{propn}\label{propn lO ffes}
  The map $l_{\mathcal{O}} \colon \mathcal{O} \to
  \widehat{\mathcal{O}}$ is fully faithful and essentially surjective
  for all $\mathcal{O} \in \PCS(\DFV)$.
\end{propn}
\begin{proof}
  For the essential surjectivity we only need to check that the map
  $\overline{u}^*(l_\xxO)\colon \overline{u}^*(\xxO)\to
  \overline{u}^*(\widehat{\xxO})$
  of the underlying enriched $\infty$-categories is essentially
  surjective. This easily follows from Corollary~\ref{cor u* complete} and \cite[Theorem 5.6.2]{enriched}. To see that the
  map is fully faithful, recall from the proof of
  Proposition~\ref{propn F Segal} that we have a Cartesian square
  \nolabelcsquare{\xxO^{E^0}(\overline\xfc_n)}{ {|} \xxO^{E^\bullet}
    (\overline\xfc_n){|}}{\xxO^{E^0}(\xfe)^{n+1}}{{|} \xxO^{E^\bullet}
    (\xfe)^{n+1}{|}.}

  By Proposition~\ref{propn F Segal} (and the fact that geometric
  realization commutes with finite products in $\mathcal{S}$) this
  says that the commutative square
  \nolabelcsquare{\xxO(\overline\xfc_n)}{\widehat{\xxO}(\overline\xfc_n)}{\xxO(\mathfrak{e})^{n+1}}{\widehat{\xxO}(\mathfrak{e})^{n+1}}
  is Cartesian. In other words, $\xxO \to \widehat{\xxO}$ is fully
  faithful.
\end{proof}

Putting together our results so far, we can now prove the main result
of this section:
\begin{theorem}\label{theo cp obj are loc}
  A map in $\PSh_{\xCS}(\DFV)$ is a local equivalence if and only if
  it is fully faithful and essentially surjective.
\end{theorem}
\begin{proof}
  Every map $f\colon \xxO\to \xxP$ in $\PSh_{\xCS}(\DFV)$ gives a
  commutative diagram \csquare{\xxO}{\xxP}{\widehat{\xxO}
  }{\widehat{\xxP} }{f}{l_\xxO}{l_\xxP}{\widehat{f}} where the
  vertical maps are local equivalences by Corollary~\ref{cor:Ohatcomplete} as well as fully faithful and essentially surjective
  by Proposition~\ref{propn lO ffes}. Since fully faithful and
  essentially surjective maps satisfy the $2$-of-$3$ property by
  Propositon~\ref{prop 2 of 3}, the map $f$ is fully faithful and
  essentially surjective if and only if $\widehat{f}$ is
  so. Corollary~\ref{cor:Ohatcomplete} implies that the map
  $\widehat{f}$ is a map between complete objects and so by
  Proposition~\ref{prop ffes are equ} it is fully faithful if and only
  if it is an equivalence. Similarly, using the 2-of-3 property for
  local equivalences we see that $f$ is a local
  equivalence \IFF{} $\widehat{f}$ is an equivalence. Thus the map
  $f$ is a local equivalence \IFF{} it is fully faithful and
  essentially surjective.
\end{proof}

From Corollary~\ref{cor Ohat adj} and Theorem~\ref{theo cp obj are
  loc} we immediately get:
\begin{corollary}\label{cor cp obj are loc}
  The adjunction 
  \[ \widehat{(\blank)} : \PCS(\DFV) \rightleftarrows \PCCS(\DFV)
  \]
  (where the right adjoint is the inclusion) exhibits $\PCCS(\DFV)$ as
  the localization of $\PCS(\DFV)$ with respect to the class of fully
  faithful and essentially surjective morphisms.
\end{corollary}

\begin{remark}
  Suppose $\mathcal{V}$ is a large symmetric monoidal \icat{}, not
  necessarily presentable, and let $\OpdIV$ denote the full
  subcategory of $\AlgDFS(\mathcal{V})$ spanned by the complete
  objects. By embedding $\mathcal{V}$ in a presentably symmetric
  monoidal \icat{} in a larger universe, it follows by exactly the
  same argument as in the proof of \cite[Theorem 5.6.6]{enriched} that
  the inclusion $\OpdIV \hookrightarrow \AlgDFS(\mathcal{V})$ has a
  left adjoint that exhibits $\OpdIV$ as the localization at the fully
  faithful and essentially surjective morphisms.
\end{remark}

\begin{propn}
  The \icat{} $\OpdIV$ is functorial in $\mathcal{V}$ with respect to
  lax symmetric monoidal functors. Moreover, if $F \colon \mathcal{V}
  \to \mathcal{W}$ is a colimit-preserving symmetric monoidal functor
  then $F_{*} \colon \OpdIV \to \OpdI^{\mathcal{W}}$ preserves
  colimits; thus $\OpdI^{(\blank)}$ gives a functor
  $\name{CAlg}(\PrL) \to \PrL$.
\end{propn}
\begin{proof}
The previous remark and \cite[Proposition 5.7.4]{enriched} imply that the functor $F_{*}\colon \AlgDFS(\mathcal{V}) \to \AlgDFS(\mathcal{W})$ induced by a lax monoidal
functor $F \colon \mathcal{V} \to \mathcal{W}$ gives a functor $\OpdIV \to \OpdI^{\mathcal{W}}$ if $F_{*}$
  preserves fully faithful and essentially surjective functors. This
  can be proven analogously to \cite[Lemma 5.7.5]{enriched}. The
  second claim follows from an argument similar to that used in the proof of \cite[Lemma 5.7.7]{enriched}.
\end{proof}

\section{Enriched $\infty$-Operads as Dendroidal Segal
  Presheaves}\label{sec dendr}

In this section we consider an enriched version of the dendroidal
Segal spaces of Cisinski--Moerdijk; our main result is that this
approach is equivalent to that using $\DF$ we have discussed so
far. We will prove this using a variant of the argument in the
unenriched case from \cite{ChuHaugsengHeuts}.

We begin in \S\ref{sec dendroidal cat} by briefly reviewing the
definiton and basic properties of the dendroidal category $\bbO$. In
\S\ref{subsec: Seg OV} we then introduce \icats{} $\bbOV$ for
$\mathcal{V}$ a symmetric monoidal \icat{} and define (continuous)
Segal presheaves on $\bbOV$; we also discuss the dendroidal analogues
of many of the results from \S\ref{sec segal presheaves}. As a
preliminary to the comparison result, in \S\ref{subsec DFi} we observe
that we can replace the \icat{} $\DFV$ by a full subcategory $\DFiV$,
before proving the comparison in \S\ref{sec preparation for
  comparison}.

\subsection{The Dendroidal Category}\label{sec dendroidal cat}
The dendroidal category $\bbO$ was first introduced by Moerdijk and Weiss
in \cite{MoerdijkWeiss} as a category of trees whose morphisms are given by
maps of free operads. Here we recall a more combinatorial
reformulation of this definition due to Kock \cite{Kock}.

  \begin{definition}\label{defn:tree}
    A \emph{polynomial endofunctor} is a diagram of sets
    \[ T_{0} \overset{s}{\leftarrow} T_{2} \overset{p}{\to} T_{1} \overset{t}{\to} T_{0}.\]
    We call a polynomial endofunctor as above a \emph{tree} if the following conditions are satisfied:
    \begin{itemize}
      \item The sets $T_{i}$ are all finite.
      \item The function $t$ is injective.
      \item The function $s$ is injective and there is a unique element $r$ called the \emph{root} in the complement of its image.
      \item Define a successor function $\sigma \colon T_{0} \to T_{0}$ by
      $\sigma(r) = r$ and $\sigma(e) = t(p(e))$ for $e\in s(T_{2})$. Then for every $e$ there exists some $k\geq 0$ such that
      $\sigma^{k}(e) = r$.
    \end{itemize}
  \end{definition}
  \begin{defn}
          Let $T$ be a tree and let $e,e'\in T_0$. We say $e$ and $e'$
          are \emph{comparable}, if there is some $k\geq 0$ such that
          either $\sigma^k(e)=e'$ or $\sigma^k(e')=e$, and
          \emph{incomparable} otherwise.
  \end{defn}
  \begin{remark}
          The intuition behind this notion of a ``tree'' is as
          follows: We interpret $T_{0}$ as the set of edges of the
          tree, $T_{1}$ as the set of vertices, and $T_{2}$ as the set
          of pairs $(v, e)$ where $e$ is an incoming edge of $v$. The
          function $s$ is the projection $s(v,e) = e$, the function
          $p$ is the projection $p(v,e) = v$, and the function $t$
          assigns to each vertex its unique outgoing edge.
  \end{remark}
  \begin{definition}\label{def leaf}
          For a tree $T$ given by
          $ T_{0} \overset{s}{\leftarrow} T_{2} \overset{p}{\to} T_{1}
          \overset{t}{\to} T_{0}$,
          we call an edge $e\in T_0$ a \emph{leaf} if it does not lie
          in the image of $t$, and an \emph{inner edge} if it lies in
          the image of $t$ and $e\neq r$.
  \end{definition}
  
  \begin{remark}
    The name ``polynomial endofunctor'' comes from the fact that such a
    diagram induces an endofunctor of $\xSet_{/X_{0}}$ given by
    $t_{!}p_{*}s^{*}$. We refer the reader to \cite{Kock} for a more thorough
    discussion of this.
  \end{remark}
  
  \begin{definition}\label{def morphisms in Omegain}
    A morphism of polynomial endofunctors $f \colon X \to Y$ is a
    commutative diagram
    \[\begin{tikzcd}
      X_{0} \ar{d}{f_{0}}& X_{2} \ar{l} \ar{r}
                        \ar{d}{f_{2}} \arrow[phantom]{dr}[very near start]{\ulcorner}& X_{1} \ar{r}\ar{d}{f_{1}} & X_{0} \ar{d}{f_{0}}\\
      Y_{0} & Y_{2} \ar{l} \ar{r} & Y_{1} \ar{r} & Y_{0}
    \end{tikzcd}
    \]
  
    such that the middle square is Cartesian. We write $\bbO_\xint$
    for the category of trees and  morphisms of polynomial endofunctors
    between them; we will refer to these morphisms as the \emph{inert} morphisms
    between trees, or as \emph{embeddings} of subtrees.  
  \end{definition}
  
  \begin{remark}
    By \cite[Proposition 1.1.3]{Kock} every morphism of polynomial
    endofunctors between trees is injective, which justifies calling
    these morphisms embeddings.
  \end{remark}
  
\begin{definition}
  A tree $T$ is called a \emph{corolla} if it has only one vertex,
  i.e. $T_1$ is a one-element set.  For $n\leq 0$, we write $C_n$
  for the corolla given by
  $$\mathbf{n+1} \hookleftarrow \mathbf{n} \to \{n+1\}
  \hookrightarrow \mathbf{n+1}.$$
  More generally, for a finite set $A$ we let $C_A$ denote the corolla
  \[ A_{+} \hookleftarrow A \to * \hookrightarrow A_{+}\]
  (which is of course isomorphic to $C_{|A|}$).
  We write $\eta$ for the
  \emph{edge}, namely the trivial tree
  \[\mathbf{1} \hookleftarrow \mathbf{0} \to \mathbf{0}
  \hookrightarrow \mathbf{1}.\]
\end{definition}
  
\begin{definition}\label{def Omegael}
  We define $\bbOel$ to be the full subcategory of $\bbOint$
  spanned by the corollas $C_n$, $n\geq 0$, and the edge $\eta$. 
  For a tree $T$ we write $\bbOelT$ for the pullback $\bbOel
  \times_{\bbOint} \bbOintT$.
\end{definition}

\begin{definition}
  For a tree $T$, we write $\text{sub}(T)$ for the set of subtrees of
  $T$, meaning the set of morphisms $T' \to T$ in $\bbO_{\xint}$, and
  we write $\text{sub}'(T)$ for the set of subtrees of $T$ with a
  marked leaf, meaning the set of pairs of morphisms
  $(\eta\to T', T' \to T)$. We then write $\overline{T}$ for the polynomial
  endofunctor 
  \[ T_{0} \leftarrow \text{sub}'(T) \to \text{sub}(T) \to T_{0},\]
  where the first map sends a marked subtree to its marked edge, the
  second is the obvious projection, and the third sends a subtree to
  its root.
\end{definition}
  
\begin{definition}\label{def morphism in Omega}
  The dendroidal category $\bbO$ has trees as objects and the
  morphisms of polynomial endofunctors $\overline{T} \to \overline{T}'$ as
  morphisms from $T$ to $ T'$.
\end{definition}
  
\begin{remark}\label{rem kleisli}
  By \cite[Corollary 1.2.10]{Kock}, the polynomial endofunctor
  $\overline{T}$ is in fact the free polynomial monad generated by $T$, and
  the category $\bbO$ is a full subcategory of the Kleisli category of
  the monad for free polynomial monads. This means that a morphism
  $T \to S$ in $\bbO$ can be identified with a map of polynomial
  endofunctors $T \to \overline{S}$. It follows that
  $\bbOint$ is a subcategory of $\bbO$.
\end{remark}
\begin{defn}
  A map $T \to T'$ in $\bbO$ is \emph{active} if it takes the leaves
  of $T$ to the leaves of $T'$ (bijectively) and the root of $T$ to
  the root of $T'$.
\end{defn}

\begin{remark}
  By \cite[Proposition 1.3.13]{Kock} the inert and active
  morphisms form a factorization system on $\bbO$.
\end{remark}

\begin{remark}\label{rem u}
  There is a fully faithful functor $u_{\bbO} \colon \simp \to \bbO$
  which can viewed as the embedding of the full subcategory of linear
  trees into $\bbO$. More precisely, the functor $u$ takes an object
  $[n]\in \simp$ to the tree
  \[\mathbf{n+1} \overset{s}{\hookleftarrow} \{2,\ldots,n+1\}
  \overset{p}{\to} \{1,\ldots, n\} \overset{t}{\hookrightarrow}
  \mathbf{n+1},\]
  where $s,t$ are canonical inclusions and $p(i)=i-1$ for
  $1\leq i\leq n$.
\end{remark}

\begin{defn}\label{defn:CorO}
  The functor $\CorO \colon \bbO^{\op} \to \Fin_{*}$ takes an object
  $T$ corresponding to a diagram
  \[ T_{0} \xfrom{s} T_{2} \overset{p}{\to} T_{1}
  \xto{t} T_{0} \]
 to $T_{1,+} \in \Fin_{*}$ and a morphism $f \colon T \to T'$ in $\bbO$
 corresponding to a diagram
\[
\begin{tikzcd}
T_{0} \arrow{d}{f_{0}} & T_{2} \arrow[phantom]{dr}[very near
start]{\ulcorner} \arrow{l}\arrow{r}\arrow{d}{f_{2}} &
T_{1}\arrow{r}\arrow{d}{f_{1}} & T_{0} \arrow{d}{f_{0}} \\
T'_{0} & \name{sub}'(T')\arrow{l}\arrow{r} & \name{sub}(T') \arrow{r} & T'_{0}
\end{tikzcd}
\]
to the morphism $\CorO(f) \colon T'_{1,+} \to T_{1,+}$ defined by 
\[
\CorO(f)(x) = 
\begin{cases}
  y, & x \text{ is a vertex of the subtree } f_{1}(y),\\
  *, & x \text{ is not in the image of } f.
\end{cases}
\]
\end{defn}

The following observation shows that the functor $\CorO$ is
well-defined:
\begin{lemma}
If $\phi\colon T\to S$ is a map in $\bbO$ and $t\neq t'\in T_1$, then $\phi(t)$ and $\phi(t')$ are two subtrees of $S$ with disjoint sets of vertices.
\end{lemma}
\begin{proof}
The map $\phi$ is given by a morphism of polynomial endofunctors of the form 
\[
\begin{tikzcd}
T_{0} \ar{d}& \text {sub}'(T) \ar{l} \ar{r}
\ar{d} \arrow[phantom]{dr}[very near start]{\ulcorner}& \text {sub}(T) \ar{r}\ar{d} & T_{0} \ar{d}\\
S_{0} & \text {sub}'(S) \ar{l} \ar{r} & \text {sub}(S) \ar{r} & S_{0}.
\end{tikzcd}
\]
By identifying $T_1$ with the subset of $\text {sub}(T)$ consisting of
corollas whose roots are given by elements in $T_1$, we regard
$\phi(t)$ and $\phi(t')$ as two subtrees in $S$. Let us assume that
there exists a corolla $C_n$ lying in the intersection of $\phi(t)$
and $\phi(t')$. Without loss of generality we can assume that there
exists a path from a leaf $l'$ of the subtree $\phi(t')$ to the root
$r'$ of $\phi(t')$ which passes through the root of $C_n$ and the root
$t$ of $\phi(t)$. In particular, this path induces a linear ordering
$l'_S\leq r\leq r'$. Since the middle square in the diagram is
Cartesian, $\phi$ carries the leaves of $t'$ to leaves of
$\phi(t')$. Therefore, the path in $T$ going from $t$ to $t'$
necessarily passes through a unique leaf $l'_T$ of the corolla
corresponding to $t'$. We then obtain another linear ordering
$t\leq l'_T\leq t'$. As this contradicts \cite[Proposition
1.3.7]{Kock}, the subtrees $\phi(t)$ and $\phi(t')$ necessarily have
disjoint sets of corollas.
\end{proof}

\begin{lemma}
  The functor $\CorO$ preserves active-inert factorizations.
\end{lemma}
\begin{proof}
  If $f \colon T \to T'$ in $\bbO$ is inert, then the subtree
  $f_{1}(x)$ is a corolla for every $x \in T_{1}$, so that $\CorO(f)$
  is injective away from the base point, i.e.\ $\CorO(f)$ is
  inert. Similarly, if $f$ is active, then it is easy to see that 
  every corolla of $T'$ must lie in the subtree $f_{1}(x)$ for some $x
  \in T_{1}$; thus nothing is sent to the base point, i.e.\ $\CorO(f)$
  is active.
\end{proof}

\subsection{Segal Presheaves on $\bbOV$}\label{subsec: Seg OV}
We now introduce the dendroidal analogue of the Segal presheaves of
\S\ref{subsec DFV}. We will state the analogues of many of the results
of \S\ref{sec segal presheaves} in this setting, but without giving
the proofs, seeing as they are entirely analogous.
\begin{definition}
  A presheaf $\xxO\colon \bbO^\op\to \xS$ is called a \emph{Segal
    presheaf} (or a \emph{dendroidal Segal space}) if the canonical
  map
  \[ \xxO(T)\to \lim_{S\in \bbOelTop} \xxO(S)\]
  is an equivalence for every $T \in \bbO$. We write
  $\PSeg(\bbO)$ for the full subcategory of $\PSh(\bbO)$ spanned by
  the Segal presheaves.
\end{definition}

\begin{defn}
 Let $\mathcal{V}$ is a symmetric monoidal \icat{}. By identifying $\Fin_{*}$ with its skeleton $\xF_{*}$, we define the
  \icat{} $\bbOV$ by the pullback
  \csquare{\bbOV}{\mathcal{V}_{\otimes}}{\bbO}{\xF_{*}^{\op}.}{}{}{}{\CorO^\op}
  We also write $\bbO_\xint^\xV$ and $\bbO_\xel^\xV$ for the pullbacks
  $\bbO_\xint\times_{\bbO} \bbO^\xV$ and
  $\bbO_\xel\times_{\bbO} \bbO^\xV,$ respectively. We let $\overline T$ or $T(v_c)_{c\in T_1}$ denote an object in $\bbO^\xV$ lying over an object $T\in \bbO$.
\end{defn}

\begin{remark}
  An object in $\bbO^\xV$ should be thought of as a tree in $\bbO$
  whose vertices are labeled by objects of $\xV$. 
\end{remark}

\begin{definition}\label{def dend segal spaces}
  Given a symmetric monoidal \icat{} $\xV^\otimes$, a presheaf
  $\xxO\colon \bbOVop \to \xS$ is called a \emph{Segal
    presheaf on $\bbOV$} if for every object $\overline{T}$ of $\bbOV$
  lying over $T \in \bbO$, the canonical map
  \[\xxO(\overline T)\to \lim_{\psi\in \bbOelTop}
  \xxO(\psi^*(\overline T))\]
  is an equivalence, where $\psi^*(\overline T)\to \overline T$ is the
  Cartesian lift of the inert map $\psi$ (corresponding to a
  coCartesian morphism in $\mathcal{V}^{\otimes}$). We write
  $\PSeg(\bbOV)$ for the full subcategory of $\PSh(\bbOV)$ spanned by
  the Segal presheaves. Similarly, we define the full subcategory
  $\PSeg(\bbOVint)$ of Segal presheaves on $\bbOVint$.
\end{definition}

\begin{defn}
  Let $\mathcal{V}$ be a presentably symmetric monoidal \icat{}. We
  say a presheaf $\mathcal{O} \in \PSh(\bbOV)$ is a \emph{continuous
    Segal presheaf} if it is a Segal presheaf and moreover the functor
  \[ \mathcal{V}^{\op} \simeq (\bbOVop)_{C_{n}} \to
  \xS_{/\mathcal{O}(\overline \eta)^{n+1}},\]
  induced by the Cartesian lifts of the $n+1$ morphisms $\eta\to C_n$,
  preserves all small limits in $\xV$ for every $n$. We write
  $\PCS(\bbOV)$ for the full subcategory of $\PSh(\bbOV)$ spanned by
  the continuous Segal presheaves. Similarly, we define the full
  subcategories $\PCS(\bbOVint)$ and $\PCts(\bbOVel)$ of continous
  Segal presheaves on $\bbOVint$ and continuous presheaves on
  $\bbOVel$.
\end{defn}

\begin{propn}
  The following are equivalent for a presheaf $\mathcal{O} \in \PSh(\bbOV)$:
  \begin{enumerate}
  \item $\mathcal{O}$ is a Segal presheaf.
  \item $\mathcal{O}$ is local with respect to the morphisms
    $\overline{T}_{\Seg} := \colim_{C \in \bbOelTop} \overline{C} \to
    \overline{T}$.
  \item $\mathcal{O}|_{\bbOVintop}$ is the right Kan extension of
    $\mathcal{O}|_{\bbOVelop}$.
  \end{enumerate}
\end{propn}
\begin{proof}
  As Proposition~\ref{propn:SegDFVcond}.
\end{proof}

\begin{propn}
  The following are equivalent for a Segal presheaf $\mathcal{O}$ in
  $\PSeg(\bbOV)$:
  \begin{enumerate}
  \item $\mathcal{O}$ is continuous.
  \item For every $n$, the presheaf
    \[ \mathcal{O}(C_{n}(\blank)) \colon \mathcal{V}^{\op} \simeq
    (\bbOVop)_{C_{n}} \to \mathcal{S} \]
    preserves weakly contractible limits, and the natural map
    $\mathcal{O}(C_{n}(\emptyset)) \to \prod_{n+1}
    \mathcal{O}(\overline{\eta})$ is an equivalence.
  \item $\mathcal{O}$ is local with respect to the map $\coprod_{n+1}
    \overline{\eta} \to C_{n}(\emptyset)$ and the map
    $\colim_{\mathcal{I}}C_{n}(\phi) \to C_{n}(\colim_{\mathcal{I}} \phi)$ for every weakly contractible diagram
    $\phi$ in $\mathcal{V}$.
  \item $\xxO$ is local with respect to the map $\coprod_{n+1}
    \overline{\eta} \to C_{n}(\emptyset)$ and the map
    $\colim_{\mathcal{I}^{\triangleleft}}C_{n}(\phi) \to
    C_{n}(\colim_{\mathcal{I}^{\triangleleft}} \phi)$ for
    every diagram $\phi$ such that $\phi(-\infty) \simeq \emptyset$.
  \end{enumerate}
\end{propn}
\begin{proof}
  As Proposition~\ref{propn:CtsDFV}.
\end{proof}
  
\begin{defn}
    Given a space $X$, we write $\bbOX \to \bbO$ for the right fibration associated to the functor $\bbO^{\op} \to \mathcal{S}$ given by the right Kan extension of the functor $\{\eta\} \to
    \mathcal{S}$ with value $X$ along the inclusion $\{\eta\}
    \hookrightarrow \bbO^{\op}$. 
    We write $\tilde T$ for an object in $\bbOX$ lying over $T \in \bbO$, and we view $\bbOX^{\op}$ as living over
    $\xF_{*}$ via the composite map \[\bbOX^{\op} \to \bbO^{\op}
    \xto{\CorO} \xF_{*}.\]
\end{defn}

\begin{remark}
  This right Kan extension $\bbO^{\op} \to \mathcal{S}$ takes an object $T$ to a product of copies
  of $X$ indexed by the number of edges of $T$.
\end{remark}

\begin{defn}
  If $\mathcal{V}^{\otimes}\to \xF_{*}$ is a symmetric monoidal
  \icat{}, then an \emph{$\bbOX^{\op}$-algebra} in $\mathcal{V}$ is a
  functor $\bbOX^{\op} \to \mathcal{V}^{\otimes}$ over $\xF_{*}$
  that takes the inert morphisms to coCartesian
  morphisms. We write $\Alg_{\bbOX^{\op}}(\mathcal{V})$ for the full
  subcategory of $\Fun_{\xF_{*}}(\bbOX^\op, \mathcal{V}^{\otimes})$
  spanned by the algebras. This is clearly functorial in $X$, and we
  write $\AlgOS(\mathcal{V}) \to \mathcal{S}$ for the
  Cartesian fibration associated to the functor $\mathcal{S}^{\op} \to
  \CatI$ taking $X$ to $\Alg_{\bbOX^{\op}}(\mathcal{V})$.
\end{defn}

\begin{thm}\label{thm:PCSisAlgDend}
  Let $\mathcal{V}$ be a presentably symmetric monoidal \icat{}. There is an equivalence of \icats{} over $\mathcal{S}$,
  \[
  \begin{tikzcd}
    \PCS(\bbOV) \arrow{dr}{\name{ev}_{\eta}} \arrow{rr}{\sim} &
      & \AlgOS(\mathcal{V}) \arrow{dl} \\
      & \mathcal{S}.
  \end{tikzcd}
  \]
\end{thm}
\begin{proof}
  As Theorem~\ref{thm:PCSisAlgLT}.
\end{proof}

\begin{thm}\label{thm enr psh dend}
  Let $\xU^\otimes$ be a small symmetric monoidal \icat{}. Let $y
  \colon \bbOU \to \bbO^{\PSh(\mathcal{U})}$ denote the functor
  induced by the Yoneda embedding $\mathcal{U} \to \PSh(\mathcal{U})$,
  viewed as a symmetric monoidal functor. Then right Kan extension
  along $y$ gives a fully
  faithful functor
  $y_{*} \colon \PSh(\bbOU) \hookrightarrow
  \PSh(\bbO^{\PSh(\mathcal{U})})$
  which restricts to an equivalence
  $\PSeg(\bbOU) \isoto \PCS(\bbO^{\PSh(\mathcal{U})})$.
\end{thm}
\begin{proof}
  As Theorem~\ref{thm enr psh}.
\end{proof}
Building on this, there are also obvious dendroidal variants of the results of
\S\ref{subsec enr loc}. We will not spell these out explicitly, except
for the following analogue of Corollaries~\ref{cor:kappaseg} and
\ref{cor PCSDFV pres}, which we will need below:
\begin{cor}
  Suppose $\mathcal{V}$ is a presentably symmetric monoidal
  \icat{}. Then $\mathcal{V}$ is $\kappa$-presentably symmetric
  monoidal for some regular cardinal $\kappa$, and we have an
  equivalence
  \[ \PCS(\bbOV) \isoto \PkSeg(\bbO^{\mathcal{V}^{\kappa}}).\]
  In particular, $\PCS(\bbOV)$ is a presentable \icat{}.
\end{cor}

\subsection{Reduction to $\DFi$}\label{subsec DFi} 
Just as in \cite{ChuHaugsengHeuts}, in order to connect $\DF$ to
$\bbO$ it is convenient to introduce an intermediate category. While
$\bbO$ consists of trees, the objects of $\DF$ can be thought of as
``forests'' of trees with levels; the category $\DFi$ is then the
full subcategory of trees:
\begin{defn}
  As in \cite{ChuHaugsengHeuts} we write $\DFi$ for the full
  subcategory of $\DF$ spanned by the objects $([n], f)$ such that
  $f(n) = \mathbf{1}$. Since $\DFel$ is contained in $\DFi$ we can
  define Segal presheaves on $\DFi$ as presheaves $\mathcal{O} \in
  \PSh(\DFi)$ such that for every $I \in \DFi$ the canonical map
  \[ \mathcal{O}(I) \to \lim_{J \in \DFelIop} \mathcal{O}(J) \]
  is an equivalence. Let $i$ denote the inclusion $\DFi
  \hookrightarrow \DF$. If $\mathcal{V}$ is a symmetric monoidal \icat{},
  we define $\DFiV$ by the pullback square
  \csquare{\DFiV}{\DFV}{\DFi}{\DF;}{\overline{\imath}}{}{}{i}
  we also similarly
  define $\DFiVint$. We also define Segal presheaves on $\DFiV$ and,
  when $\mathcal{V}$ is presentably symmetric monoidal, continuous
  Segal presheaves on $\DFiV$, by the obvious variants of
  Definitions~\ref{def Segal presheaf} and \ref{def cts Seg psh}; we
  write $\PSeg(\DFiV)$ and $\PCS(\DFiV)$ for the full subcategories of
  $\PSh(\DFiV)$ spanned by the Segal presheaves and the continuous
  Segal presheaves, respectively.
\end{defn}

In \cite[Lemma 2.11]{ChuHaugsengHeuts} we proved that the functor
$i^{*} \colon \PSh(\DF) \to \PSh(\DFi)$ given by composition with $i$
restricts to an equivalence $\PSeg(\DF) \to \PSeg(\DFi)$. We now want
to prove an enriched analogue of this statement:
\begin{propn}\label{propn:DFiV}
  Let $\mathcal{V}$ be a symmetric monoidal \icat{}.
  \begin{enumerate}[(i)]
  \item The functor
    $\bari^{*} \colon \PSh(\DFV) \to \PSh(\DFiV)$
    restricts to an equivalence \[\PSeg(\DFV) \isoto \PSeg(\DFiV).\]
  \item If $\mathcal{V}$ is presentably symmetric monoidal, then this
    restricts further to an equivalence
    \[ \PCS(\DFV) \isoto \PCS(\DFiV).\]
  \end{enumerate}
\end{propn}

\begin{lemma}\label{lem:bariSeg}
  Let $\mathcal{V}$ be a symmetric monoidal \icat{}.
  \begin{enumerate}[(i)]
  \item The functor $\bari^{*}$ preserves Segal presheaves.
  \item The functor $\bari^{*}$ has a right adjoint $\bari_{*}$, given by
    right Kan extension along $\bari^{\op}$.
  \item The functor $\bari_{*}$ preserves Segal presheaves.
  \end{enumerate}
\end{lemma}
\begin{proof}
  A presheaf on $\DFiV$ is a Segal presheaf \IFF{}
  it is local with respect to the maps $\overline{I}_{\Seg} \to
  \overline{I}$ for $\overline{I}$ in $\DFiV$. For $\mathcal{O} \in
  \PSeg(\DFV)$ we have natural equivalences \[\Map_{\PSh(\DFiV)}(\overline{I},
  \bari^{*}\mathcal{O}) \simeq \mathcal{O}(\bari(\overline{I})) \simeq
  \Map_{\PSh(\DFV)}(\bari(\overline{I}), \mathcal{O}),\]
  \[ \Map_{\PSh(\DFiV)}(\overline{I}_{\Seg}, \bari^{*}\mathcal{O})
  \simeq \Map_{\PSh(\DFV)}(\bari(\overline{I})_{\Seg}, \mathcal{O}),\]
  where the second arises by moving the colimit in
  $\overline{I}_{\Seg}$ outside and then using equivalences of the
  first kind. Since $\bari(\overline{I})_{\Seg} \to
  \bari(\overline{I})$ is a generating Segal equivalence in
  $\PSeg(\DFV)$, it follows that $\bari^{*}\mathcal{O}$ is a Segal
  presheaf. Thus we have proved (i).

  To prove (ii), it suffices to show that for $F \in \PSh(\DFiV)$ and
  $\overline{I}$ in $\DF$, the limit
  $\lim_{\overline{J} \in (\DFiVop)_{\overline{I}/}} F(\overline{J})$
  exists in $\mathcal{S}$. By Lemma~\ref{lem:cofslice} there is a
  coinitial map to $(\DFiVop)_{\overline{I}/}$ from an \icat{}
  equivalent to $(\DFiop)_{I/}$; this is small, and $\mathcal{S}$ has
  all small colimits, so the right Kan extension along $\bari^{\op}$
  exists.

  To prove (iii), it suffices to show that $\bari^{*}$ takes a
  collection of generating Segal equivalences in $\PSh(\DFV)$ to Segal
  equivalences in $\PSh(\DFiV)$. By Proposition~\ref{propn:SegDFVcond}
  it is enough to check this for the images of the maps
  \[ \overline{I}_{\Seg}  \to \overline{I}\]
  for all $\overline{I}$ in $\DFiV \subseteq \DFV$, and
  \[ \coprod_{i \in f(n)} \overline{I}_{i} \to \overline{I} \]
  for all $I = ([n], f)$, where $I_{i} = ([n], f_{i})$ is obtained
  by taking the fibres at $i \in f(n)$.
  
  Since $\bari$ is fully faithful, the images of the first class of
  maps are clearly just the generating Segal equivalences in
  $\PSh(\DFiV)$. Moreover, from the definition of $\DFV$ it is easy to
  see that for $\overline{J} \in \DFiV$ we have
  \[\coprod_{i \in f(n)}
  \Map_{\DFV}(\overline{J}, \overline{I}_{i}) \isoto
  \Map_{\DFV}(\overline{J}, \overline{I}),\]
  so that
  $\bari^{*}\left(\coprod_{i \in f(n)} \overline{I}_{i}\right) \to
  \bari^{*}\overline{I}$ is an equivalence. This gives (iii).
\end{proof}

\begin{proof}[Proof of Proposition~\ref{propn:DFiV}]
  By Lemma~\ref{lem:bariSeg}, the adjunction $\bari^{*} \dashv
  \bari_{*}$ restricts to an adjunction
  \[ \bari^{*} : \PSeg(\DFV) \rightleftarrows \PSeg(\DFiV) :
  \bari_{*}.\]
  Since $\bari$ is fully faithful, the right Kan extension $\bari_{*}$
  is also fully faithful, and so we automatically have that the counit
  map $\bari^{*}\bari_{*} \to \id$ is an equivalence. If $\mathcal{O}$
  is a Segal presheaf on $\DFV$, then so is
  $\bari_{*}\bari^{*}\mathcal{O}$. To show that the unit map
  $\mathcal{O} \to \bari_{*}\bari^{*}\mathcal{O}$ is an equivalence,
  it therefore suffices to see that this map of presheaves is an
  equivalence when evaluated at $\overline{\mathfrak{e}}$ and
  $\mathfrak{c}_n(v)$ for $v \in \mathcal{V}$. Since these
  objects all lie in $\DFiV$, this follows from $\bari$ being fully
  faithful. This proves (i), and (ii) is clear since this only depends
  on objects in the image of $\bari$.
\end{proof}

\subsection{Comparison}\label{sec preparation for comparison}
In \cite{ChuHaugsengHeuts} we defined a functor
$\tau \colon \DFi \to \bbO$ and proved that composition with $\tau$
induces an equivalence
$\tau^{*} \colon \PSeg(\bbO) \isoto \PSeg(\DFi)$. In this subsection
we will prove the enriched analogue of this result. To state this, we
first introduce some notation:
\begin{defn}
  Let $\tau \colon \DFi \to \bbO$ be the functor constructed in
  \cite[\S 4]{ChuHaugsengHeuts}; by \cite[Lemmas 4.4 and
  4.5]{ChuHaugsengHeuts} this is compatible with the inert--active
  factorization systems and restricts to a functor
  $\tauint \colon \DFiint \to \bbOint$ and an equivalence
  $\tauel \colon \DFel \isoto \bbOel$.  Given a symmetric monoidal
  \icat{} $\mathcal{V}$, let $\bartau \colon \DFiV \to \bbOV$ be the
  functor defined by the pullback square
  \csquare{\DFiV}{\bbOV}{\DFi}{\bbO.}{\bartau}{}{}{\tau} We write
  $\bartauint \colon \DFiVint \to \bbOVint$ and
  $\bartauel \colon \DFVel \to \bbOVel$ for the appropriate
  restrictions of $\bartau$.
\end{defn}

\begin{remark}\label{rmk bartauel eq}
  By \cite[Lemma 4.4]{ChuHaugsengHeuts}, the functor $\tau$ restricts
  to an equivalence $\tauel \colon \DFel \isoto \bbOel$. The pullback
  $\bartau_{\name{el}} \colon \DFVel \to \bbOVel$ is therefore also an equivalence.
\end{remark}

Our goal is then to prove:
\begin{thm}\label{thm:PSegeq}
  Let $\mathcal{U}$ be a small symmetric monoidal \icat{}. The functor $\bartau^{*} \colon \PSh(\bbOU) \to \PSh(\DFiU)$
  restricts to an equivalence
  \[ \PSeg(\bbOU) \to \PSeg(\DFiU).\]
\end{thm}

Before we turn to the proof, we note that as an immediate consequence
our models for enriched \iopds{} as continuous Segal presheaves are
equivalent. This is a special case of the following observation:
\begin{cor}\label{cor:PSSegeq}
  Let $\mathcal{U}$ be a small symmetric monoidal \icat{} and let
  $\mathbb{S}$ be a small set of morphisms in $\PSh(\mathcal{U})$,
  compatible with the symmetric monoidal structure. Then:
  \begin{enumerate}[(i)]
  \item Composition with $\bartau \colon \DFiU \to \bbOU$ restricts
    to an equivalence
    \[ \PSSeg(\bbOU) \isoto \PSSeg(\DFiU).\]
  \item Composition with $\bartau \colon
    \DF^{1,\PSU} \to
    \bbO^{\PSU}$ restricts to an equivalence
    \[ \PCS(\bbO^{\PSU}) \isoto
    \PCS(\DF^{1,\PSU}).\]
  \end{enumerate}
\end{cor}

\begin{proof}
  By Theorem~\ref{thm:PSegeq} composition with $\bartau$ restricts to
  an equivalence $\bartau^{*} \colon \PSeg(\bbOU) \isoto \PSeg(\DFiU)$. The full
  subcategories $\PSSeg(\bbOU)$ and $\PSSeg(\DFiU)$ consist of Segal
  presheaves  whose restrictions to $\bbOUel$ and $\DFUel$
  satisfy a locality condition, as in Definition~\ref{def S Segal
    presheaves}. Since $\bartauel \colon \DFUel \to
  \bbOUel$ is an equivalence by Remark~\ref{rmk bartauel eq}, these
  subcategories correspond under the equivalence $\bartau^{*}$. This
  gives (i).

  To prove (ii), observe that we have a commutative square
  \csquare{\PCS(\bbO^{\PSU})}{\PCS(\DF^{1,\PSU})}{\PSSeg(\bbOU)}{\PSSeg(\DFiU),}{\bartau^*}{}{}{\bartau^*}
  where the vertical maps, given by composition with the maps induced
  by the Yoneda embedding $\mathcal{U} \to \PSU$, are equivalences by analogues of
  Corollary~\ref{cor:kappaseg} for $\bbO$ and $\DFi$. Since the bottom
  horizontal map is an equivalence by (i), so is the upper horizontal map.
\end{proof}

\begin{cor}
  If $\mathcal{V}$ is a presentably symmetric monoidal \icat{}, then
  the functor $\colon \PSh(\bbOV) \to \PSh(\DFiV)$
  restricts to an equivalence
  \[ \PCS(\bbOV) \to \PCS(\DFiV).\]
\end{cor}
\begin{proof}
  Since $\mathcal{V}$ is presentably symmetric monoidal, by
  Proposition~\ref{propn Vkappa symmon} we can choose
  a regular cardinal $\kappa$ such that $\mathcal{V}^{\kappa}$ is a
  symmetric monoidal subcategory of $\mathcal{V}$, and $\mathcal{V}
  \simeq \Ind_{\kappa}\mathcal{V}^{\kappa}$ as a symmetric monoidal
  \icat{}. The result is then a special case of
  Corollary~\ref{cor:PSSegeq} applied to $\mathcal{V}^{\kappa}$ with
  $\mathbb{S}$ as in the proof of Corollary~\ref{cor:kappaseg}.
\end{proof}

We now turn to the proof of Theorem~\ref{thm:PSegeq}. Our approach
closely follows that of \cite[Theorem 5.1]{ChuHaugsengHeuts}, and
where the proofs are obtained from the unenriched version simply by
adding superscript $\mathcal{V}$'s and overlines, we will not repeat
them here.
\begin{lemma}
  The functor $\bartau^{*} \colon \PSh(\bbOV) \to \PSh(\DFiV)$
  preserves Segal presheaves.
\end{lemma}
\begin{proof}
  As \cite[Lemma 4.5]{ChuHaugsengHeuts}.
\end{proof}

\begin{lemma}\label{lem:tauinteq}
  Composition with the functor $\bartauint \colon \DFiVint \to
  \bbOVint$ induces an equivalence \[\PSeg(\bbOVint) \isoto \PSeg(\DFiVint).\]
\end{lemma}
\begin{proof}
  As \cite[Lemma 5.2]{ChuHaugsengHeuts}.
\end{proof}

We let $\barj_{\bbO}$ and $\barj_{\DFi}$ denote the inclusions $\bbOVint
\to \bbOV$ and $\DFiVint \to \DFiV$, respectively. Moreover, we write
$\overline{L}_{\bbO}$ and $\overline{L}_{\DFi}$ for the localizations
$\PSh(\bbOV) \to \PSeg(\bbOV)$ and $\PSh(\DFiV) \to \PSeg(\DFiV)$,
respectively.

\begin{lemma}\label{lem:barjmonad}\ 
  \begin{enumerate}[(i)]
  \item The functor $\barj_{\bbO}^{*} \colon \PSeg(\bbOV) \to
    \PSeg(\bbOVint)$ has a left adjoint
    $\overline{F}_{\bbO}:=\overline{L}_{\bbO}\barj_{\bbO,!}$, and the adjunction
    $\overline{F}_{\bbO} \dashv \barj_{\bbO}^{*}$ is
    monadic.
  \item The functor $\barj_{\DFi}^{*} \colon \PSeg(\DFiV) \to
    \PSeg(\DFiVint)$ has a left adjoint
    $\overline{F}_{\DFi} := \overline{L}_{\DFi}\barj_{\DFi,!}$, and the adjunction
    $\overline{F}_{\DFi} \dashv \barj_{\DFi}^{*}$ is
    monadic.
  \end{enumerate}
\end{lemma}
\begin{proof}
  As \cite[Lemma 5.3]{ChuHaugsengHeuts}.
\end{proof}

\begin{propn}
  The functor $\bartau^{*} \colon \PSh(\bbOV) \to \PSh(\DFiV)$
  preserves Segal equivalences.
\end{propn}
\begin{proof}
  We need to prove that for $\overline{T}$ in $\bbOV$, the map
  $\tau^{*}\overline{T}_{\Seg} \to \tau^{*}\overline{T}$ is a Segal
  equivalence in $\PSh(\DFiV)$. We prove this by inducting on the
  number of vertices of $\overline{T}$, noting that the statement is
  vacuous if $\overline{T}$ has zero or one vertices.

  For $T$ in $\bbO$, let $\Dext T$ and $(\Dext T)_{\Seg}$ be as in
  \cite[Definition 5.7]{ChuHaugsengHeuts}. Writing $\pi$ for the
  projection $\bbOV \to \bbO$, we then define $\Dext \overline{T}$ for
  $\overline{T}$ in $\bbOV$ by the pullback square
  \nolabelcsquare{\Dext \overline{T}}{\overline{T}}{\pi^* \Dext
    T}{\pi^* T.}  We also define $(\Dext \overline{T})_{\Seg}$
  similarly. By \cite[Lemma 5.8]{ChuHaugsengHeuts} the natural map
  $(\Dext T)_{\Seg} \to T_{\Seg}$ is an equivalence; since pullbacks
  in $\PSh(\bbOV)$ preserve colimits, we see that the natural map
  $(\Dext \overline{T})_{\Seg} \to \overline{T}_{\Seg}$ is a pullback
  of this map, and hence this is also an equivalence. We have a
  commutative square
  \nolabelcsquare{(\Dext \overline{T})_{\Seg}}{\Dext
    \overline{T}}{\overline{T}_{\Seg}}{\overline{T}.}
  Here the upper horizontal morphism is a colimit of generating Segal
  equivalences for trees with fewer vertices than $\overline{T}$, and
  is therefore mapped to a Segal equivalence in $\PSh(\DFiV)$ by the
  inductive hypothesis. It therefore suffices to show that
  $\bartau^{*}\Dext \overline{T} \to \bartau^{*}\overline{T}$ is a
  Segal equivalence. If $\psi$ denotes the projection $\DFiV \to
  \DFi$, then as $\bartau^{*}$ preserves pullbacks we have a Cartesian
  square
  \nolabelcsquare{\bartau^{*}\Dext
    \overline{T}}{\bartau^{*}\overline{T}}{\psi^* \tau^* \Dext
    T}{\psi^* \tau^* T.}
  The proof of \cite[Proposition 5.6]{ChuHaugsengHeuts} shows that
  $\tau^{*} \Dext T  \to \tau^{*} T$ is an inner anodyne map in
  $\PSh(\DFi)$. By Proposition~\ref{propn:pbwsat} it follows that
  $\bartau^{*}\Dext \overline{T} \to \bartau^{*}\overline{T}$ is an
  inner anodyne map in $\PSh(\DFiV)$, and hence a Segal equivalence by
  Proposition~\ref{propn:iaDFV}.
\end{proof}

As an immediate consequence, we get:
\begin{cor}\label{cor:bartauRKE}
  The functor $\bartau_{*}$ given by right Kan extension along
  $\bartau$ restricts to a functor $\bartau_{*}\colon  \PSeg(\DFiV)
  \to \PSeg(\bbOV)$, right adjoint to $\bartau^{*}$.
\end{cor}

\begin{lemma}\label{lem:rightadjeqce}
  The canonical map $\bartauint^{*}\barj_{\bbO}^{*}\bartau_{*}F \simeq
  \barj_{\DFi}^{*}\bartau^{*}\bartau_{*}F \to \barj_{\DFi}^{*}F$ is an
  equivalence for $F \in \PSeg(\DFiV)$.
\end{lemma}
\begin{proof}
  As \cite[Lemma 5.5]{ChuHaugsengHeuts}.
\end{proof}

\begin{proof}[Proof of Theorem~\ref{thm:PSegeq}]
  We have a commutative square
  \csquare{\PSeg(\bbOV)}{\PSeg(\DFiV)}{\PSeg(\bbOVint)}{\PSeg(\DFiVint),}{\bartau^*}{\barj_{\bbO}^*}{\barj_{\DFi}^*}{\bartauint^*}
  where the lower horizontal morphism is an equivalence by
  Lemma~\ref{lem:tauinteq} and the vertical morphisms are monadic
  right adjoints by Lemma~\ref{lem:barjmonad}. It follows from
  \cite[Corollary 4.7.4.16]{ha} that to show $\bartau^{*}$ is an
  equivalence it is enough to prove that the canonical natural
  transformation $\overline{F}_{\DFi}\bartauint^{*} \to
  \bartau^{*}\overline{F}_{\bbO}$ is an equivalence. But by
  Corollary~\ref{cor:bartauRKE} both functors are left adjoints, so
  this transformation is an equivalence \IFF{} the corresponding
  transformation of right adjoints $\barj_{\bbO}^{*}\bartau_{*} \to
  (\bartauint^{*})^{-1} \barj_{\DFi}^{*f}$ is an equivalence, which it
  is by Lemma~\ref{lem:rightadjeqce}.
\end{proof}

\begin{defn}
  Let $u_{\bbO}^*\colon \PSh(\bbO)\to \PSh(\simp)$ be the functor
  induced by the inclusion $u_{\bbO} \colon \simp \to \bbO$ of
  Remark~\ref{rem u}. If $\mathcal{O}$ is a Segal presheaf on $\bbO$,
  then $u_{\bbO}^{*}\mathcal{O}$ is a Segal space, and we say that $\xxO$ is
  \emph{complete} if $u_{\bbO}^{*}\mathcal{O}$ is a complete Segal
  space. Similarly, we say that a continuous Segal presheaf
  $\mathcal{O} \in \PCS(\bbOV)$ is complete if
  $\overline{u}_{\bbO}^{*}\mathcal{O} \in \PCS(\simp^{\mathcal{V}})$
  is complete, where
  $\overline{u}_{\bbO} \colon \simp^{\mathcal{V}} \to \bbOV$ denotes
  the pullback of $u_{\bbO}$ along the projection $\bbOV \to \bbO$.
\end{defn}
Since $u_{\bbO} = \tau \circ u$, the complete objects correspond
under the equivalence $\bartau^{*}$, giving:
\begin{cor}
  Let $\mathcal{V}$ be a presentably symmetric monoidal
  \icat{}. Composition with $\bartau$ induces an equivalence
  \[ \PCCS(\bbOV) \isoto \PCCS(\DFiV).\]
\end{cor}

\section{Rectification of Enriched $\infty$-Operads}\label{sec alg}
In this section we relate our homotopy theory of enriched \iopds{} to
the existing literature on model categories of enriched operads. In
\S\ref{sec opd} we relate our dendroidal model to algebras for the
($\infty$-)operads for coloured operads, and in \S\ref{subsec rect} we
prove a rectification result for \iopds{} enriched in a symmetric
monoidal \icat{} coming from a nice symmetric monoidal model category.

\subsection{Operads for Operads}\label{sec opd} 
In this subsection we will prove that for a set $S$, algebras for
$\bbO_{S}^{\op}$ in a symmetric monoidal \icat{} are equivalent to
algebras for the operad for $S$-coloured operads. We will do this by
showing that $\bbO_{S}^{\op}$ is an ``approximation'' to this operad,
in the sense of \cite[\S 2.3.3]{ha}. Our first task is therefore to
give a convenient definition of these operads, which requires some
observations about certain pushouts in $\bbO$ and $\bbO_{S}$:

\begin{lemma}\label{lem:bbOsubst}
  Given an active morphism $C_{i} \to T$ and an inert morphism $C_{i}
  \to S$ in $\bbO$, the pushout
  \nolabelcsquare{C_i}{T}{S}{S \amalg_{C_{i}} T}
  exists in $\bbO$.
\end{lemma}
\begin{proof}
  It follows from \cite[Proposition 1.1.19]{Kock} that for any tree
  $S$ the ``Segal diagram'' $\bbO_{\name{el}/S}^{\triangleright} \to \bbOint$ is a
  colimit diagram. The composite diagram
  $\bbO_{\name{el}/S}^{\triangleright} \to \bbO$ is therefore also a colimit, since this is
  given by composition with the free polynomial monad functor, which
  is a left adjoint. Thus the pushout $S \amalg_{C_{i}} T$, if it
  exists, is equivalent to the colimit of the diagram
  $\bbO_{\name{el}/S}^{\triangleright} \to \bbO$ obtained from the Segal diagram
  of $S$ by replacing the corolla $C_{i}$ by the tree $T$. This
  colimit is still given by grafting of trees, and so exists by
  \cite[Proposition 1.1.19]{Kock}.
\end{proof}

\begin{remark}\label{rmk:glueing}
  The tree $S \amalg_{C_{i}} T$ is obtained by substituting the
  corolla $C_{i}$ in $S$ by the tree $T$. We refer to this as
  \emph{substituting} the tree $T$ into $S$. If we are given also an
  active morphism $C_{j}\to U$ and an inert morphism $C_{j} \to T$,
  then this procedure is associative in the sense that
  $(S \amalg_{C_{i}} T) \amalg_{C_{j}} U$ is canonically isomorphic to
  $S \amalg_{C_{i}} (T \amalg_{C_{j}} U)$, since we have the following
  commutative diagram where all squares are pushouts:
  \[
  \begin{tikzcd}
  {} & C_{j}  \arrow{r} \arrow[hookrightarrow]{d} & U
  \arrow[hookrightarrow]{d} \\
  C_{i} \arrow{r} \arrow[hookrightarrow]{d} & T \arrow{r} \arrow[hookrightarrow]{d}& T
  \amalg_{C_{j}} U \arrow[hookrightarrow]{d}\\
  S \arrow{r} & S \amalg_{C_{i}} T \arrow{r} & X.
\end{tikzcd}
\]
Here $X$ can be identified with both
$(S \amalg_{C_{i}} T) \amalg_{C_{j}} U$ and
$S \amalg_{C_{i}} (T \amalg_{C_{j}} U)$.
\end{remark}

The analogous result also holds in $\bbO_X$ for any space $X$. To see this, we use:
\begin{lemma}\label{lem:rightfibpushout}
  Let $\mathcal{E} \to \mathcal{B}$ be a right fibration
  corresponding to a functor $F \colon \mathcal{B}^{\op} \to
  \mathcal{S}$. Suppose
  \csquare{a}{b}{b'}{c}{f}{f'}{g}{g'}
  is a pushout diagram in $\mathcal{B}$ and that $F$ takes this to a
  pullback square in $\mathcal{S}$. Then for any morphisms $\bar{b} \xfrom{\bar{f}}
  \bar{a} \xto{\bar{f}'} \bar{b'}$ lying over $b \from a \to b'$ there is a
  pushout square in $\mathcal{E}$ lying over the given pushout square in
  $\mathcal{B}$.
\end{lemma}
\begin{proof}
  We have $\mathcal{E}_{c} \simeq
  \mathcal{E}_{b}\times_{\mathcal{E}_{a}} \mathcal{E}_{b'}$, and the
  morphisms $\bar{f}$, $\bar{f}'$ determine a point $\bar{c} \in
  \mathcal{E}_{c}$ and a commutative square
  \csquare{\bar{a}}{\bar{b}}{\bar{b}'}{\bar{c}}{\bar{f}}{\bar{f}'}{\bar{g}}{\bar{g}'}
  For $\bar{x}$ in $\mathcal{E}$ lying over $x \in \mathcal{B}$ we
  then have a commutative square
  \[
  \begin{tikzcd}
    \Map_{\mathcal{E}}(\bar{c}, \bar{x}) \arrow{r} \arrow{d} &
    \Map_{\mathcal{E}}(\bar{b}, \bar{x})
    \times_{\Map_{\mathcal{E}}(\bar{a}, \bar{x})}
    \Map_{\mathcal{E}}(\bar{b}', \bar{x}) \arrow{d} \\
    \Map_{\mathcal{B}}(c, x) \arrow{r}{\sim} & \Map_{\mathcal{B}}(b, x)
    \times_{\Map_{\mathcal{B}}(a, x)} \Map_{\mathcal{B}}(b', x).
  \end{tikzcd}
  \]
  Since the bottom horizontal morphism is an equivalence, to show the
  top horizontal morphism is an equivalence it suffices to prove this
  square is Cartesian. This is equivalent to the map between the
  fibres at any $\phi \colon c \to x$ being an equivalence. This map
  can be identified with
  \[ \Map_{\mathcal{E}_{c}}(\bar{c}, \phi^{*}\bar{x}) \to
  \Map_{\mathcal{E}_{b}}(\bar{b}, g^{*}\phi^{*}\bar{x})
  \times_{\Map_{\mathcal{E}_{a}}(\bar{a}, f^{*}g^{*}\phi^{*}\bar{x})}
  \Map_{\mathcal{E}_{b'}}(\bar{b}', g'^{*}\phi^{*}\bar{x}),\]
  which is an equivalence since $\mathcal{E}_{c}$ is a pullback.
\end{proof}

\begin{lemma}
  Let $X$ be a space. Given an active morphism $C_{i} \to T$ and an
  inert morphism $C_{i} \to S$ in $\bbO$, and morphisms $\tilde{C}_{i}
  \to \tilde{T}$ and $\tilde{C}_{i} \to \tilde{S}$ in $\bbO_{X}$ lying
  over these, the pushout
  \nolabelcsquare{\tilde{C}_i}{\tilde{T}}{\tilde{S}}{\tilde{S}
    \amalg_{\tilde{C}_{i}} \tilde{T}} 
  exists in $\bbO_{X}$ and forgets to a pushout in $\bbO$.
\end{lemma}
\begin{proof}
  Using Lemma~\ref{lem:rightfibpushout} and Lemma~\ref{lem:bbOsubst}
  this follows from the observation that in this situation the sets of
  edges give a pushout.
\end{proof}

\begin{defn}
  Let $S$ be a set. The operad $\sOp_{S}$ is defined as follows: The
  objects of $\sOp_{S}$ are pairs $(A, \alpha)$ where $A$ is a finite
  set and $\alpha \colon A_{+} \to S$ is a function; equivalently, the
  objects are objects $\tilde{C}_A$ of $\bbO_{S}$ lying over the
  corolla $C_A$ in $\bbO$. A multimorphism
  $((A_{1},\alpha_{1}),\ldots, (A_{n}, \alpha_{n})) \to (B, \beta)$ is
  given by an object $\tilde{T}$ of $\bbO_{S}$ together with inert
  maps $\tilde{C}_{A_{i}} \to \tilde{T}$ such that each hits a distinct
  vertex of $\tilde{T}$ and all vertices are hit, and an active map
  $\tilde{C}_{B} \to \tilde{T}$. (In other words, the tree $\tilde{T}$
  is assembled from the corollas $\tilde{C}_{A_{i}}$ in such a way that
  the labels match up, the tree $\tilde{T}$ has $|B|$ leaves, and
  the labels of the leaves and root match those of $\tilde{C}_{B}$.) More precisely, a multimorphism is an
  \emph{isomorphism class} of this data. It is convenient to represent
  this as a cospan
  $\coprod \tilde{C}_{A_{i}} \to \tilde{T} \from \tilde{C}_{B}$ (which
  can be thought of as living in the category obtained by freely
  adjoining coproducts to $\bbO_{S}$).
  Composition is given by substitution of trees: Given a multimorphism
  $\phi\colon (A_{1}, \ldots, A_{n}) \to B$ as above and another multimorphism
  $\psi\colon (C_{1},\ldots, C_{m}) \to A_{i}$ corresponding to
  $\coprod_{j} \tilde{C}_{C_{j}} \to \tilde{S} \from \tilde{C}_{A_{i}}$,
  then the composite $\psi \circ_{i} \phi$
  is given by the cospan
  \[
  \coprod_{i} \tilde{C}_{A_{i}} \amalg_{\tilde{C}_{A_{i}}} \coprod_{j}
  \tilde{C}_{C_{j}} \to \tilde{T} \amalg_{\tilde{C}_{A_{i}}} \tilde{S}
  \from \tilde{C}_{B}.  
  \]
  The composition is associative because of the associativity of
  tree substitution discussed in Remark~\ref{rmk:glueing}, and it is easy to
  see that the other requirements for an operad are satisfied.
\end{defn}

Recall that if $\mathbf{O}$ is an operad, we can define its
\emph{category of operators} $\mathbf{O}^{\otimes} \to \xF_{*}$.
This has objects lists $(x_{1},\ldots,x_{n})$ of objects of
$\mathbf{O}$, with a morphism
$(x_{1},\ldots,x_{n}) \to (x'_{1},\ldots,x'_{m})$ given by a morphism
$\phi \colon \angled{n} \to \angled{m}$ in $\xF_{*}$ and for all
$i = 1,\ldots,m$ a multimorphism
$(x_{j})_{j \in \phi^{-1}(j)} \to x'_{i}$. Equivalently, we can
replace the skeleton $\xF_{*}$ by $\Fin_{*}$ and take the objects
to be $(A \in \Fin, (x_{i})_{i \in A})$. Applying this construction to
the operad $\sOp_{S}$ gives a category $\OOp_{S} \to \Fin_{*}$; this
is an \iopd{}. We now wish to define a functor from $\bbOop_{S}$ to
$\OOp_{S}$.

\begin{defn}
  We first define the functor
  $\Theta \colon \bbO^{\op} \to \OOp_{*}$ over $\Fin_{*}$ as
  follows: If $T \in \bbO^{\op}$ corresponds to
  a diagram
  \[ T_{0} \from T_{2} \to T_{1} \to T_{0}\]
  then $\Theta(T)$ is the sequence $(T_{2,t})_{t\in T_1}$. For a morphism
  $f \colon T
  \to T'$ in $\bbO^{\op}$, given by a diagram
  \[
  \begin{tikzcd}
    T'_{0} \arrow{d}{f_{0}} & T'_{2} \arrow[phantom]{dr}[very near
    start]{\ulcorner} \arrow{l}\arrow{r}\arrow{d}{f_{2}} &
    T'_{1}\arrow{r}\arrow{d}{f_{1}} & T'_{0} \arrow{d}{f_{0}} \\
    T_{0} & \name{sub}'(T)\arrow{l}\arrow{r} & \name{sub}(T)
    \arrow{r} & T_{0},
  \end{tikzcd}
  \]
  the morphism $\Theta(f)$ is given by the morphism
  $\CorO(f) \colon T_{1,+} \to T'_{1,+}$ in $\Fin_{*}$ as in
  Definition~\ref{defn:CorO} together with the subtrees
  $f_{1}(x) \in \name{sub}(T)$ for $x \in T'_{1}$ and the isomorphisms
  between $T'_{2,x}$ and the leaves of this tree given by the pullback
  square
  \nolabelcsquare{T'_{2,x}}{\{x\}}{\name{sub}'(T)_{f_1(x)}}{\{f_1(x)\}.}
  It is easy to see that this is compatible with composition, because
  composition in $\bbO$ can be described in terms of substitution of
  trees.

  Now for a set $S$, we define $\Theta_{S} \colon \bbO_{S}^{\op} \to
  \OOp_{S}$ in the same way, just carrying the labelling of edges along.
\end{defn}

We wish to prove that the functor $\Theta_S$ is an
\emph{approximation} in the sense of \cite[Definition 2.3.3.6]{ha},
which we first recall for the reader's convenience:
\begin{definition}\label{def approximation}
  Suppose $p\colon \xxO \to \xF_{*}$ is an $\infty$-operad
  and $\xcc$ an $\infty$-category. We call a functor
  $f\colon \xcc\to \xxO$ an \emph{approximation to $\xxO$}, if
  it satisfies the following conditions:
  \begin{enumerate}
  \item Let $p'=p\circ f$, let $c\in \xcc$ be an object and let
    $p'(c)=\langle n\rangle$.  For every $1\leq i\leq n$, there is a
    locally $p'$-coCartesian morphism $ \alpha_i\colon c\to c_i $ in
    $\xcc$ lying over the inert map
    $\rho^i\colon\langle n\rangle\to \langle 1\rangle$ given by the
    projection at the $i$th element, and the morphism $f(\alpha_i)$ in
    $\xxO$ is inert.
  \item Let $c\in \xcc$ and let $\alpha\colon u\to f(c)$ be an active
    morphism in $\xxO$. There exists an $f$-Cartesian morphism
    $\overline \alpha\colon \overline u\to c$ lifting $\alpha$.
  \end{enumerate}
\end{definition}

\begin{propn}
  The functor $\Theta_{S} \colon \bbOop_{S} \to \OOp_{S}$ is an
  approximation.
\end{propn}
\begin{proof}
  We need to verify the two conditions mentioned in the previous
  definition.

  Let $p$ denote the projection $\OOp_{S} \to \xF_{*}$ and let
  $q\colon \bbO^\op_S\to \bbO^\op$ be the natural left fibration, then
  by construction we have $p\Theta = \CorO q$. For an object
  $\tilde T \in \bbO^{\op}_S$ lying over some $T\in \bbO^\op$ and an
  element $i \in \CorO(T)$, there is an inert map $T \to C_{n}$ in
  $\bbO^{\op}$ which lifts $\rho^i$ and it is easy to check that it is
  locally $\CorO$-coCartesian. Since $q$ is a left fibration, we can
  lift $T \to C_{n}$ to a locally $\CorO q$-coCartesian morphism.

  To simplify the notation we give the remainder of the proof in the
  case where $S$ is a singleton; the general case is proved by the
  same argument. Consider an active morphism
  $\phi \colon (A_{b})_{b \in B} \to \Theta(T)$ in $\OOp$. This
  corresponds to giving, for every vertex $v$ of $T$, a tree $S_{v}$
  whose leaves are identified with the incoming edges of $v$. We can
  view this as specifying active maps $S_{v} \to C_{v}$ in
  $\bbO^{\op}$ where $T \to C_{v}$ is the inert map corresponding to
  the vertex $v$, and then define a new tree $T'$ with an active map
  $T' \to T$ in $\bbO^{\op}$ by iterating the construction of
  Remark~\ref{rmk:glueing}. In other words, the tree $T'$ is (slightly
  informally) given by
  $T \amalg_{\coprod_{v} C_{v}} \coprod_{v} S_{v}$. It is easily
  verified using the definition of morphisms and composition in $\OOp$
  that the map $T' \to T$ is $\Theta$-Cartesian over $\phi$.
\end{proof}

Applying \cite[Theorem 2.3.3.23]{ha}, this implies:
\begin{cor}\label{cor:OOpSeq}
  If $\mathcal{V}$ is a symmetric monoidal \icat{}, then the functor
  $\Theta_{S}^{*} \colon \Alg_{\OOp_{S}}(\mathcal{V}) \to
  \Alg_{\bbOop_{S}}(\mathcal{V})$,
  induced by composition with $\Theta_{S}$, is an equivalence.
\end{cor}

\begin{defn}
  Let $\AlgOpSet(\mathcal{V}) \to \Set$ denote the Cartesian
  fibration corresponding to the functor $\Set^{\op} \to \CatI$ taking
  $S$ to $\Alg_{\OOp_{S}}(\mathcal{V})$, and let
  $\AlgOSet(\mathcal{V}) \to \Set$ denote the
  pullback of the Cartesian fibration $\AlgOS(\mathcal{V})
  \to \mathcal{S}$ along the inclusion $\Set \hookrightarrow
  \mathcal{S}$. Since the functors $\Theta_{S}$ are natural in $S$,
  they induce a functor \[\Theta^{*} \colon
  \AlgOpSet(\mathcal{V}) \to
  \AlgOSet(\mathcal{V})\] 
  over $\Set$ that preserves Cartesian morphisms.
\end{defn}

\begin{cor}\label{cor:AlgOpdeq}
  The functor \[\Theta^{*} \colon
  \AlgOpSet(\mathcal{V}) \to
  \AlgOSet(\mathcal{V})\] is an equivalence.
\end{cor}
\begin{proof}
  Since this is a functor between Cartesian fibrations that preserves
  Cartesian morphisms, it is an equivalence as it is an equivalence on
  fibres over every $S \in \Set$ by Corollary~\ref{cor:OOpSeq}.
\end{proof}

\begin{propn}\label{propn:FFESSet}
  If $\mathcal{V}$ is a presentably symmetric monoidal \icat{}, then
  the inclusion $\AlgOSet(\mathcal{V}) \hookrightarrow
  \AlgOS(\mathcal{V})$ induces an equivalence
  \[ \AlgOSet(\mathcal{V})[\name{FFES}^{-1}] \isoto
  \AlgOS(\mathcal{V})[\name{FFES}^{-1}],\]
  where $\name{FFES}$ denotes the class of fully faithful and
  essentially surjective morphisms on both sides.
\end{propn}
\begin{proof}
  As \cite[Theorem 5.3.17]{enriched}.
\end{proof}

\begin{cor}\label{cor:AlgOpdFFES}
  There is an equivalence of \icats{}
  \[ \AlgOpSet(\mathcal{V})[\name{FFES}^{-1}] \simeq \OpdI^{\mathcal{V}}.\]
\end{cor}
\begin{proof}
  Combine the equivalences of Corollary~\ref{cor:AlgOpdeq},
  Proposition~\ref{propn:FFESSet}, Theorem~\ref{thm:PCSisAlgDend},
  Theorem~\ref{thm:PSegeq}, Proposition~\ref{propn:DFiV}, and
  Theorem~\ref{theo cp obj are loc}.
\end{proof}

\subsection{Rectification}\label{subsec rect}
In this subsection we will prove that for a suitable monoidal model
category $\mathbf{V}$, the \icat{} $\OpdI^{\mathbf{V}[W^{-1}]}$ of
\iopds{} enriched in the symmetric monoidal \icat{}
$\mathbf{V}[W^{-1}]$, obtained by inverting the weak equivalences $W$
in $\mathbf{V}$, is equivalent to the \icat{} obtained from operads
strictly enriched in $\mathbf{V}$.

\begin{defn}
  Let $\mathbf{V}$ be a symmetric monoidal model category.  An operad
  $\mathbf{O}$ is called \emph{admissible} for $\mathbf{V}$ if there
  is a model structure on $\Alg_{\mathbf{O}}(\mathbf{V})$ such that
  the weak equivalences and fibrations are those maps whose underlying
  maps in $\mathbf{V}$ are weak equivalences and fibrations,
  respectively.
\end{defn}

\begin{propn}[Pavlov--Scholbach]\label{propn:AlgOpScomp}
  Suppose the operad $\sOp_{S}$ is admissible for $\mathbf{V}$ for
  a set $S$, and let $W_{S}$ denote the class of weak equivalences in
  $\Alg_{\sOp_{S}}(\mathbf{V})$. Then the natural functor of \icats{}
  \[ \Alg_{\sOp_{S}}(\mathbf{V})[W_{S}^{-1}] \to
  \Alg_{\OOp_{S}}(\mathbf{V}[W^{-1}])\]
  is an equivalence, where $\mathbf{V}[W^{-1}]$ denotes the symmetric
  monoidal \icat{} induced by $\mathbf{V}$, with $W$ being the class
  of weak equivalences in $\mathbf{V}$.
\end{propn}
\begin{proof}
  The operad $\sOp_{S}$ is $\Sigma$-cofibrant (which here just means
  that the $\Sigma_{n}$-actions are all free), so this is a special
  case of \cite[Theorem 7.10]{PavlovScholbach}. (As stated, this
  result requires a simplicial model category, but since our operad
  $\sOp_S$ is merely an operad in sets the same proof goes through
  without this assumption.)
\end{proof}

\begin{examples}\label{exs:admissible}
  Proposition~\ref{propn:AlgOpScomp} applies to the following
  model categories:
  \begin{enumerate}[(i)]
  \item the category $\Set_{\Delta}$ of simplicial sets, equipped with the Kan--Quillen model structure,
  \item the category $\name{Top}$ of compactly generated weak Hausdorff spaces, equipped with the usual model structure,
  \item the category $\name{Ch}_{k}$ of chain complexes of $k$-vector,
    spaces, where $k$ is a field of characteristic $0$ (or more
    generally a ring containing $\mathbb{Q}$), equipped with the
    projective model structure,
  \item the category $\name{Sp}^{\Sigma}$ of symmetric spectra,
    equipped with the positive stable model structure.
  \end{enumerate}
  These are the standard examples of model categories for which
  \emph{all} operads (and so in particular the operads $\sOp_{S}$ for
  all $S$) are admissible, as discussed in \cite[\S
  7]{PavlovScholbachSymm}. Further examples of such model categories
  include orthogonal spectra \cite{KroOperad}, orthogonal $G$-spectra,
  stable module categories, and the ``folk'' model structure on
  categories \cite{WhiteYau}.  Unfortunately, we are not aware of any
  examples of symmetric monoidal model categories for which the
  operads $\sOp_{S}$ are admissible other than those for which all
  operads are admissible.
\end{examples}

\begin{remark}\label{rmk:fixedobmodstr}
  There is a substantial literature on model structures for operads
  enriched in a model category (monochromatic or with a fixed set of
  objects), and more generally for algebras over a fixed operad in a
  model category.  We list some key results of this kind:
  \begin{itemize}
  \item Hinich~\cite{HinichAlg,HinichErr} constructed a model
    structure for monochromatic operads in chain complexes over a ring
    containing $\mathbb{Q}$.
  \item Berger--Moerdijk~\cite{BergerMoerdijk} constructed a model
    structure on reduced monochromatic operads (i.e.\ ones with no
    nullary operations) in suitable model categories with a
    commutative Hopf interval, including simplicial sets, topological
    spaces, and chain complexes over a ring. They later extended this
    result to algebras for coloured operads in
    \cite{BergerMoerdijkCol}, giving model structures for operads with
    any fixed set of colours as a special case.
  \item Kro~\cite{Kro} extended the work of Berger--Moerdijk to get a
    model structure for monochromatic reduced operads in orthogonal
    spectra.
  \item Elmendorf--Mandell~\cite{ElmendorfMandell} constructed a model
    structure on algebras over any simplicial operad in symmetric
    spectra (in $\name{Top}$); this was used by
    Gutiérrez--Vogt~\cite{GutierrezVogt} to obtain model structures on
    operads in symmetric spectra with a fixed set of colours.
  \item Most recently, Pavlov--Scholbach~\cite{PavlovScholbach} have
    studied general assumptions on a model category under which all
    operads are admissible, and applied this to symmetric spectra in
    general model categories in \cite{PavlovScholbachSp}.
  \end{itemize}
\end{remark}

\begin{remark}
  It follows from the results of \cite{NikolausSagave} that for any
  presentably symmetric monoidal \icat{} $\mathcal{V}$ there
  exists a symmetric monoidal simplicial combinatorial model category
  modelling $\mathcal{V}$ for which all (simplicial) operads are admissible.
\end{remark}

\begin{remark}
  Since the operads $\sOp_{S}$ are $\Sigma$-cofibrant, work of
  Spitzweck~\cite{spi} shows that for any cofibrantly generated
  symmetric monoidal model category $\mathbf{V}$, the category
  $\Alg_{\sOp_{S}}(\mathbf{V})$ has a \emph{semi}-model structure. We
  strongly suspect that this is sufficient for
  Proposition~\ref{propn:AlgOpScomp} to hold. However, for the proof
  to go through for semi-model categories one would have to extend most of
  Lurie's results relating model categories to \icats{}, such as the
  connection between homotopy colimits in model categories and
  colimits in the associated \icat{}, to the semi-model case.
\end{remark}

\begin{cor}\label{cor:OpdWeq}
  Suppose $\mathbf{V}$ is a symmetric monoidal model category for
  which the operads $\sOp_{S}$ are admissible for all sets $S$. Then
  there is an equivalence
  \[\OpdV[W^{-1}] \isoto
  \AlgOpSet(\mathbf{V}[W^{-1}])\]
  over $\Set$, where $W$ denotes the morphisms that are bijective on
  objects and given by weak equivalences on all multimorphism objects.

\end{cor}
\begin{proof}
  The forgetful functor $\OpdV \to \Set$ taking an
  operad in $\mathbf{V}$ to its set of colours is a
  Grothendieck fibration (and opfibration) with fibre
  $\Alg_{\sOp_{S}}(\mathbf{V})$. Applying \cite[Corollary 4.22]{enrcomp} or
  \cite[Proposition 2.1.4]{hinloc} it follows that $\OpdV[W^{-1}] \to \Set$
  is the Cartesian (and coCartesian) fibration corresponding to the
  functor taking $S$ to
  $\Alg_{\sOp_{S}}(\mathbf{V})[W_{S}^{-1}]$. Thus the functor
  \[\OpdV[W^{-1}] \to
  \AlgOpSet(\mathbf{V}[W^{-1}])\]
  over $\Set$ is a functor between Cartesian fibrations that
  preserves Cartesian morphisms. It is therefore an equivalence as
  it is an equivalence on fibres at each $S$ by Proposition~\ref{propn:AlgOpScomp}.
\end{proof}

\begin{remark}
  In the situation above there is a model structure on
  $\OpdV$ where the morphisms in $W$ are the weak
  equivalences by \cite[Proposition 4.25]{enrcomp} or \cite[Theorem
  3.0.12]{HarpazPrasma}.
\end{remark}

\begin{defn}
  If $\mathbf{V}$ is a symmetric monoidal model category, we say a
  morphism $F \colon \mathbf{O} \to \mathbf{O}'$ of
  $\mathbf{V}$-enriched operads is a \emph{Dwyer--Kan equivalence} if:
  \begin{enumerate}[(1)]
  \item The map $\mathbf{O}(x_{1},\ldots,x_{n}; y) \to
    \mathbf{O}'(F(x_{1}), \ldots, F(x_{n}); F(y))$ is a weak
    equivalence in $\mathbf{V}$ for all $x_{1},\ldots,x_{n},y$ in
    $\mathbf{O}$.
  \item The functor $\mathbf{V} \to h\mathbf{V}$ to the homotopy
    category is symmetric monoidal, so to a $\mathbf{V}$-enriched
    operad $\mathbf{O}$ we can associate an $h\mathbf{V}$-enriched
    operad $h\mathbf{O}$. The induced functor of
    $h\mathbf{V}$-enriched operads
    $h F \colon h \mathbf{O} \to h\mathbf{O}'$ is essentially
    surjective (i.e.\ its underlying functor of enriched categories is
    essentially surjective).
  \end{enumerate}
\end{defn}

\begin{thm}\label{thm:rect}
  Suppose $\mathbf{V}$ is a symmetric monoidal model category for
  which the operads $\sOp_{S}$ are admissible for all sets $S$. Then
  there is an equivalence 
  \[\OpdV[\name{DK}^{-1}] \isoto
  \OpdI^{\mathbf{V}[W^{-1}]},\]
  where $\name{DK}$ denotes the class of Dwyer--Kan equivalences.
\end{thm}
\begin{proof}
  The class of Dwyer--Kan equivalences clearly corresponds under the
  equivalence of Corollary~\ref{cor:OpdWeq} to the class of fully
  faithful and essentially surjective morphisms in
  $\AlgOpSet(\mathbf{V}[W^{-1}])$, so we get an equivalence
  \[ \OpdV[\name{DK}^{-1}] \isoto
  \AlgOpSet(\mathbf{V}[W^{-1}])[\name{FFES}^{-1}].\]
  The result now follows by combining this with the equivalence of
  Corollary~\ref{cor:AlgOpdFFES}.
\end{proof}

\begin{remark}
  By Example~\ref{exs:admissible}, we can apply Theorem~\ref{thm:rect}
  to get the following comparisons:
  \begin{itemize}
  \item The homotopy theory of simplicial operads is equivalent to
    that of \iopds{}; this was already shown by
    Cisinski--Moerdijk~\cite{CisinkiMoerdijk3}.
  \item The homotopy theory of spectral operads, or more precisely
    operads enriched in symmetric spectra, is equivalent to that of
    spectral \iopds{}, i.e.\ \iopds{} enriched in the \icat{} of
    spectra.
  \item If $k$ is a ring containing $\mathbb{Q}$, then the homotopy
    theory of dg-operads over $k$, i.e.\ operads enriched in chain
    complexes of $k$-modules, is equivalent to that of \iopds{}
    enriched in the derived \icat{} of $k$.
  \end{itemize}
\end{remark}

\begin{remark}\label{rmk:opdmodstr}
  In good cases, there is a model structure on
  $\OpdV$ with the Dwyer--Kan equivalences as weak
  equivalences. Such model structures were constructed by
  Cisinski--Moerdijk~\cite{CisinkiMoerdijk3} and
  Robertson~\cite{Robertson} for simplicial operads, and by
  Caviglia~\cite{Caviglia} for a general class of model categories
  that also includes topological spaces and chain complexes over a
  field of characteristic $0$. In unpublished work \cite{Caviglia2},
  Caviglia has furthermore extended this result so that it also
  applies to symmetric spectra.
\end{remark}

\bibliographystyle{amsplainurl}

\providecommand{\bysame}{\leavevmode\hbox to3em{\hrulefill}\thinspace}
\providecommand{\MR}{\relax\ifhmode\unskip\space\fi MR }
\providecommand{\MRhref}[2]{%
  \href{http://www.ams.org/mathscinet-getitem?mr=#1}{#2}
}
\providecommand{\href}[2]{#2}

\end{document}